\documentclass[11pt]{article}
\usepackage{amssymb,amsmath,amscd}
\usepackage{amsthm}
\newtheorem{theorem}{Theorem}[section]
\newtheorem{theorem*}{Theorem}
\newtheorem{proposition}{Proposition}[section]
\newtheorem{corollary}{Corollary}[section]
\newtheorem{definition}{Definition}[section]
\newtheorem{proposition*}{Proposition A\!\!}
\newtheorem{corollary*}{Corollary A\!\!}
\newtheorem{lemma}{Lemma}[section]
\newtheorem{lemma*}{Lemma A\!\!}

\begin{document}

\title {Singular conformally invariant trilinear forms, II\\
The higher multiplicity cases}

\author{Jean-Louis Clerc}

\date{July 6, 2015}

 \maketitle
 
 \begin{abstract}
Let $S$ be the sphere of dimension $n-1, n\geq 4$. Let $(\pi_{\lambda}) _{\lambda\in \mathbb C}$ be the scalar principal series of representations of the conformal group $SO_0(1,n)$, realized on $\mathcal C^\infty(S)$. For $\boldsymbol \lambda = (\lambda_1,\lambda_2,\lambda_3)\in \mathbb C^3$, let $Tri(\boldsymbol \lambda)$ be the space of continuous trilinear forms on $\mathcal C^\infty(S)\times \mathcal C^\infty(S)\times \mathcal C^\infty(S)$ which are invariant under $\pi_{\lambda_1}\otimes\pi_{\lambda_2}\otimes\pi_{\lambda_3}$.
 For each value of $\boldsymbol \lambda$, the dimension of $Tri(\boldsymbol \lambda)$ is computed and a basis of $Tri(\boldsymbol \lambda)$ is described. 
\end{abstract}

\section*{ Introduction}

This article continues the study of conformally invariant trilinear forms on the sphere $S=S^{n-1}, n\geq 4$ (see \cite{co, bc}), and more specifically achieves the work begun in \cite{c}. Notation is same as in  \cite{c} and  sections of the present paper are numbered in continuity with those of \cite{c}.

For $\boldsymbol \lambda=(\lambda_1,\lambda_2,\lambda_3)\in \mathbb C^3$  let $Tri(\boldsymbol \lambda)$ be the space of continuous trilinear forms on $\mathcal C^\infty(S)\times \mathcal C^\infty(S)\times \mathcal C^\infty(S)$ which are invariant under the action of the conformal group $G=SO_0(1,n)$ by (the tensor product of) three representations of the scalar principal series $(\pi_{\lambda_1}, \pi_{\lambda_2}, \pi_{\lambda_3})$. Such trilinear forms, when viewed as distributions on $S\times S\times S$  are said to be $\boldsymbol \lambda$-invariant. In \cite {co} was constructed (by analytic continuation) a  family $\mathcal K^{\boldsymbol \lambda}$ of $\boldsymbol \lambda$-invariant trilinear forms, depending \emph{meromorphically} on $\boldsymbol \lambda$. Later, in \cite{c}, a \emph{holomorphic} family $\widetilde {\mathcal K}^{\boldsymbol \lambda}$ was introduced by renormalizing the family $\mathcal K^{\boldsymbol \lambda}$. The zero set $Z$ of the function $\boldsymbol \lambda \longmapsto \widetilde {\mathcal K}^{\boldsymbol \lambda}$ was computed, and a \emph{generic} multiplicity 1 theorem was proved, namely \[\boldsymbol \lambda\notin Z\quad \Longleftrightarrow \quad \dim Tri(\boldsymbol \lambda) = 1\quad \Longleftrightarrow\quad Tri(\boldsymbol \lambda) = \mathbb C \widetilde {\mathcal K}^{\boldsymbol \lambda}\ .\]
The main purpose of this article is to compute $\dim Tri(\boldsymbol \lambda)$ when $\boldsymbol \lambda\in Z$, and to give a basis of $Tri(\boldsymbol \lambda)$.

The fine structure of $Z$ is studied in section 8. The set $Z$ is a denumerable union of affine lines in $\mathbb C^3$, which come in six families, each family being made of parallel lines. There is a natural partition of $Z$, namely
\[ Z=Z_1\mathaccent\cdot\cup Z_2\mathaccent\cdot\cup Z_3\]
where $Z_d$ is the subset of points of $Z$ which belong to exactly $d$ lines of $Z$  (compare \cite{o1}). Notice that $Z_2$ and $Z_3$ are denumerable sets of isolated points in $\mathbb C^3$.

The main result of the article is the following.
\begin{theorem*} Let $\boldsymbol \lambda\in Z$.
\[\text{for }\boldsymbol \lambda \in Z_1\qquad  \dim Tri(\boldsymbol \lambda) =2 
\]
\[\text{for }\boldsymbol \lambda \in Z_2\qquad  \dim Tri(\boldsymbol \lambda) =2 
\]
\[\text{\ \ for }\boldsymbol \lambda \in Z_3\qquad  \dim Tri(\boldsymbol \lambda) =3\ .
\]
\end{theorem*}
These results may be heuristically anticipated as follows. Let $\boldsymbol \lambda_0\in Z$. Let $\boldsymbol \lambda(s)$ be a holomorphic curve in $\mathbb C^3$ such that $\boldsymbol \lambda(0) = \boldsymbol \lambda_0$. Then $s\mapsto \widetilde {\mathcal K}^{\boldsymbol \lambda(s)}$ is a distribution-valued holomorphic function. The first non vanishing term of its Taylor expansion at $s=0$ clearly is a $\boldsymbol \lambda_0$-invariant distribution. Variations on this idea lead to the following considerations.

When $\boldsymbol \lambda_0$ is in $Z_1$, hence belongs to a unique affine line ${d_1}\subset Z$, then  $ \nabla_1 \widetilde {\mathcal K}^{\boldsymbol \lambda}(\boldsymbol \lambda_0)=0$, where $\nabla_1$ stands for the derivative in the direction of $d_1$. Let $\overrightarrow{d_2}$ and $\overrightarrow{d_3}$ be two complementary directions  to $\overrightarrow{d_1}$ in $\mathbb C^3$. Then $\nabla_2\widetilde {\mathcal K}^{\boldsymbol \lambda}(\boldsymbol \lambda_0)$ and $ \nabla_3 \widetilde {\mathcal K}^{\boldsymbol \lambda}(\boldsymbol \lambda_0)$ are ``natural'' candidates to generate $Tri(\boldsymbol \lambda_0)$.

When $\boldsymbol \lambda_0$ is in $Z_2$, hence belongs to two distinct lines $d_1$ and $d_2$, let $\overrightarrow{d_3}$ be a complementary direction to $\overrightarrow{d_1}$ and $\overrightarrow {d_2}$. Then $\nabla_1 \widetilde {\mathcal K}^{\boldsymbol \lambda}(\boldsymbol \lambda_0)=\nabla_2 \widetilde {\mathcal K}^{\boldsymbol \lambda}(\boldsymbol \lambda_0)=0$. In this case, the second mixed derivative $\nabla_1\nabla_2\widetilde {\mathcal K}^{\boldsymbol \lambda}(\boldsymbol \lambda_0)$ is easily seen to be a $\boldsymbol \lambda_0$-invariant distribution. Then $\nabla_3 \widetilde {\mathcal K}^{\boldsymbol \lambda}(\boldsymbol \lambda_0)$ and $\nabla_1\nabla_2\widetilde {\mathcal K}^{\boldsymbol \lambda}(\boldsymbol \lambda_0)$ are ``natural'' candidates to generate $Tri(\boldsymbol \lambda_0)$.

When $\boldsymbol \lambda_0$ is in $Z_3$, hence belong to three lines $d_1,d_2,d_3$ (it turns out that $\overrightarrow{d_1}, \overrightarrow{d_2},\overrightarrow{d_3}$ form a basis of $\mathbb C^3$), then all partial derivatives of  $\widetilde {\mathcal K}^{\boldsymbol \lambda}$ vanish at $\boldsymbol \lambda_0$ and  any second order partial derivative of  $\widetilde {\mathcal K}^{\boldsymbol \lambda}(\boldsymbol \lambda)$ at $\boldsymbol \lambda_0$ is $\boldsymbol \lambda_0$-invariant. Then $\nabla_j\nabla_j\widetilde {\mathcal K}^{\boldsymbol \lambda}(\boldsymbol \lambda_0)=0$ for $j=1,2,3$, so that the three second mixed derivatives $\nabla_i\nabla_j \widetilde {\mathcal K}^{\boldsymbol \lambda}(\boldsymbol \lambda_0)=0,\  (1\leq i<j\leq 3)$ are ``natural'' candidates to generate $Tri(\boldsymbol \lambda_0)$.

When going to the actual construction of these invariant distributions and proof that they indeed generate $Tri(\boldsymbol \lambda)$, a more refined (and more technical) partition of $Z$ is needed, essentially related to the type of the poles (see details in section 8). 

An important ingredient of this article is the connection with the results obtained in \cite{bc}, were residues (in a generic sense) of the meromorphic function $\boldsymbol \lambda \mapsto {\mathcal K}^{\boldsymbol \lambda}$ were computed. The normalization (passage from ${\mathcal K}^{\boldsymbol \lambda}$ to $\widetilde{\mathcal K}^{\boldsymbol \lambda}$) involves four $\Gamma$ factors. At a generic pole $\boldsymbol \lambda_0$ (i.e. belonging to a unique plane of poles) exactly one $\Gamma$ factor becomes singular at $\boldsymbol \lambda_0$, and $\widetilde {\mathcal K}^{\boldsymbol \lambda_0}$ is easily seen to be a multiple of the residue of $\mathcal K^{\lambda}$ at $\boldsymbol \lambda_0$. For more singular poles, some variation of this relation does make sense. For each plane of poles, the residues at generic poles form a meromorphic family, which can be renormalized to get a holomorphic family. If $\boldsymbol \lambda_0$ is a non generic pole (i.e. beings to several planes of poles, so that several $\Gamma$ factors of the renormalization vanish at $\boldsymbol \lambda_0$), then the values at $\boldsymbol \lambda_0$ of the renormalized families of residues corresponding to the planes of poles containing $\boldsymbol \lambda_0$ are equal (up to constants) to partial derivatives of $\widetilde {\mathcal K}^{\boldsymbol \lambda}$ at $\boldsymbol \lambda_0$.

To each plane of poles, a holomorphic family of invariant distributions is constructed, which generically coincides (up to a scalar) with the residue.
Recall that planes of poles come in two types called  type I  and type II (see \cite{co, c}).  

For a plane of poles of type I, the family (generically denoted by $\widetilde{\mathcal T}$) constructed in section 9 is obtained by renormalizing the meromorphic family constructed in \cite{bc}. The distributions are supported on a submanifold of codimension $n-1$ and  their expression involve \emph{covariant differential operators}. 

For a plane of poles of type II, the family (generically denoted by $\mathcal S$) is constructed in section 11. Although based on computations done in \cite{bc}, the construction as done in this article use new arguments
and in particular a more satisfactory determination of the (partial) Bernstein-Sato polynomial for the kernel $k_{\boldsymbol \alpha}$ (Proposition 11.3), which has its own interest. The distributions are supported on the diagonal of $S\times S\times S$ and their expressions involve \emph{covariant bi-differential operators}.

A last family (generically denoted by $\mathcal R$) is constructed in section 12 and has a different nature. The planes in $\mathbb C^3$ to which they are attached are not planes of poles. I was led to introduce this family by the heuristic consideration above about derivatives of $\widetilde {\mathcal K}^{\boldsymbol \lambda}$, when dealing with poles of type II which are in $Z_1$. The corresponding distributions  generically have full support in $S\times S\times S$.

These holomorphic families will be almost sufficient for describing bases of $Tri(\boldsymbol \lambda)$. However, when  $\boldsymbol \lambda$ belongs to $Z_2$, some exceptional distributions, denoted by $\mathcal Q_{l_1,m_2,m_3}$ are needed. They are obtained as second order derivatives of the function $\widetilde {\mathcal K}^{\boldsymbol \lambda}$.

Let us observe that all the distributions involved in the description of the invariant trilinear forms are in the \emph{closure} of the meromorphic family $\mathcal K^{\boldsymbol \lambda}$ in the sense of Oshima (see \cite{osh} section 6).

Having constructed candidates for a basis of $Tri(\boldsymbol \lambda)$, a proof of their linear independence  is required.   This is usually obtained by using their $K$-coefficients\footnote{Recall that $K\simeq SO(n)$ is a maximal compact subgroup of $G$}. In \cite{c} was introduced the family $\big(p_{a_1,a_2,a_3}\big)_{a_1,a_2,a_3\in \mathbb N}$ of polynomial functions on $S\times S\times S$ which have the property that a $K$-invariant distribution $T$ is equal to $0$ if and only if all its $K$-coefficients $T(p_{a_1,a_2,a_3})$ are equal to $0$. It turns out that it is possible to \emph{explicitly} compute the $K$-coefficients of all the distributions we are concerned with, which allows to answer the questions of independence. The main ingredient is the evaluation of an integral (Proposition 3.2 in \cite{c}), related to the computation of so-called \emph{Bernstein-Reznikov integrals} (see \cite{ckop} for more examples), and which  is recalled  for convenience in appendix A1.

On the other hand, the verification that the candidates indeed generate $Tri(\boldsymbol \lambda)$ is more difficult and involve some delicate analysis dealing with singular distributions. This sort of problem was already encountered in establishing the multiplicity 1 theorem for $\boldsymbol \lambda\notin  Z$. We refer to the presentation of the main ideas in section 5 and 6 of \cite{c}. The proofs are based on two techniques. The first one consists in non-extension results, obtained in \cite{c} and presented (slightly reformulated) in appendix A2. The second technique concerns the determination of the dimension of the space $Tri(\boldsymbol \lambda, diag)$ of $\boldsymbol \lambda$-invariant distributions which are  supported on the diagonal of $S\times S\times S$. The usual technique of Bruhat (see \cite{br}) is not powerful enough to give a sharp bound for this dimension, due to the fact that the stabilizer of a generic point in the diagonal is a parabolic subgroup of $G$, hence \emph{not reductive} in $G$. So it is necessary to make a specific study. Those singular distributions are in 1-1 correspondence with \emph{covariant bi-differential operators}. In turn those have been studied in \cite{or} and their determination is reduced to solving  an explicit system of linear homogeneous equations. A systematic study of this system is presented in appendix A3. However a better understanding (i.e. based on  geometry and/or harmonic analysis)  of these results is desirable. As it is well-known, the determination of the covariant bi-differential operators is linked to problems about tensor product of two generalized Verma modules (see \cite{s} for some insight in this question).

Let us finally mention two connected problems which are not treated in this (already long) article. First, when one (or several) of the  representations $\pi_{\lambda_j}, 1\leq j\leq 3$ are \emph{reducible}, the question arises wether a (or which) $\boldsymbol \lambda$-invariant trilinear distribution(s) 
do generate (by restriction and/or passage to quotients) non trivial invariant trilinear forms on the subspaces/quotients. Only one such situation is treated in this article, namely a case where it leads to trilinear forms on three finite dimensional representations (see 15.3). Secondly let mention that it is possible to prove factorization results for covariant bi-differential operators, similar  to the results obtained by A. Juhl for conformally covariant differential operators from $S^{d+1}$ to $S^d$ (see \cite{j,ks}).

\bigskip

\centerline {\bf Summary}
\bigskip

\hskip 0.4cm {\bf 8. Refined description of $Z$}
\medskip

\hskip 0.4cm {\bf 9. The holomorphic families $\widetilde {\mathcal T}^{(j,k)}_{\, .\, , \,.\,}$}
\medskip

\hskip 0.4cm {\bf 10. The multiplicity 2 theorem for $\boldsymbol \alpha \in Z_{1,I}$}
\medskip

\hskip 1.4cm {\bf 10.1} The general case
\medskip

\hskip 1.4cm {\bf 10.2} The special case
\medskip

\hskip 0.4cm {\bf 11. The holomorphic families $\mathcal S^{(k)}_{\lambda_1, \lambda_2}$}
\medskip

\hskip 1.4cm {\bf 11.1} The family $\widetilde E_{\lambda_1,\lambda_2}$ 
\medskip

\hskip 1.4cm {\bf 11.2} A Bernstein-Sato identity for the kernel $k_{\boldsymbol \alpha}$
\medskip

\hskip 1.4cm {\bf 11.3} The families $\mathcal S^{(k)}_{\lambda_1, \lambda_2}$

\hskip 0.4cm {\bf 12. The holomorphic families $\mathcal R^{(j,m)}_{\,.\,, \, .}$}
\medskip

\hskip 0.4cm {\bf 13. The multiplicity 2 theorem for $\boldsymbol \lambda \in Z_{1,II}$}
\medskip

\hskip 1.4cm {\bf 13.1}  The general case
\medskip

\hskip 1.4cm {\bf 13.2} The special case
\medskip

\hskip 0.4cm {\bf 14. The multiplicity 2 theorem for $\boldsymbol \alpha \in Z_{2,I}$} 
\medskip

\hskip 0.4cm {\bf 15. The multiplicity 2 theorem for $\boldsymbol \alpha \in Z_{2,II}$}
\medskip

\hskip 1.4cm {\bf 15.1}  The odd case
\medskip

\hskip 1.4cm {\bf 15.2}  The  even case 
\medskip

\hskip 1.4cm {\bf 15.3}  Invariant trilinear form for three f.d. representations
\medskip

\hskip 0.4cm {\bf 16. The multiplicity 3 theorem for $\boldsymbol \alpha \in Z_3$}
\medskip

\hskip 1.4cm {\bf 16.1}  The  odd case
\medskip

\hskip 1.4cm {\bf 16.1}  The  even case
\medskip

\hskip 0.4cm {\bf Appendix}
\medskip

\hskip 1.4cm {\bf A1} Evaluation of the $K$-coefficients of $\widetilde {\mathcal K}_{\boldsymbol \alpha}$
\medskip

\hskip 1.4cm {\bf A2} Non extension results
\medskip

\hskip 1.4cm {\bf A3} Discussion of the system ${\it S}(\lambda_1,\lambda_2;k)$

\setcounter{section}{7}

\section{ Refined description of $Z$ }

For $\epsilon \in \{ +,-\}, j\in \{1,2,3\}$, $l\in \mathbb N$ and $m\in \mathbb Z$, define 
\begin{equation}
\mathcal D^{\epsilon\,,\, j}_{l,m} = \big\{ \boldsymbol \lambda \in \mathbb C^3,\quad \lambda_j= -\rho-l,\quad \lambda_{j+1}+\epsilon\, \lambda_{j+2} = m\big\}
\end{equation}
with the convention that $j+1$ and $j+2$ are interpreted modulo 3.
For each quadruple $(\epsilon,\, j,\, l,\,m)$, $\mathcal D^{\epsilon,\, j}_{l,\,m}$ is a complex affine line in $\mathbb C^3$. A quadruple $(\epsilon, j, l, m)$ is said to be \emph{admissible} if
\[m\equiv l \mod (2),\quad \vert m\vert \leq l\ .
\]
Theorem 4.2 of \cite{c} can be rewritten as
\begin{equation}
 Z = \bigcup_{\begin{matrix} (\epsilon,\, j,\, l,\, m)\\ admissible\end{matrix}} \mathcal D^{\epsilon,\, j}_{l,\,m}\quad .
\end{equation}

The description of these lines can also be made in terms of the geometric\footnote{for the correspondance between the \emph{spectral parameter} $\boldsymbol \lambda$ and the \emph{geometric parameter} $\boldsymbol \alpha$, see \cite{c} 2.1} parameter $\boldsymbol \alpha$. The line $D^{-1, j}_{l,m}$ admits the equations

\begin{equation}
\alpha_{j+1} = -(n-1)-2k_{j+1},\quad \alpha_{j+2} = -(n-1)-2k_{j+2}
\end{equation}
where $k_{j+1} = \frac{l+m}{2}, k_{j+2} = \frac{l-m}{2}\in \mathbb N$. The line $D^{+1, j}_{l,m}$ admits the equations
\begin{equation}\alpha_1+\alpha_2+\alpha_3 = -2(n-1)-2k, \quad \alpha_j = 2 p
\end{equation}
where $k = \frac{l-m}{2}, p=\frac{l+m}{2}\in \mathbb N$. This justifies the following terminology.

\begin{definition} A line $\mathcal D$ contained in $Z$ will be called 
\smallskip

$i)$ \emph{of type I} (or more precisely \emph{ of type I$_j$}) if $\mathcal D = \mathcal D^{-,j}_{l,m}$ for some admissible quadruplet $(-,j,l,m)$

$ii)$  \emph{ of type II} (or more precisely \emph{ of type II$_j$}) if $\mathcal D = \mathcal D^{+,j}_{l,m}$ for some admissible quadruplet $(+,j,l,m)$.
\end{definition}

\begin{proposition}\label{intersect2} \ 

\smallskip

$i)$ if a  line of type I$_j$ and a line of type II$_{j'}$ intersect, then $j=j'$
\smallskip

$ii)$ if a line of type I$_j$ and a (distinct) line of type I$_{j'}$ intersect, then $j\neq j'$
\smallskip

$iii)$ if a line of type II$_j$ and a (distinct) line of type II$_{j'}$ intersect, then $j\neq j'$.
\end{proposition}

\begin{proof} For $i)$, let $\mathcal D =  \mathcal D^{+,\, 1}_{l_1,m_1}$ and $\mathcal D' = D^{-,\, 2}_{l_2,m_2}$. Then we need to show that $ \mathcal D^{+,\, 1}_{l_1,m_1}\cap \mathcal D^{-,\, 2}_{l_2,m_2} = \emptyset$. The intersection is given by the set of equations
\[ \lambda_1 = -\rho-l_1, \quad\lambda_2+\lambda_3 = m_1,\quad \lambda_2 = -\rho-l_2,\quad \lambda_3-\lambda_1 = m_2\ .\]
The equations imply
\[-\lambda_1+\lambda_2+\lambda_3 = \rho+l_1+m_1 = -\rho-l_2+m_2 \ .
\]
Hence a necessary condition for a nonvoid intersection is that
\[ l_1+m_1 = -(n-1) -l_2+m_2\ .\]
As $\vert m_1\vert \leq l_1$ and $\vert m_2\vert \leq  l_2$, the LHS is nonnegative, whereas the RHS is $\leq -(n-1)$. Hence $i)$ is proved.

For $ii)$ observe that two \emph{distinct} lines $\mathcal D^{-,j}_{l,m}$ and $\mathcal D^{-,j}_{l',m'}$ necessarily belong to two \emph{distinct} parallel planes, hence cannot intersect. A similar observation for $\mathcal D^{+,j}_{l,m}$ and $\mathcal D^{+,j}_{l',m'}$, thus proving $iii)$.
\end{proof}
As a consequence, to determine all possible cases where two distinct lines in $ Z$ intersect, it suffices to examine the cases where $\epsilon' = -\epsilon$ and $j=j'$, and the cases where $\epsilon = \epsilon'$ and $j\neq j'$.

\begin{lemma} Let $(+,1,l,m)$ and $(-,1,l',m')$ be admissible quadruples.
\smallskip 

$i)$ if $l\neq l'$, then \[ \mathcal D^{+,\, 1}_{l,\,m} \cap \mathcal D^{-,\, 1}_{l',\,m'} = \emptyset\ .\]

$ii)$ if $l=l'$, then
\begin{equation}
\mathcal D^{+,\, 1}_{l,\,m} \cap \mathcal D^{-,\, 1}_{l,\,m'}= \Big\{(-\rho-l, \frac{m+m'}{2}, \frac{ m-m'}{2})\Big\}\ .
\end{equation}
\end{lemma}

\begin{proof} An element $\boldsymbol \lambda$ in $ \mathcal D^{+,\, 1}_{l,\,m} \cap \mathcal D^{-,\, 1}_{l',\,m'}$ has to satisfy the equations
\[ \lambda_1 =-\rho-l,\quad \lambda_2+\lambda_3 = m,\quad \lambda_1 = -\rho-l',\quad \lambda_2-\lambda_3 = m'
\]
The intersection is clearly empty if $l\neq l'$. lf $l=l'$, then the two lines intersect and the intersection point is equal to $(-\rho-l, \frac{m+m'}{2}, \frac{ m-m'}{2})$. \end{proof}

\begin{lemma} Let $(+,1, l_1,m_1)$ and $(+,2,l_2,m_2)$ be admissible quadruples. 
\smallskip

$i)$ if $l_1-m_1 \neq l_2-m_2$, then
 \[\mathcal D^{+,\, 1}_{l_1,\,m_1}\cap \mathcal D^{+,\, 2}_{l_2,\,m_2} = \emptyset\  .\] 

$ii)$ if $l_1-m_1 = l_2-m_2$, then
 \begin{equation}
\mathcal D^{+,\, 1}_{l_1,\,m_1}\cap \mathcal D^{+,\, 2}_{l_2,\,m_2} = \Big\{(-\rho-l_1, -\rho-l_2, \rho+m_1+l_2)\Big\}\ .
\end{equation}
\end{lemma}

\begin{proof} An element $\boldsymbol \lambda$ belongs to the intersection $\mathcal D^{+,\, 1}_{l_1,\,m_1}\cap \mathcal D^{+,\, 2}_{l_2,\,m_2}$ if and only if the following equations are satisfied
\[ \lambda_1 = -\rho-l_1,\quad \lambda_2+\lambda_3 = m_1,\quad \lambda_2=-\rho-l_2,\quad \lambda_1+\lambda_3 = m_2\ .\]
A combination of these equations yields
\[\lambda_1+\lambda_2+\lambda_3 = -\rho-l_1+m_1 = -\rho-l_2+m_2\ .
\]
Hence a necessary condition for a non empty intersection is \[l_1-m_1 = l_2-m_2\]
Assume that this condition is satisfied, then the intersection point of the two lines is equal to $(-\rho-l_1, -\rho-l_2, \rho+m_1+l_2)$. For later reference, observe that the last coordinate can be written as $\rho+m_3$, with $m_3 = m_2+l_1 = m_1+l_2$. Then
 \[m_3\equiv l_1\pm l_2\mod 2, \quad \vert l_1-l_2\vert\leq m_3 \leq l_1+l_2\ .\]
\end{proof}
\begin{lemma}\label{inter--}
 Let $(-,1,l_1,m_1)$ and $(-,2,l_2,m_2)$ be admissible quadruples.
\smallskip

$i)$ if $l_1-l_2\neq m_1+m_2$, then $\mathcal D^{-,1}_{l_1,m_1} \cap \mathcal D^{-,2}_{l_2,m_2} = \emptyset$.
\smallskip

$ii)$ if $l_1-l_2 = m_1+m_2$, then 
\[\mathcal D^{-,1}_{l_1,m_1} \cap \mathcal D^{-,2}_{l_2,m_2}  = (-\rho-l_1, -\rho-l_2, -\rho-l_1+m_2)\ .
\]
Moreover, set $l_3 =l_1-m_2, m_3 = -l_1+l_2$. Then $ m_3\equiv l_3 \mod 2, \vert m_3\vert \leq l_3$, so that $(-,3,l_3,m_3)$ is an admissible quadruple. The three lines 
$\mathcal D^{-,1}_{l_1,m_1}, \mathcal D^{-,2}_{l_2,m_2}$ and  $\mathcal D^{-,3}_{l_3,m_3}$ satisfy
\[\mathcal D^{-,1}_{l_1,m_1}\cap \mathcal D^{-,2}_{l_2,m_2}\cap\mathcal D^{-,3}_{l_3,m_3} = \{ \big(-\rho-l_1,-\rho-l_2,-\rho-l_3\big)\}\ .
\]
\end{lemma}

\begin{proof} An element $\boldsymbol \lambda$ belongs to the intersection $\mathcal D^{-,1}_{l_1,m_1} \cap \mathcal D^{-,2}_{l_2,m_2}$ if and only if the following equations are satisfied :
\[\lambda_1=-\rho-l_1,\quad \lambda_2-\lambda_3 = m_1,\quad \lambda_2 = -\rho-l_2,\quad \lambda_3-\lambda_1 = m_2\ .
\]
These equations imply $\lambda_2-\lambda_1=(\lambda_2-\lambda_3)+(\lambda_3-\lambda_1) = l_1-l_2=m_1+m_2$, and hence the intersection is empty unless $l_1-l_2 = m_1+m_2$. Assume that this condition is satisfied. Let 
\[l_3 = l_1-m_2=l_2+m_2,\qquad m_3 = -l_1+l_2\ .\] 
Then $l_3+m_3=l_2-m_2\equiv 0 \mod 2$ and hence $m_3\equiv l_3 \mod 2$. Moreover,
 \[m_3 = -l_1+l_2\leq m_1+l_2 = l_3, -m_3 = l_1-l_2\leq l_1-m_2=l_3\ .
 \]
 so that $\vert m_3\vert \leq l_3$, showing that $(-,3,l_3,m_3)$ is an admissible quadruplet. Now 
 \[\mathcal D^{-,1}_{l_1,m_1}\cap \mathcal D^{-,2}_{l_2,m_2}= (-\rho-l_1,-\rho-l_2, -\rho-l_1+m_2)\ .\]
  Observe that the third coordinate can be written as $-\rho-l_3$, whereas $-\rho-l_1-(\-\rho-l_2) = -l_1+l_2 = m_3$, so that the intersection point also belongs to $\mathcal D^{-,3}_{l_3,m_3}$.
 \end{proof}
 For later reference, observe that in the last situation,
 \begin{equation}\label{inter3}
 l_1+l_2+l_3\equiv 0 \mod 2, \qquad \vert l_1-l_2\vert \leq l_3\leq l_1+l_2\ ,
\end{equation}
and observe that these conditions are invariant under permutation of $1,2,3$.
Conversely, let $l_1,l_2,l_3$ satisfy \eqref{inter3}. Then the three lines $\mathcal D^{-,1}_{l_1,-l_2+l_3}$  $\mathcal D^{-,2}_{l_2,-l_3+l_1}$ and $\mathcal D^{-,3}_{l_3,-l_1+l_2}$ have a common point, namely $(-\rho-l_1,-\rho-l_2,-\rho-l_3)$.

\begin{proposition} \ 
\smallskip

$i)$ The intersection  of three distinct lines  $\mathcal D, \mathcal D',\mathcal D''$ contained in $Z$ is non empty, if and only if, up to a permutation of $\{1,2,3\}$ there exist $l_1,l_2,l_3\in \mathbb N$ satisfying \eqref{inter3} such that

\[\mathcal D = \mathcal D^{-,1}_{l_1,-l_2+l_3},\qquad \mathcal D' = \mathcal D^{-,2}_{l_2,-l_3+l_1},\qquad \mathcal D'' =  \mathcal D^{-,3}_{l_3,-l_1+l_2}\ .
\]
If this is the case, then
\[ \mathcal D\cap \mathcal D'\cap \mathcal D'' = \{ (-\rho-l_1, -\rho-l_2,-\rho-l_3)\}\ .\]

$ii)$ The intersection of four distinct lines contained in $Z$ is always empty.

\end{proposition}

\begin{proof} Suppose first the lines are of type I$_j$,I$_{j'}$ and II$_k$. If they have a non empty intersection, by Proposition \ref{intersect2}, $k=j$ and $k=j'$, but $j\neq j'$, a contradiction.
A similar argument shows that  three distinct lines of type I$_j$,II$_k$,II$_{k'}$ have an empty intersection. If the three lines are of type II$_j$,II$_{j'}$,II$_{j''}$, by Proposition \ref{intersect2} $j,j',j''$ have to be mutually distinct, hence up to a permutation, we may assume that $j=1,j'=2$ and $j'' = 3$. The equations of the lines are
\[\lambda_1 = -\rho-l_1,\quad  \lambda_2+\lambda_3 = m_1, 
\]
\[\lambda_2 = -\rho-l_2,\quad \lambda_3+\lambda_1 = m_2
\]
\[\lambda_3 = -\rho-l_3,\quad  \lambda_1+\lambda_2 = m_3
\]
so that the coordinates of the intersection have to satisfy
\[2(\lambda_1+\lambda_2+\lambda_3)= -3(n-1)-2(l_1+l_2+l_3)= m_1+m_2+m_3
\]
which contradicts the conditions $\vert m_j\vert \leq l_j$ for $j=1,2,3$.

Hence it remains to examine the case where the three lines are of type I$_j$, I$_{j'}$ and I$_{j''}$. Again by Proposition \ref{intersect2}, $j,j',j''$ have to be mutually distinct, hence, up to a permutation of $\{1,2,3\}$, we may assume $j=1,j'=2,j''=3$. But then it is just the situation observed in Lemma \ref{inter--}. Hence $i)$ follows.

Now $ii)$ is a consequence of $i)$ and Proposition \ref{intersect2}.
\end{proof}

\begin{definition} For $d=1,2,3$ let
\[Z_d = \{ \boldsymbol \lambda\in \mathbb C^3, \boldsymbol \lambda \text{ belongs to  \emph{exactly }} d \text { lines contained in }Z\}\ .
\]
For $d=1$, let
\[ Z_{1,I} = \{ \boldsymbol \lambda\in Z_1, \boldsymbol \lambda \text{ belongs to a line of type I}\}
\]
\[ Z_{1,II} = \{ \boldsymbol \lambda\in Z_1, \boldsymbol \lambda \text{ belongs to a line of type II}\}
\]
For $d=2$, let
\[ Z_{2,I} = \{ \boldsymbol \lambda \in Z_2, \boldsymbol \lambda \text { belongs to a line of type I and to a line of type II }\}
\]
\[ Z_{2,II} = \{ \boldsymbol \lambda \in Z_2, \boldsymbol \lambda \text { belongs to two lines of type II }\}
\]
\end{definition}
To these definitions correspond the following partitions 

\[ Z= Z_1\mathaccent\cdot\cup Z_2\mathaccent\cdot\cup Z_3,\qquad Z_1=Z_{1,I}\mathaccent\cdot\cup Z_{1,II},\qquad Z_2 = Z_{2,I}\mathaccent\cdot\cup Z_{2,II}\ .
\] 
The next propositions describe (up to a permutation of the indices $1,2,3$) the different subsets of $Z$ just introduced, in terms of the spectral  parameter and in terms of the geometric parameter. They are easily obtained form the previous study of the intersections of the lines contained in $Z$.

\begin{proposition} Let $\mathcal D= \mathcal D^{-,1}_{l,m}$ be a line of type I. Let $\boldsymbol \lambda \in \mathcal D$. Then $\boldsymbol \lambda$  belongs to $Z_1$ if and only if
\[\lambda_2 \notin \Big\{ \frac{l-m}{2},  \frac{l-m}{2}+1,\dots, \frac{l+m}{2}\Big\}\cup \Big( -\rho-\frac{l-m}{2}-\mathbb N\Big)\ .\]

\end{proposition}

\begin{proposition}\label{Z1I}
 Let $\mathcal D$ be the line of type I defined by the equations
\[\alpha_2 =-(n-1)-2k_2,\quad  \alpha_3 = -(n-1)-2k_3,
\]
where $k_2,k_3\in \mathbb N$. Let $\boldsymbol \alpha\in \mathcal D$. Then
\[
\boldsymbol \alpha \in Z_1 \Longleftrightarrow \alpha_1 \notin \{0,2,\dots, 
2(k_2+k_3)\} \cup \Big(-(n-1)-2\mathbb N\Big)\ .
\]
\end{proposition}

\begin{proposition}\label{Z1IIlambda}
 Let $\boldsymbol \lambda$ belong to the line $D^{+,1}_{l,m}$.
Then $\boldsymbol \lambda\in Z_1$ if and only if

\[ \lambda_2 \notin \Big\{ \frac{m}{2}-\frac{l}{2}, \frac{m}{2}-\frac{l}{2}+1,\dots, \frac{m}{2}+\frac{l}{2}\Big\}\cup \Big( -\rho+\frac{m}{2}-\frac{l}{2}-\mathbb N\Big)\cup \Big(\rho+\frac{m}{2}+\frac{l}{2} +\mathbb N\Big)
\]
\end{proposition}

\begin{proposition} Let $\boldsymbol \alpha$ satisfy
\[\alpha_1+\alpha_2+\alpha_3 =-2(n-1)-2k,\quad \alpha_1 = 2p\ ,
\]
where $k,p\in \mathbb N$.
Then $\boldsymbol \alpha \in Z_1$ if and only if
\[\alpha_2 \notin \Big\{ -(n-1),-(n-1)-2,\dots, -(n-1)-2k-2p\Big\}\,\cup\, 2\mathbb N \,\cup\,\big(-2(n-1)-2k-2p-2\mathbb N\big)
\]

\end{proposition}

\begin{proposition} Let $\boldsymbol \lambda\in \mathbb C^3$.
\smallskip

$i)$ $\boldsymbol \lambda$ belongs to $Z_{2,I}$ if and only if , up to a cyclic permutation of $1,2,3$, 

\begin{equation}\label{Z2I}
\boldsymbol \lambda = (-\rho-l_1,m_2,m_3),\qquad\begin{matrix}  m_2\pm m_3 \equiv l_1 \quad \mod  (2),\\  \\\vert m_2\pm m_3\vert  \leq l_1\end{matrix}
\end{equation}
\smallskip

$ii)$ $\boldsymbol \lambda$ belongs to $Z_{2,II}$ if and only if, up to a permutation of $1,2,3$
 
\begin{equation}\label{Z2II}
 \boldsymbol \lambda = (-\rho-l_1, -\rho-l_2, \rho+ m_3), \qquad\begin{matrix} m_3\equiv l_1\pm l_2\quad \mod  (2)\\  \\ \quad\vert l_1-l_2\vert\leq m_3 \leq l_1+l_2\end{matrix} \ .
 \end{equation}
\smallskip

 $iii)$ $\boldsymbol \lambda$ belongs to $Z_3$ if and only 
\begin{equation}\label{Z3}
\boldsymbol \lambda = (-\rho-l_1, \-\rho-l_2, -\rho-l_3), \quad 
\begin{matrix}  l_1+l_2 +l_3\equiv 0\mod (2)\\ \\ \quad \vert l_1-l_2\vert \leq l_3 \leq l_1+l_2 \end{matrix} \ .
\end{equation} 
\end{proposition}
\noindent
{\bf Remark.} The conditions \eqref{Z3} are symmetric in the three variables $(l_1,l_2,l_3)$.

\begin{proposition} Let $\boldsymbol \alpha \in \mathbb C^3$.
Then 
\smallskip

$i)$  $\boldsymbol \alpha$ belongs to $Z_{2, I}$ if and only if, up to a cyclic permutation of $\{1,2,3\}$, there exists $k_1,k_2,k_3 \in \mathbb N$ such that 
\begin{equation}
\boldsymbol \alpha = \big(2k_1, -(n-1)-2k_2, -(n-1) -2k_3\big),\quad k_1\leq k_2+k_3 
\end{equation}

$ii)$ $\boldsymbol \alpha$ belongs to $Z_{2, II}$ if and only if, up to a cyclic permutation of $\{1,2,3\}$, there exists $k_1,k_2,k_3 \in \mathbb N$ such that
\begin{equation}
\boldsymbol \alpha = \big(2k_1, 2k_2, -2(n-1) -2k_3\big),\qquad k_1+ k_2\leq k_3 
\end{equation}

$iii)$ $\boldsymbol \alpha$ belongs to $Z_3$ if and only if 
\begin{equation}
\boldsymbol \alpha = \big(-(n-1)-2k_1,-(n-1)-2k_2,-(n-1)-2k_3\big)
\end{equation}
for some $k_1,k_2,k_3\in \mathbb N$.

\end{proposition}

The set $Z$ is contained in the set of poles of the meromorphic function $\mathcal K_{\boldsymbol \alpha}$. The precise nature of the pole requires some further refinement in the description of $Z$. 

\begin{proposition}\label{poleZ1I}
 Let $\boldsymbol \alpha$ satisfy 
\[
 \alpha_2 =-(n-1)-2k_2,\qquad \alpha_3 = -(n-1)-2k_3
\]
and assume that $\boldsymbol \alpha\in Z_{1,I}$.
Then $\boldsymbol \alpha$ belongs  to exactly two planes of poles (namely $\beta_2=-(n-1)-2k_2$ and $\beta_3 =-(n-1)-2k_3$), except if $\alpha_1=-2p$ for some $p\in \{ 1,2,\dots, \}$, in which case $\boldsymbol \alpha$ belongs  moreover to the plane of poles  \[\beta_1+\beta_2+\beta_3 = -2(n-1)-2(k_2+k_3+p)\ .\]\end{proposition}

\begin{proposition} Let $\boldsymbol \alpha$ satisfy
\[\alpha_1+\alpha_2+\alpha_3 =-2(n-1)-2k,\quad  \alpha_1 = 2p
\]
where $k,p\in \mathbb N$
and assume that $\boldsymbol \alpha$ belongs to $Z_{1,II}$. Then $\boldsymbol \alpha$ belongs to a exactly one plane of poles (namely $\beta_1+\beta_2+\beta_3=-2(n-1)-2k$), except if $\boldsymbol \alpha$ is (up to a permutation of the indices $2$ and $3$)
 of the form
 \[\alpha_1 = 2p,\quad \alpha_2 = -(n-1)-2k_2,\quad \alpha_3 = -(n-1)+2q,
 \]
 where $p,k_2,q\in \mathbb N, q\geq 1, k_2-p-q\geq 0$, in which case $\boldsymbol \alpha$ belongs moreover to the plane of poles $ \beta_2 = -(n-1)-2k_2$.
\end{proposition}

\begin{proposition} Let $\boldsymbol \alpha\in Z_{2,I}$, i.e. (up to a  permutation of $1,2,3$) satisfy
\[\alpha_1 = 2k_1, \alpha_2 =-(n-1)-2k_2, \alpha_3 =-(n-1)-2k_3
\]
where $k_1,k_2,k_3\in \mathbb N, k_2+k_3-k_1 \geq 0$. Then $\boldsymbol \alpha$ belongs to three  planes of poles, namely
$\beta_1 =-(n-1)-2k_1$, $\beta_2 = -(n-1)-2k_2$ and 
$\beta_1+\beta_2+\beta_3 = -2(n-1)-2(k_2+k_3-k_1)$.
\end{proposition}

\begin{proposition} Let $\boldsymbol \alpha\in Z_{2,II}$, i.e. up to a  permutation of $1,2,3$ satisfy
 \[\alpha_1 = 2k_1,\quad \alpha_2 = 2k_2,\quad \alpha_3 =-2(n-1)-2k_3
 \]
 where $k_1,k_2,k_3\in \mathbb N$ and $k_1+k_2\leq k_3$.
 \smallskip
 
 $i)$ if $n-1$ is odd, then $\boldsymbol \alpha$ belongs to the unique plane of poles \[\alpha_1+\alpha_2+\alpha_3 = -2(n-1)-2(-k_1-k_2+k_3)\ .\]
\smallskip

 $ii)$ if $n-1$ is even, then $\boldsymbol \alpha$ belongs to two  planes of poles, given respectively  by the equations
 \[\alpha_3 =-(n-1)-2(\rho+k_3)\quad \text{ and }\quad  \alpha_1+\alpha_2+\alpha_3 = -2(n-1)-2(-k_1-k_2+k_3)\ .\] 
\end{proposition}
\begin{proposition} Let $\boldsymbol \alpha\in Z_3$, i.e.
\[\alpha_1 = -(n-1)-2k_1,\alpha_2=-(n-1)-2k_2, \alpha_3 =-(n-1)-2k_3
\]
where $k_1,k_2,k_3\in \mathbb N$.

$i)$ if $n-1$ is odd, then $\boldsymbol \alpha$ belongs to exactly three planes of poles.

$ii)$ if $n-1$ is even, then $\boldsymbol \alpha$ belongs to three planes of poles of type I and moreover to the plane of poles of type II given by the equation
\[\alpha_1+\alpha_2+\alpha_3 = -2(n-1)-2(\rho+k_1+k_2+k_3)\ .
\]
\end{proposition}

\section{ The holomorphic families $\widetilde {\mathcal T}^{(j,k)}_{\,.\,,\,.}$}

The first holomorphic families to be introduced are related to the residues of the meromorphic distribution-valued function $\boldsymbol \alpha \longmapsto \mathcal K_{\boldsymbol \alpha}$ at poles of type I (see \cite{bc} section 2). 
Let $k\in \mathbb N$ and consider the plane of poles of type (say)\footnote{Needless to say, a similar construction can be made when $\boldsymbol \alpha$ is a pole of type I$_1$ or I$_2$} I$_3$ given by the equation  $\alpha_3 = -(n-1)-2k$. Use $\alpha_1,\alpha_2$ as coordinates in this plane. 

Introduce the differential operator on $S$ given by
\[\Delta_k = \prod_{j=1}^k \big(\Delta-(\rho+j-1)(\rho-j)\big)\ ,
\]
where $\Delta$ is the usual Laplacian on $S$. Recall that $\Delta_k$ is a \emph{conformally covariant} differential operator on $S$, in the sense that, for any $g\in G$,
\[\Delta_k\circ \pi_{-k}(g) = \pi_k(g)\circ \Delta_k\ .
\]
Let $\alpha_1, \alpha_2\in \mathbb C$. For $f_1, f_2, f_3 \in \mathcal C^\infty(S)$, consider the expression defined by 
\begin{equation}\label{Tk}
{\mathcal T}^{(3,k)}_{\alpha_1,\alpha_2} (f_1,f_2,f_3) = \iint_{\!S\times S} \!\!\!\!\!\!\!\!f_2(y)f_3(z) \Delta_k[ f_1(\, .\, ) \vert z-\,.\,\vert^{\alpha_2}](y) \vert y-z\vert^{\alpha_1} dy\,dz
\end{equation}
or equivalently, for $\varphi\in \mathcal C^\infty(S\times S\times S)$
\begin{equation}\label{intT}
\mathcal T^{(3,k)}_{\alpha_1,\alpha_2}(\varphi) = 
\iint_{S\times S} \Delta_{k}^{(1)}\big(\varphi(.,y,z)\vert z-.\vert^{\alpha_2}\big) (y) \vert y-z\vert^{\alpha_1} dydz\ ,
\end{equation}
where $\Delta_{k}^{(1)}$ is the covariant differential operator $\Delta_{k}$ acting on $x$.

\begin{proposition}\label{residuetype1} {\ }
\smallskip

$i)$ for $\varphi\in \mathcal C^\infty(S\times S\times S)$ with $Supp(\varphi)\subset \mathcal O_4^c$, the integral \eqref{intT} defines a $\boldsymbol \lambda$-invariant distribution $\mathcal T^{(3,k)}_{\alpha_1,\alpha_2,\mathcal O_4^c}$ on $\mathcal O_4^c$.
\smallskip

 $ii)$
for $\Re \alpha_1$ and $\Re \alpha_2$ large enough, the integral \eqref{Tk} converges and defines a continuous trilinear form on $\mathcal C^\infty(S)$.
\smallskip

$iii)$ the trilinear form $ \mathcal T^{(3,k)}_{\alpha_1,\,\alpha_2} $ can be continued meromorphically (in the parameters $\alpha_1,\alpha_2)$ to $\mathbb C^2$, with simple poles along the complex lines

\[\alpha_1+\alpha_2 = -(n-1)+2k-2l, \quad l=0,1,2,\dots\quad .
\]
\end{proposition}

\begin{proof} For $i)$, let $Supp(\varphi)\subset \mathcal O_4^c$. The function
\[  \Phi(y,z) = \Delta_k^{(1)}\big(\varphi(.,y,z)\vert z-.\vert^{\beta_2}\big) (y)
\]
is in $\mathcal C^\infty(S\times S)$ and $Supp(\Phi) \cap \{y=z\}=\emptyset$. Hence the integral \eqref{intT} converges and defines a distribution on
$\mathcal O^c_4$. The open set $\mathcal O_4^c$ is invariant by $G$ and the distribution $\mathcal T^{3,k_3}_{\beta_1,\beta_2,\, \mathcal O_4^c}$ is verified to be $\boldsymbol \lambda$-invariant as in \cite{bc} Proposition  2.1.

The convergence result $ii)$ and the meromorphic extension $iii)$ are established  in \cite{bc} Theorem 2.2. The fact that the poles are simple (although not explicitly stated) is an easy consequence of the proof.
\end{proof}
\noindent
{\bf Remark.} If $\boldsymbol \alpha$ satisfies $\alpha_3 = -(n-1)-2k$ and $\alpha_1+\alpha_2 = -(n-1)+2k-2l$ for some $l\in \mathbb N$, then $\alpha_1+\alpha_2+\alpha_3 = -2(n-1)-2l$, so that $\boldsymbol \alpha$ is a pole of type I + II. 

Now renormalize ${\mathcal T}^{(3,k)}_{\alpha_1,\,\alpha_2}$ by setting
\begin{equation}\label{normT}
 \widetilde {\mathcal T}^{(3,k)}_{\alpha_1,\,\alpha_2} = \frac{1}{\Gamma(\frac{\alpha_1+\alpha_2}{2}+\rho-k)} \ T^{(3,k)}_{\alpha_1,\,\alpha_2}\ ,
\end{equation}
and use Hartog's theorem to  extend  $(\alpha_1,\alpha_2)\longmapsto \widetilde {\mathcal T}^{(3,k)}_{\alpha_1,\,\alpha_2} $ holomorphically to all of $\mathbb C^2$.

Recall that a pole $\boldsymbol \alpha$ is said to be \emph{generic} if $\boldsymbol \alpha$ belongs to a unique plane of poles.
For $\boldsymbol \alpha$ a generic pole (say) of type I$_3$, the residue $Res(\mathcal K_{\boldsymbol \alpha})$ is defined as
 \[ Res(\mathcal K_{\boldsymbol \alpha}) = \lim_{s\rightarrow 0} s\,\mathcal K_{\alpha_1,\alpha_2, \alpha_3+2s}\ ,\]
and was computed in \cite{bc}.

\begin{proposition} Let  $\boldsymbol \alpha$ satisfy $\alpha_3 =-(n-1)-2k$ and assume that $\boldsymbol \alpha$ is a generic pole. The residue of $\mathcal K_{\boldsymbol \alpha}$ at $\boldsymbol \alpha$ is equal to 
\begin{equation}\label{resI}
Res\big(\mathcal K_{\boldsymbol \alpha})\  =\  \frac{\pi^\rho\, 2^{-2k}}{\Gamma(\rho+k) k!}\,
\mathcal T^{(3,k)}_{\alpha_1,\alpha_2}\ .
\end{equation}
\end{proposition}

\begin{proposition} Let $k\in \mathbb N$. Then
\begin{equation}\label{KT}
\widetilde {\mathcal K}_{\alpha_1,\,\alpha_2,\,-(n-1)-2k}  =  \frac{(-1)^k\,2^{-2k}  \,\pi^\rho}{\Gamma(\rho+k) \Gamma(\rho+\frac{\alpha_1}{2}) \Gamma(\rho+\frac{\alpha_2}{2})} \ \widetilde {T}^{(3,k)}_{\alpha_1,\,\alpha_2}\ .
\end{equation}
\end{proposition}

\begin{proof} First assume that $\boldsymbol \alpha$ is a generic pole. For $s\in \mathbb C$, let
\[\boldsymbol \alpha(s)=(\alpha_1,\alpha_2,-(n-1-2k+2s)\ .\]
For $s\neq 0$, $\boldsymbol \alpha(s)$ is not a pole, hence
\[ \widetilde {\mathcal K}_{\boldsymbol \alpha(s)} = \frac{1}{\Gamma(\frac{\alpha_1}{2}+\rho)\Gamma(\frac{\alpha_2}{2}+\rho)\Gamma( -k+s) \Gamma( \frac{\alpha_1+\alpha_2}{2}+\rho-k+s)}\,\mathcal K_{\boldsymbol \alpha(s)}
\]
As $s\rightarrow 0$, $\displaystyle \frac{1}{\Gamma(-k+s)} \sim (-1)^k\, k!\, s$, so that
\[\widetilde {\mathcal K}_{\alpha_1,\,\alpha_2,\,-(n-1)-2k}  = \frac{(-1)^k \,k!}{\Gamma(\frac{\alpha_1}{2} +\rho)\,\Gamma(\frac{\alpha_2}{2} +\rho)\,\Gamma(\frac{\alpha_1+\alpha_2}{2} +\rho-k)}\, \text{Res\ } \big({\mathcal K}_{\boldsymbol \alpha}\big)
\]
and \eqref{KT} follows from \eqref{normT} and \eqref{resI}. 

As both sides of \eqref{KT} depend holomorphically on $(\alpha_1,\alpha_2)$, the equality is valid on all of $\mathbb C^2$ by analytic continuation.
\end{proof}
The $K$-coefficients of  $\widetilde {\mathcal T}^{(3,k)}_{\alpha_1,\,\alpha_2}$  can be computed from \eqref{KT} and the evaluation of the $K$-coefficients of $\widetilde {\mathcal K}_{\boldsymbol \alpha}$ (see \eqref{Kpalpha} in Appendix).
 
\begin{proposition} Let $\alpha_1,\alpha_2\in \mathbb C$, and let $k\in \mathbb N$. Then
\[\widetilde {\mathcal T}^{(3,k)}_{\alpha_1,\alpha_2}(p_{a_1,a_2,a_3}) = 0 \qquad \text{ for }a_3>k\] and for $a_3\leq k$
\begin{equation}\label{Tp}
\begin{split}
&\widetilde {\mathcal T}^{(3,k)}_{\alpha_1,\alpha_2}(p_{a_1,a_2,a_3})=
2^{-\frac{5}{2}(n-1)-1}\,\pi^{n-1}\, (-1)^k\,\Gamma(\rho+k) \,2^{\alpha_1+\alpha_2}\,2^{2(a_1+a_2+a_3)}\\
&\times (-k)_{a_3}\,\big(\frac{\alpha_1+\alpha_2}{2} +\rho -k\big)_{a_1+a_2+a_3}\\
&\times \big(\frac{\alpha_1}{2} +\rho+a_1-(k-a_3)\big)_{k-a_3}\,\big(\frac{\alpha_2}{2} +\rho+a_2-(k-a_3)\big)_{k-a_3}\\
&\times \frac{ 1}{\Gamma(\frac{\alpha_1+\alpha_2}{2}+2\rho+a_1+a_2)}\quad .
\end{split}
\end{equation}

\end{proposition}

\begin{proposition}\label{suppT3}
 Let $\boldsymbol \alpha=(\alpha_1,\alpha_2,-(n-1)-2k)$.
\smallskip

$i)$ assume that $\boldsymbol \alpha$ is not a pole of type II. Then
\[Supp(\widetilde {\mathcal T}_{\alpha_1,\alpha_2}^{(3,k)}) = \overline {\mathcal O_3}\ .
\]

$ii)$ assume that $\boldsymbol \alpha$ is moreover a pole of type II. Then
\[Supp(\widetilde {\mathcal T}_{\alpha_1,\alpha_2}^{(3,k)}) \subset \mathcal O_4\ .
\] 
\end{proposition}

\begin{proof} In general, as $\boldsymbol \alpha$  is a pole of type I$_3$, $Supp(\widetilde {\mathcal K}_{\boldsymbol \alpha}) \subset \overline {\mathcal O_3}$ (see Proposition 3.3 in \cite{c}), from which follows easily $Supp(\widetilde{\mathcal T}_{\alpha_1,\alpha_2}^{(3,k)})\subset \overline {\mathcal O_3}$. As $Supp(\widetilde{\mathcal T}_{\alpha_1,\alpha_2}^{(3,k)})$ has to be invariant by $G$, either $Supp(\widetilde{\mathcal T}_{\alpha_1,\alpha_2}^{(3,k)}) = \overline {\mathcal O_3}$ or $Supp(\widetilde{\mathcal T}_{\alpha_1,\alpha_2}^{(3,k)}) \subset \mathcal O_4$.

For $i)$, assume that $\boldsymbol \alpha$ is not a pole of type II, which is equivalent to assuming that $\frac{\alpha_1+\alpha_2}{2} +\rho -k\notin -\mathbb N$. Hence the factor $\big(\frac{\alpha_1+\alpha_2}{2} +\rho -k\big)_{a_1+a_2+a_3}$ in \eqref{Tp} never vanishes. Now let $a_3=k$ and choose $a_1,a_2$ large enough. Then $\widetilde {\mathcal T}^{(3,k)}_{\alpha_1,\alpha_2}(p_{a_1,a_2,a_3}) \neq 0$. But the function $p_{a_1,a_2,a_3}$ vanishes on $\mathcal O_4$ at an arbitrary large order for an appropriate choice of $a_1$ and $a_2$.  This is incompatible with the possibility that $Supp(\widetilde{\mathcal T}_{\alpha_1,\alpha_2}^{(3,k)})\subset\mathcal O_4$ (see similar argument in the proof of Lemma 6.5 in \cite{c}). 

For $ii)$, assume that $\boldsymbol \alpha$ is a pole of type II. 
 By Proposition 3.3 in \cite{c} $Supp(\widetilde {\mathcal K}_{\boldsymbol \alpha}) \subset \mathcal O_4$. If $ \alpha_1, \alpha_2 \notin -(n-1)-2\mathbb N$, then \eqref{KT} shows that $Supp(\widetilde{\mathcal T}_{\alpha_1,\alpha_2}^{(3,k)}) \subset \mathcal O_4$. The result follows by continuity. 
\end{proof}

\section{The multiplicity 2 theorem for $\boldsymbol \alpha\in Z_{1,I}$}

Let  $\boldsymbol \alpha=(\alpha_1,\alpha_2, \alpha_3)$ belongs to $Z_{1,I}$. Up to a permutation of $\{1,2,3\}$, $\boldsymbol \alpha$ satisfies (see Proposition \ref{Z1I})
\begin{equation}\label{assZ1I}
\begin{split}
\alpha_2 = -(n-1)-2k_2,\qquad \alpha_3 =-(n-1)-2k_3, \\ \alpha_1 \notin \{0,2,\dots, 2(k_2+k_3)\}\cup \big(-(n-1)-2\mathbb N\big)\ .
\end{split}
\end{equation}

The next proposition is a preparation for the main results. The domain of $\boldsymbol \alpha$'s  to which it applies is  larger than $Z_{1,I}$.

\begin{proposition}\label{indepT2T3}
 Let $\boldsymbol \alpha$ satisfy
\[ \alpha_2=-(n-1)-2k_2, \quad \alpha_3 =-(n-1)-2k_3,\quad   \alpha_1 \notin \{ 0,2,\dots, 2(k_2+k_3)\}\ .
\]
 Then $\widetilde {\mathcal T}^{(2,k_2)}_{\alpha_1,-(n-1)-2k_3}$ and $\widetilde {\mathcal T}^{(3,k_3)}_{\alpha_1,-(n-1)-2k_2}$ are two linearly independant $\boldsymbol \lambda$-invariant distributions.
\end{proposition}

\begin{proof} For $a_3 = k_3$, the value of the $K$-coefficient $\widetilde {\mathcal T}^{(3,k_3)}_{\alpha_1,\alpha_2}(p_{a_1,a_2,a_3})$ given by \eqref{Tp} takes the particular form\footnote{Recall that the symbol NVT is used for a non vanishing term}
\[ NVT \times \big(\frac{\alpha_1}{2}-k_2-k_3\big)_{a_1+a_2+k_3}\ 
 \Gamma(\frac{\alpha_1}{2} -k_2+\rho+a_1+a_2)^{-1}\]
Assume first that $\alpha_1\notin 2k_2+2k_3-2\mathbb N$. Then the first factor is always $\neq 0$. The second factor, for a given value of $a_2$, does not vanish for $a_1$ large enough. Hence  for $a_3=k_3$, $a_2 = k_2+1$ and $a_1$ large enough, 
\[\widetilde {\mathcal T}^{(3,k_3)}_{\alpha_1,-(n-1) -2 k_2}(p_{a_1,\, k_2+1,\, k_3})\neq 0\ ,\]
 whereas \[\widetilde {\mathcal T}^{(2,k_2)}_{\alpha_1,\,-(n-1) -2 k_3}(p_{a_1,\,k_2+1,\,k_3})=0\ ,\] because $a_2=k_2+1> k_2$, and the statement follows in this case. 
 \smallskip
  
 Next, assume that $\alpha_1 =-2l_1$ for  some integer $l_1\geq 1$. Let \[a_3 = k_3,\quad  a_2 = k_2+1, \quad a_1 = l_1-1\ .\]
 Then
 $(-l_1-k_2-k_3)_{a_1+a_2+a_3}= (-l_1-k_2-k_3)_{l_1+k_2+k_3}\neq 0
 $.
Moreover, $\frac{\alpha_1}{2} -k_2+\rho+a_1+a_2 = \rho$. Hence
 $ \widetilde {\mathcal T}^{(3,k_3)}_{\alpha_1,-(n-1) -2 k_2}(p_{a_1,a_2,a_3})\neq 0$, whereas  
 $\widetilde {\mathcal T}^{(2,k_2)}_{\alpha_1,-(n-1)-2k_3}(p_{a_1,a_2,a_3})=0$ as $a_2>k_2$.
 
Permuting the indices $2$ and $3$, the independence of the two trilinear forms follows. 
\end{proof}

\noindent
{\bf Remark.}  Let $\boldsymbol \alpha \in \mathbb C^3$ such that
\[\alpha_2 = -(n-1)-2k_2,\quad \alpha_3 = -(n-1)-2k_3,\quad\alpha_1\in \{ 0,2, \dots, 2k_2+2k_3\}\ .\]
 Then the two trilinear forms $\widetilde {\mathcal T}^{(3,\,k_3)}_{\alpha_1,\,-(n-1) -2 k_2}$ and $\widetilde {\mathcal T}^{(2,\,k_2)}_{\alpha_1,\,-(n-1)-2k_3}$ are proportional, as will be proved later (see Proposition \ref{dKZ2I}).

The main result of this section can now be formulated.

\begin{theorem}\label{mult2Z1I}
 Let $\boldsymbol \alpha \in Z_{1,I}$ and let $\boldsymbol \lambda$ be its associated spectral parameter. Then
$\dim Tri(\boldsymbol \lambda)= 2\ .
$
More precisely, assume that $\boldsymbol \alpha$ satisfies \eqref{assZ1I}. Then
\[Tri(\boldsymbol \lambda) = \mathbb C \widetilde {\mathcal T}^{(2, k_2)}_{\alpha_1,\alpha_3}\oplus \mathbb C  \widetilde {\mathcal T}^{(3, k_3)}_{\alpha_1,\alpha_2}\ .
\]
\end{theorem}

Because of Proposition \ref{indepT2T3}, it suffices to prove that $\dim Tri(\boldsymbol \lambda)\leq 2$. The proof of this inequality strongly depends on the nature of $\boldsymbol \alpha$ as a pole (see Proposition \ref{poleZ1I}), hence will be divided in two subsections : one for the general case ($\boldsymbol \alpha$ is not a pole of type II) and one for the special case ($\boldsymbol \alpha$ is moreover a pole of type II).

\subsection{The general case}

 Let $\boldsymbol \alpha$ satisfy conditions \eqref{assZ1I}, and in this subsection, assume moreover that $\boldsymbol \alpha$ is not a pole of type II. This amounts to
\begin{equation}\label{assZ1Ireg}
\begin{split}
\alpha_2=-(n-1)-2k_2,\qquad \alpha_3 = -(n-1)-2k_3\\\alpha_1 \notin  \big(2k_2+2k_3-2\mathbb N\big) \cup \big(-(n-1)-2\mathbb N\big)\ 
\end{split}
\end{equation}
 for some $k_2,k_3\in \mathbb N$.

 In this case, $Supp(\widetilde {\mathcal T}^{(2,k_2)}_{\alpha_1,\alpha_3}) = \overline{\mathcal O_2}$ and $Supp(\widetilde {\mathcal T}^{(3,k_3)}_{\alpha_1,\alpha_2}) = \overline{\mathcal O_3}$ (Proposition \ref{suppT3}), which gives in this case another proof of the linear  independence of   the two distributions (Proposition \ref{indepT2T3}).
 
As suggested in the introduction, $\widetilde {\mathcal T}^{(2,k_2)}_{\alpha_1,\alpha_3}$ and  $\widetilde {\mathcal T}^{(3,k_3)}_{\alpha_1,\alpha_2}$ are related to partial derivatives of $\boldsymbol \beta\longmapsto \widetilde {\mathcal K}_{\boldsymbol \beta}$ at $\boldsymbol \alpha$.

\begin{lemma}\label{derKT}
\[ \frac{d}{ds} \Big( \widetilde {\mathcal K}_{\alpha_1,\alpha_2+2s, \alpha_3}\Big)_{s=0} = \frac{(-1)^{k_3} 2^{-2k_3}\pi^\rho(-1)^{k_2}k_2!}{\Gamma(\rho+k_3) \Gamma(\rho+\frac{\alpha_1}{2})}\widetilde {\mathcal T}^{(3,k_3)}_{\alpha_1,\alpha_2}
\]
\end{lemma} 
 \begin{proof} For $s\in \mathbb C$, let $\boldsymbol \alpha(s) = (\alpha_1,\alpha_2+2s,\alpha_3)$ and use \eqref{KT} to relate $\widetilde{\mathcal K}_{\boldsymbol \alpha(s)}$ and $\widetilde {\mathcal T}^{(3,k_3)}_{\alpha_1,\alpha_2+2s}$. As $s\rightarrow 0$, \[\frac{1}{\Gamma(\rho+\frac{\alpha_2}{2}+s)} = \frac{1}{\Gamma(-k_2+s)}\sim (-1)^{k_2} k_2! s\]
 and pass to the limit to obtain the lemma.
 \end{proof}
 \begin{proposition} \label{noextZ1Ireg}
 Let $\boldsymbol \alpha$ satisfy \eqref{assZ1Ireg}. Let $\boldsymbol \lambda$ be its associated spectral parameter.
 The distribution $\mathcal K_{ \boldsymbol\alpha, \mathcal O_0}$ cannot be extended to a $\boldsymbol \lambda$-invariant distribution on $S\times S\times S$.
 \end{proposition}
 \begin{proof} 
At $\boldsymbol \alpha$, exactly two $\Gamma$-factors in the normalization of $\mathcal K_{\boldsymbol \beta}$ become singular, namely $\Gamma(\frac{\beta_2}{2}+\rho)$ and $\Gamma(\frac{\beta_3}{2}+\rho)$. For $s\in \mathbb C$, let 
\[\boldsymbol \alpha(s) = (\alpha_1, \alpha_2 + 2s, \alpha_3+2s),
\qquad \boldsymbol \lambda(s) = (\lambda_1+2s, \lambda_2+s, \lambda_3+s)\ .
\]
Observe that $\boldsymbol \alpha(0) =\boldsymbol \alpha$, that for $s\neq 0, \vert s\vert \text{ small}$, $\boldsymbol \alpha(s)$ is not a pole. Moreover,  $\boldsymbol \alpha(s)$ is transverse to both planes of poles $\beta_2=-(n-1)-2k_2$ and  $\beta_3=-(n-1)-2k_3$.
For $s\neq 0$ and $\vert s\vert $ small, consider the distribution $\mathcal F(s)$ defined by
\[\mathcal F(s) = \frac{(-1)^{k_2}}{k_2!}\,\frac{(-1)^{k_3}}{k_3!}\Gamma(\frac{\alpha_1}{2}+\rho)\, \Gamma(\frac{\alpha_1}{2}-k_2-k_3+2s)\, \frac{1}{s}\,\widetilde {\mathcal K}_{\boldsymbol \alpha(s)}\ .
\]
As $\boldsymbol \alpha\in Z$, the (distribution-valued) function $s\longmapsto \mathcal F(s)$ can be extended  analytically near $0$ and its Taylor expansion of order $2$ reads
\[\mathcal F(s) =  F_0+ s \,F_1+ O(s^2)\ ,
\]
where $F_0,F_1$ are distributions on $S\times S\times S$. Notice further that $\mathcal F(s)$ is $\boldsymbol \lambda(s)$-invariant.

\begin{lemma}\label{F01Z1I} The distributions $F_0$ and $F_1$ satisfy 
\smallskip

$i)$  $F_0$ is  $\boldsymbol \lambda$-invariant and $Supp(F_0) = \overline{\mathcal O_2\cup \mathcal O_3}$

\smallskip

$ii)$ the restriction of $F_1$ to $\mathcal O_0$ is equal to $\mathcal K_{\boldsymbol \alpha,\, \mathcal O_0}$

\end{lemma}

\begin{proof} \ 
For $i)$, 
\[F_0 = NVT \frac{d}{ds}\Big( \widetilde {\mathcal K}_{\alpha(s)}\Big)_{s=0}
= NVT \Big(\frac{d}{ds} \big(\widetilde{\mathcal K}_{\alpha_1, \alpha_2+2s, \alpha_3}\big)_{s=0} +\frac{d}{ds} \big(\widetilde{\mathcal K}_{\alpha_1, \alpha_2, \alpha_3+2s}\big)_{s=0}\Big) \]
and $i)$ follows from Lemma \ref{derKT} and Proposition \ref{indepT2T3}.

For $ii)$, let $\varphi\in \mathcal C^\infty_c(\mathcal O_0)$. Then, for $s\neq 0$, $\mathcal K_{\boldsymbol \alpha(s)}(\varphi)$ is well defined and satisfies
\[ \mathcal K_{\boldsymbol \alpha(s)}(\varphi) = (-1)^{k_2}\,k_2! \,\Gamma( -k_2+s)\, (-1)^{k_3}\, k_3!\,\Gamma(-k_3+s) s\,\mathcal F(s)(\varphi)\]
\[ =  \frac{1}{s}F_0(\varphi) + F_1(\varphi) +O(s)\ .
\]
As $s\longrightarrow 0$, the left hand side converges to $\mathcal K_{\boldsymbol \alpha, \mathcal O_0}(\varphi)$. Hence,
\[F_0(\varphi) = 0,\quad F_1(\varphi) = \mathcal K_{\boldsymbol \alpha}(\varphi),
\]
thus proving $ii)$.
\end{proof}

Let $\mathcal O = \mathcal O_0\cup \mathcal O_2$. The open subset $\mathcal O $ is  $G$-invariant, contains $\mathcal O_2$ as a (relatively) closed submanifold. Now restrict $\mathcal F(s)$ to $\mathcal O$, and similarly for $F_0$ and $F_1$. Then, $\mathcal F(s)_{\vert \mathcal O}$ is $\boldsymbol \lambda(s)$-invariant, $\mathcal F(s)_{\vert \mathcal O} = F_{0\vert\mathcal O}+s\,F_{1\vert\mathcal O} +O(s^2)$, and $Supp(F_{0\vert \mathcal O}) = \mathcal O_2$. By Proposition A\ref{noextI}, $\mathcal K_{\boldsymbol \alpha, \mathcal O_0}$ cannnot be extended to a $\boldsymbol \lambda$-invariant distribution on $\mathcal O$, a fortiori to $S\times S\times S$, which finishes the proof of   Proposition \ref{noextZ1Ireg}.
 \end{proof}

After this preparation, we are in condition to prove Theorem \ref{mult2Z1I} in the general case.
Let $T$ be a $\boldsymbol \lambda$-invariant distribution. The restriction  of $T$ to $\mathcal O_0$ has to be proportional to $\mathcal K_{\boldsymbol \alpha, \mathcal O_0}$, hence has to be $0$ by Proposition \ref{noextZ1Ireg}. Next,  as $\alpha_1\notin -(n-1)-2\mathbb N$, the restriction of $T$ to $\mathcal O_0 \cup \mathcal O_1$ has to be $0$ (use Lemma 4.1 in \cite{c}). So, $T$ is supported on $\mathcal O_2\cup \mathcal O_3\cup \mathcal O_4$. Now we first restrict $T$ to the $G$-invariant open set $ \mathcal O_0\cup \mathcal O_1\cup \mathcal O_2$. The restriction $T_{\vert  \mathcal O_0\cup \mathcal O_1\cup \mathcal O_2}$ is $\boldsymbol \lambda$-invariant and supported on $\mathcal O_2$. By Lemma 4.1 in \cite{c}, $T_{\vert \mathcal O_0\cup \mathcal O_1\cup \mathcal O_2}$ has to be proportional to (the restriction to $\mathcal O_0\cup \mathcal O_1\cup \mathcal O_2$ of) ${\widetilde {\mathcal T}_{\alpha_1,\alpha_3}^{(2,k_2)}}$. A similar argument applies for the restriction of $T$ to $\mathcal O_0\cup \mathcal O_1\cup \mathcal O_3$. Hence, there exist two constants $C_2$ and $C_3$ such that
\[T - C_2\widetilde{\mathcal T}_{\alpha_1,\alpha_3}^{(2,k_2)}-C_3 \widetilde{\mathcal T}_{\alpha_1,\alpha_2}^{(3,k_3)}
\]
is supported on $\mathcal O_4$, and still $\boldsymbol \lambda$-invariant. By Lemma 4.2 in \cite{c}, as $\alpha_1+\alpha_2+\alpha_3\notin -2(n-1)-2\mathbb N$, $T - C_2\widetilde{ \mathcal T}_{\alpha_1,\alpha_3}^{(2,k_2)}-C_3 \widetilde{\mathcal T}_{\alpha_1,\alpha_2}^{(3,k_3)}
=0$. Hence any $\boldsymbol \lambda$-invariant distribution is a linear combination of $\widetilde{ \mathcal T}_{\alpha_1,\alpha_3}^{(2,k_2)}$ and $\widetilde{ \mathcal T}_{\alpha_1,\alpha_2}^{(3,k_3)}$. Q.E.D.

\subsection {The special case} 

In this subsection, assume $\boldsymbol \alpha$ satisfies \eqref{assZ1I} and is moreover a pole of type II, i.e. $\boldsymbol \alpha$ is assumed to satisfy the conditions
\begin{equation}\label{assZ1Is}
\begin{split}
 &\alpha_2 = -(n-1)-2k_2,\quad \alpha_3 = -(n-1)-2k_3,\quad \alpha_1 =-2p_1\\
& k_2,k_3\in \mathbb N,\quad  p_1 \in \mathbb \{1,2,\dots\},\quad p_1 \notin \rho+\mathbb N\ .
\end{split}
\end{equation}
Notice in particular that $\alpha_1+\alpha_2+\alpha_3 = -2(n-1)-2k$, where
$k=p_1+k_2+k_3$ so that $\boldsymbol \alpha$ is indeed a pole of type II.

Let $\boldsymbol \lambda$ be the associated spectral parameter.  In this case, the two $\boldsymbol \lambda$-invariant distributions$\widetilde {\mathcal T}^{(2,k_2)}_{\alpha_1,\alpha_3}$ and $\widetilde {\mathcal T}^{(3,k_3)}_{\alpha_1,\alpha_2}$ are now supported on $\mathcal O_4$, but still linearly independent (see Proposition \ref{suppT3} and Proposition  \ref{indepT2T3}). Some preparatory results are needed.

\begin{proposition}\label{noextTT}
Let $A_2$ and $A_3$ be two complex numbers, $(A_2,A_3)\neq (0,0)$.  The distribution on $\mathcal O_4^c$
\[A_2\mathcal T^{(2,k_2)}_{\alpha_1,\alpha_3,\, \mathcal O_4^c} + A_3\mathcal T^{(3,k_3)}_{\alpha_1,\alpha_2,\, \mathcal O_4^c}
\]
cannot be extended to a $\boldsymbol \lambda$-invariant distribution on $S\times S\times S$.
\end{proposition}

\begin{proof} For $s\in \mathbb C$  let 
\[\boldsymbol \alpha(s) = (\alpha_1+2s, \alpha_2,\alpha_3), \qquad \boldsymbol \lambda(s)=(\lambda_1,\lambda_2+s, \lambda_3+s)\] 
and notice that $\boldsymbol \alpha(s)$ is transverse to the plane $\beta_1+\beta_2+\beta_3 =-2(n-1)-2k$ at $\boldsymbol \alpha=\boldsymbol \alpha(0)$. Let
\[\mathcal F(s)= \frac{(-1)^k}{k!} \big(A_2\widetilde {\mathcal T}^{(2,k_2)}_{\alpha_1+2s,\alpha_3} +A_3 \widetilde { \mathcal T}^{(3,k_3)}_{\alpha_1+2s,\alpha_2}\big)\ .\]
Notice that the distribution $\mathcal F(s)$ is $\boldsymbol \lambda(s)$-invariant.
Let \[\mathcal F(s) = F_0+sF_1+O(\vert s\vert ^2)\] be the Taylor expansion of $\mathcal F$ at $0$.

\begin{lemma}\
\smallskip

$i)$ $Supp(F_0) = \mathcal O_4$
\smallskip

$ii)$ on $O_4^c$
\[{F_1}_{\vert \mathcal O_4^c}=A_2\mathcal T^{(2,k_2)}_{\alpha_1,\alpha_3,\, \mathcal O_4^c} + A_3\mathcal T^{(3,k_3)}_{\alpha_1,\alpha_2,\, \mathcal O_4^c}\]
\end{lemma}
\begin{proof}
For $i)$, note that $F_0= \frac{(-1)^k}{k!} \big(A_2\widetilde {\mathcal T}^{(2,k_2)}_{\alpha_1,\alpha_3} +A_3 \widetilde { \mathcal T}^{(3,k_3)}_{\alpha_1,\alpha_2}\big)$, hence $ F_0\neq 0$ and $Supp(F_0) = \mathcal O_4$.

For $ii)$, Let $\varphi\in \mathcal C^\infty_c(\mathcal O_4^c)$. Then, using \eqref{normT}
\begin{equation}\label{FsTT}
\mathcal F(s)(\varphi) =\frac{(-1)^k}{k!} \frac{1}{\Gamma(-k+s)} (A_2 \mathcal T^{(2,k_2)}_{\alpha_1+2s, \alpha_3} (\varphi)+ A_3 \mathcal T^{(3,k_3)}_{\alpha_1+2s, \alpha_2}(\varphi)\big)
\end{equation}

Let $s\longrightarrow 0$  to obtain 
\[F_1(\varphi) = A_2 \mathcal T^{(2,k_2)}_{\alpha_1+2s, \alpha_3,\mathcal O_4^c} (\varphi)+ A_3 \mathcal T^{(3,k_3)}_{\alpha_1+2s, \alpha_2,\mathcal O_4^c}(\varphi)
\]
proving $ii)$.
\end{proof}

Now apply Proposition A\ref{noextII} to conclude that $A_2\mathcal T^{(2,k_2)}_{\alpha_1,\alpha_3,\, \mathcal O_4^c} + A_3\mathcal T^{(3,k_3)}_{\alpha_1,\alpha_2,\, \mathcal O_4^c}$ cannot be extended to $S\times S\times S$ as a $\boldsymbol \lambda$-invariant distribution.
\end{proof}
We are ready to start the proof of Theorem \ref{mult2Z1I} for the special case. As in the general case, the following proposition holds.
\begin{proposition}\label{noextZ1Is} Let $\boldsymbol \alpha$ satisfy conditions \eqref{assZ1I}. The distribution $\mathcal K_{\boldsymbol \alpha,\, \mathcal O_0}$ cannot be extended to a $\boldsymbol \lambda$-invariant distribution on $S\times S\times S$.

\end{proposition}
\begin{proof}
Three $\Gamma$ factors used in the renormalization of $\mathcal K_{\boldsymbol \beta}$ are singular at $\boldsymbol \alpha$, namely
$\Gamma(\frac{\beta_2}{2}+\rho), \Gamma( \frac{\beta_3}{2}+\rho)$ and $ \Gamma( \frac{\beta_1+\beta_2+\beta_3}{2}+2\rho)$. For $s\in \mathbb C$, $\vert s\vert$, let
\[\boldsymbol \alpha(s) = (\alpha_1,\,\alpha_2+2s,\, \alpha_3 +2s),\quad \boldsymbol \lambda(s) = (\lambda_1+2s, \lambda_2 +s,\lambda_3+s)\ .
\]
 Notice that $\alpha(s)$ is transverse to the three planes of poles $\beta_2=-(n-1)-2k_2, \beta_3 =-(n-1)-2k_3$ and $\beta_1+\beta_2+\beta_3 = -2(n-1)-2k$.
For $s\neq 0$  let 
 \[\mathcal F(s) = \frac{1}{2}\frac{ (-1)^{k}}{k!}\,\frac{(-1)^{k_2}}{k_2 !} \,\frac{(-1)^{k_3}}{k_3 !}\Gamma(-p_1+\rho)\,\frac{1}{s}\,\widetilde {\mathcal K}_{\boldsymbol \alpha(s)}\ .
 \] 
 As $\boldsymbol \alpha$ is in $Z$, $\mathcal F(s)$ can be extended as a distribution-valued holomorphic function near $0$.  Its Taylor expansion at $0$ reads
 \begin{equation}\label{TaylorF1F2F3}
  \mathcal F(s) = F_0+sF_1+s^2F_2+O( s^3)\ ,
 \end{equation}
 where $F_0,F_1,F_2$ are distributions on $S\times S\times S$.
 
 \begin{lemma}\label{F1F2F3} The distributions $F_0,F_1,F_2$ satisfy
\smallskip

 $i)$ $F_0$ is $\boldsymbol \lambda$-invariant  and $Supp(F_0)=\mathcal O_4$
 \smallskip
 
$ii)$ $Supp(F_1)= \overline{\mathcal O_2\cup \mathcal O_3}$ and the restriction of $F_1$ to $ \mathcal O_4^c$ is $\boldsymbol \lambda$-invariant.
\smallskip

$iii)$ the restriction  of  $F_2$  to $\mathcal O_0$ coincides with  $\mathcal K_{\alpha,\mathcal O_0}$.
 \end{lemma}
 \begin{proof}
 The proof of $i)$ is similar to what was done in the general case (see Lemma \ref{derKT} and Lemma \ref{F01Z1I}), except that the distributions $\widetilde{\mathcal T}^{(2,k_2)}_{\alpha_2,\alpha_3}$ and $\widetilde{\mathcal T}^{(3,k_3)}_{\alpha_1,\alpha_2}$ are now supported in $\mathcal O_4$ but still linearly independent by Proposition \ref{indepT2T3}.
 \smallskip 
 
For $iii)$, let $\varphi\in \mathcal C^\infty_c(\mathcal O_0)$. Then, for $s\neq 0$, 
\[ \big(\mathcal F(s),\varphi\big) =\]\[ \frac{1}{2}\frac{ (-1)^{k}}{k!}\,\frac{1}{\Gamma(-k+2s)}\frac{(-1)^{k_2}}{k_2 !}\frac{1}{\Gamma(-k_2+s)} \,\frac{(-1)^{k_3}}{k_3 !}\frac{1}{\Gamma(-k_3+s)}\frac{1}{s}\mathcal K_{\boldsymbol \alpha(s),\mathcal O_0} (\varphi)\]

\[ =s^2 \mathcal K_{\boldsymbol \alpha, \mathcal O_0} (\varphi) + O(s^3)
\]
As $s\longrightarrow 0$, $F_0(\varphi) = F_1(\varphi)=0$ and
$F_2(\varphi) = K_{\boldsymbol \alpha,\mathcal O_0}(\varphi)$, proving 
 $iii)$ and the fact that  $Supp( F_1)\subset \overline{\mathcal O_1\cup\mathcal O_2\cup \mathcal O_3}$. 

It remains to prove that $Supp(F_1) = \overline{\mathcal O_2\cup \mathcal O_3}$. As
\[F_{1\,\vert \mathcal O_4^c} = \lim_{s\rightarrow 0} \frac{1}{s}\, \mathcal F(s)_{\vert \mathcal O_4^c},
\]
$F_{1\, \vert \mathcal O_4^c}$ is $\boldsymbol \lambda$-invariant. In particular  $F_{1\, \vert \mathcal O_0\cup \mathcal O_1}$ is $\boldsymbol \lambda$-invariant and supported on $\mathcal O_1$. But $\alpha_1\notin \big(-(n-1)-2\mathbb N\big)$, so that $F_{1\, \vert \mathcal O_0\cup \mathcal O_1} = 0$ (use Lemma 4.1 in \cite{c}). Hence $Supp(F_{1\, \vert \mathcal O_4^c} \subset \mathcal O_2\cup \mathcal O_3$. There are only three possibilities : 
\[ Supp(F_{1\,\vert  \mathcal O_4^c}) =\mathcal O_2,\quad  Supp(F_{1\, \vert \mathcal O_4^c})=\mathcal O_3\  \text{ or }\ Supp(F_{1\, \vert \mathcal O_4^c})=\mathcal O_2\cup \mathcal O_3\ .\]
Assume $Supp(F_{1\, \vert \mathcal O_4^c})=\mathcal O_3$. Then $Supp(F_1)= \overline{\mathcal O_3}$. To see a contradiction,  use again the evaluation of the $K$-coefficients of $\widetilde {\mathcal K}_{\boldsymbol \alpha}(p_{a_1,a_2,a_3})$ (see \eqref{Kpalpha}) and choose $a_1,a_2,a_3$ such that $a_2 = 0,\quad a_3 >k_2+k_3,\quad a_1+a_3 >k$. For such a choice, 
\[\widetilde {\mathcal K}_{\boldsymbol \alpha(s)}(p_{a_1,a_2,a_3}) = NVT(s) \times(-k+s)_{a_1+a_2+a_3} (-k_3+s)_{a_3}
\]
where $NVT(s)$ stands for a function which does not vanish at $s=0$. Now both $(-k+s)_{a_1+a_2+a_3}$ and $(-k_3+s)_{a_3}$ have a simple zero for $s=0$, so that 
\[ \widetilde {\mathcal K}_{\boldsymbol \alpha(s)}(p_{a_1,a_2,a_3})= C s^2+O(s^3)
\]
with $C\neq 0$. Hence $F_1( p_{a_1,a_2,a_3})\neq 0$.  But for an appropriate choice of $a_3$, the function $p_{a_1,a_2,a_3}$  vanishes  on the submanifold $\overline{\mathcal O_3}$ at an arbitrary large order. Hence $F_1(p_{a_1,a_2,a_3})\neq 0$ for $a_3$ arbitrary large  is incompatible with $Supp(F_1)\subset \overline{\mathcal O_3}$. Exchanging the role of $2$ and $3$, we also get that $Supp(F_1)$ cannot be contained in $\overline{\mathcal O_2}$. Hence $ii)$ follows. 
\end{proof}
Proposition \ref{noextZ1Is} can now be proved. Restrict the situation to the $G$-invariant open subset $\mathcal O =\mathcal O_0\cup \mathcal O_2$, in which $\mathcal O_2$ is a (relatively) closed submanifold. The distribution  $\frac{1}{s}\, \mathcal F(s)_{\vert \mathcal O}$  can be extended analytically in a neighborhood of $s=0$ and has Taylor expansion 
\[\frac{1}{s}\, \mathcal F(s)_{\vert \mathcal O} = F_{1\, \vert \mathcal O}+s\,F_{2\, \vert  \mathcal O}+O(s^2)\ .
\]
and $Supp(F_{1\, \vert \mathcal O}) =\mathcal O_2$ whereas $F_{2\, \vert \mathcal O_0}$ coincides with $\mathcal K_{\boldsymbol \alpha,\mathcal O_0}$. 
Now apply Proposition A1 to obtain that $\mathcal K_{\boldsymbol \alpha, \mathcal O_0}$ cannot be extended to a $\boldsymbol \lambda$-invariant distribution on $\mathcal O$, a fortiori to a $\boldsymbol \lambda$-invariant distribution on $S\times S\times S$.
\end{proof}

We now proceed to the proof of Theorem \ref{mult2Z1I} for the special case. Let $T$ be a $\boldsymbol \lambda$-invariant distribution. As in the general  case, the restriction of $T$ to $\mathcal O_0$ has to be proportional to $\mathcal K_{\boldsymbol \alpha, \mathcal O_0}$, so must be $0$ thanks to Proposition \ref{noextZ1Is}. Hence $Supp( T) \subset \mathcal O_1\cup \mathcal O_2\cup \mathcal O_3\cup \mathcal O_4$. As $\alpha_1\notin -(n-1)-2\mathbb N$, using Lemma 4.1 in \cite{c}, $Supp( T) \subset  \mathcal O_2\cup \mathcal O_3\cup \mathcal O_4$. Now let $ \mathcal O=\mathcal O_0\cup \mathcal O_1 \cup \mathcal O_2$. The restriction $T_{\vert \mathcal O}$ of $T$ to $\mathcal O$ is supported in the closed submanifold $\mathcal O_2$ and is $\boldsymbol \lambda$-invariant. Now the restriction of $\mathcal T^{(2,k_2)}_{\alpha_1,\alpha_3, \mathcal O_4^c}$ to $\mathcal O$ is a $\boldsymbol \lambda$-invariant distribution on  $\mathcal O$ which is supported in $\mathcal O_2$. By the uniqueness result (see Lemma 4.1 in \cite{c}),  $T_{\vert \mathcal O}$ has to be proportional to $\mathcal T^{(2,k_2)}_{\alpha_1,\alpha_3\, \vert \mathcal O}$. Permuting the  indices $2$ and $3$, a similar result holds on the open set $\mathcal O_0\cup \mathcal O_1\cup\mathcal O_3$.  The two distributions $\mathcal T^{(2,k_2)}_{\alpha_1,\alpha_3,\mathcal O_4^c}$ and  $\mathcal T^{(3,k_3)}_{\alpha_1,\alpha_2,\mathcal O_4^c}$ have disjoint supports, so that there exists two constants $A_2$ and $A_3$ such that 
\[T_{\vert \mathcal O_4^c}= A_2 \mathcal T^{(2,k_2)}_{\alpha_1,\alpha_3, \mathcal O_4^c} + A_3 \mathcal T^{(3,k_3)}_{\alpha_1,\alpha_2,\mathcal O_4^c}\ .\] But this contradicts Proposition
\ref{noextTT}, unless $A_2=A_3=0$. Hence necessarily $Supp(T)\subset \mathcal O_4$. In other words, a $\boldsymbol \lambda$-invariant distribution is supported in the diagonal.  Now $\dim Tri(\boldsymbol \lambda, diag) = \dim S(\lambda_1,\lambda_2; k)$ (see Appendix).

\begin{lemma} Let $\boldsymbol \alpha$ satisfy conditions \eqref{assZ1Is}. Then
\[\dim S(\lambda_1,\lambda_2; k)= 2\ .
\]
\end{lemma}

\begin{proof} Conditions \eqref{assZ1Is} imply
\[\lambda_1 = -\rho-k_2-k_3,\quad \lambda_2 = -p_1-k_3\]
Now $\lambda_1\in \{ -\rho,-\rho-1,\dots, -\rho-(k-1)\}$ as $k_2+k_3=k-p_1$ and $p_1\geq 1$, but $\lambda_1\notin \{-1,-2,\dots ,-k\}$. This is obvious in case $n-1$ is odd, and when $n-1$ is even, the condition $p\notin \rho+\mathbb N$ implies $p<\rho$, so that $\rho+k_2+k_3>p+k_2+k_3=k$ and hence $\lambda_1<-k$. Moreover $k_2+k_3+ p+k_3=k+k_3\geq k_3$. These conditions guarantee that $\dim S(\lambda_1,\lambda_2;k)\leq 2$ (see Appendix 2). As $\dim Tri(\boldsymbol \lambda, diag) = \dim Tri(\boldsymbol \lambda)\geq 2$  by Proposition \ref{indepT2T3}, $\dim S(\lambda_1,\lambda_2;k)\geq 2$ and the equality follows.
\end{proof}
This achieves the proof of Theorem \ref{mult2Z1I} in the exceptional case.
\section{The holomorphic families $\mathcal S^{(k)}_{\lambda_1, \lambda_2}$}

The next holomorphic family of distributions to be constructed is related to the residue of $\mathcal K^{\boldsymbol \lambda}$ at poles of type II. The determination of the residue (at least for a Zariski open set in each plane of poles of type II) was done in in \cite{bc} in the non compact picture. The present exposition is its counterpart in the compact picture, which is more suited to the present purposes.

\subsection{The family $\widetilde{E}_{\lambda_1,\lambda_2}$}
Recall  the definition of the  \emph{normalized Knapp-Stein intertwining operator} $\widetilde J_\lambda$, where $\boldsymbol \lambda\in \mathbb C$
\[\widetilde J_{\lambda} f(x) = \frac{1}{\Gamma(\lambda)}\int_S f(y) \vert x-y\vert^{-(n-1)+2\lambda} f(y)dy\ .
\]
For $\Re(\lambda)>0$, the expression makes sense for any $f\in \mathcal C^\infty(S)$ and defines a continuous operator on $\mathcal C^\infty(S)$. It is then extended by analytic continuation to a holomorphic family of operators on $\mathcal C^\infty(S)$. The main property of this family of operators is
\[\widetilde J_\lambda \circ \pi_{\lambda}(g) = \pi_{-\lambda}(g) \circ \widetilde J_\lambda
\]
for any $g\in G$.

The operator $\widetilde J_{\boldsymbol \lambda}$ is generically invertible. However, there are exceptions, precisely when $\pi_{\lambda}$ is reducible. 

When $\lambda=\rho+k, k\in \mathbb N$, then 
\begin{equation}\label{rho+k}
Im(\widetilde J_{\rho+k}) = \mathcal P_k,\qquad  Ker (\widetilde J_{\rho+k}) = \mathcal P_k^\perp
\end{equation}
 where $\mathcal P_k$ is the space of restrictions to $S$ of polynomials on $E$ of degree $\leq k$. Dually, when $\lambda=-\rho-k$
 \begin{equation}\label{-rho-k}
 Ker(\widetilde J_{-\rho-k)}) = \mathcal P_k,\qquad Im(\widetilde J_{-\rho-k}) = \mathcal P_k^\perp\ .
 \end{equation}
 
Let $M : \mathcal C^\infty(S\times S) \longrightarrow \mathcal C^\infty(S\times S)$ be the operator given by
\[ f\longmapsto Mf,\qquad (Mf)(x,y) = \vert x-y\vert^2 f(x,y)\ .
\]
For any $\lambda_1,\lambda_2\in \mathbb C$, the opertator $M$ intertwines the representations $\pi_{\lambda_1}\otimes \pi_{\lambda_2}$ and 
$\pi_{\lambda_1-1}\otimes \pi_{\lambda_2-1}$. 
This is a consequence of the covariance property under the conformal group of the Euclidean distance  $\vert x-y\vert$ on $S$ (equation  (2) in \cite{c}).

Consider the following diagram

\[\begin{matrix}
\mathcal C^\infty(S)\otimes \mathcal C^\infty(S)&\underrightarrow {\quad E_{\lambda_1,\lambda_2}\quad}&\mathcal C^\infty(S)\otimes \mathcal C^\infty(S)\\
 \begin{matrix} \\\downarrow  \widetilde J_{\lambda_1}\otimes \widetilde J_{\lambda_2}\\ \\\end{matrix}& & \begin{matrix} \\\uparrow  \widetilde J_{-\lambda_1-1}\otimes  \widetilde J_{-\lambda_2-1}\\ \\\end{matrix}\\
 \mathcal C^\infty(S)\otimes \mathcal C^\infty(S)& \underrightarrow {\quad M \quad}&\mathcal C^\infty(S)\otimes \mathcal C^\infty(S)\\
\end{matrix}
\]
and let $E_{\lambda_1,\lambda_2}=(\widetilde J_{-\lambda_1-1}\otimes \widetilde J_{-\lambda_2-1})\circ M\circ (\widetilde J_{\lambda_1}\otimes \widetilde J_{\lambda_2})$.

\begin{proposition}\label{constrE}
{\ }
\smallskip

$i)$ $E_{\lambda_1,\lambda_2}$ intertwines the representations $\pi_{\lambda_1}\otimes \pi_{\lambda_2}$ and $\pi_{\lambda_1+1}\otimes \pi_{\lambda_2+1}$
\smallskip

$ii)$ $E_{\lambda_1,\lambda_2}$ is a differential operator on $S\times S$
\smallskip

$iii)$ $ \Gamma(\lambda_1+\rho+1) \Gamma( -\lambda_1+\rho)\Gamma( \lambda_2+\rho+1)\Gamma(-\lambda_2+\rho)\,E_{\lambda_1,\lambda_2}$ can be analytically extended to $\mathbb C^2$, yielding a homomorphic family of differential operators on $S\times S$.

\end{proposition}

\begin{proof}  $i)$ is true by construction. $ii)$ and $iii)$ are obtained by passing to the non compact picture. The analogous result  on $\mathbb R^{n-1}$ is  Proposition 3.7 in \cite{bc} (notice that  the \emph{unnormalized} Knapp-Stein operators were used in \cite{bc}).
\end{proof}

 However, a more direct (and enlightening) proof of $iii)$ is available in the compact picture, based on the following lemma.
\begin{lemma}\label{0E} {\ }

$i)$ let $\lambda_1\in \big(\rho+\mathbb N\big)$. Then $E_{\lambda_1,\lambda_2} =0$
\smallskip

$ii)$ let $\lambda_1 \in \big(-\rho-1-\mathbb N\big)$. Then $E_{\lambda_1,\lambda_2}=0$.
\end{lemma}
\begin{proof}
Let $\lambda_1=\rho+k$ for some $k\in \mathbb N$.  Then, by \eqref{rho+k}
\[Im(\widetilde J_{\lambda_1} \otimes \widetilde J_{\lambda_2})\subset \mathcal P_k\otimes \mathcal C^\infty(S)\]
 and as $\vert x-y\vert^2 =2(1-\langle x,y\rangle)$ on $S\times S$, 
 \[Im\big(M\circ (\widetilde J_{\lambda_1} \otimes \widetilde J_{\lambda_2})\big)\subset \mathcal P_{k+1} \otimes \mathcal C^\infty(S)\ .
 \]
 As $-\lambda_1-1=-\rho-k-1$, $ Ker(\widetilde J_{-\lambda_1-1})=\mathcal P_{k+1} $ by \eqref{-rho-k} and hence $E_{\lambda_1,\lambda_2} =0$. 
 \smallskip
 
 For $ii)$, let $\lambda_1= -\rho-k-1$, $Im(\widetilde J_{\lambda_1})= \mathcal P_{k+1}^\perp$ by \eqref{-rho-k} so that 
 \[Im\big(M\circ (\widetilde J_{\lambda_1}\otimes \widetilde J_{\lambda_2})\big) \subset \mathcal P_1\mathcal P_{k+1}^\perp \otimes \mathcal C^\infty(S)\subset \mathcal P_k^\perp \otimes \mathcal C^\infty(S)\ .
 \]
As $-\lambda_1 -1= \rho+k$, $Ker(\widetilde J_{-\lambda_1-1}) = \mathcal P_k^\perp$ by \eqref{rho+k}, $E_{\lambda_1,\lambda_2}=0$. 
\end{proof}
 The analog statement is valid for $\lambda_2 \in \rho+\mathbb N$ or $\lambda_2 \in -\rho-\mathbb N$.  From the knowledge of these zeroes of the operator-valued holomorphic function $(\lambda_1,\lambda_2 \longmapsto E_{\lambda_1,\lambda_2})$,  $iii)$ follows from Lemma \ref{0E} by routine arguments, ending with the use of Hartog's theorem in $\mathbb C^2$.
 
In the sequel,  let\footnote{The factor $\displaystyle\frac{16}{\pi^{2(n-1)}}$ is introduced to have a simple expression for the Bernstein-Sato polynomial  , see \eqref{explicitb}.}
$\widetilde E_{\lambda_1,\lambda_2} $ be defined by
 \begin{equation} 
 \widetilde E_{\lambda_1,\lambda_2} = \frac{16}{ \pi^{2(n-1)} }\,\Gamma(\lambda_1+\rho+1) \Gamma( -\lambda_1+\rho)\Gamma( \lambda_2+\rho+1)\Gamma(-\lambda_2+\rho)\,E_{\lambda_1,\lambda_2}\ .
 \end{equation}
\smallskip

\noindent
{\bf Remark.} In \cite{bc} Proposition 3.7, an explicit expression was given for (the analog of) the operator $\widetilde E_{\lambda_1,\lambda_2}$ in the non compact picture. It would be of some interest to find such an expression on $S\times S$.

\subsection{A Bernstein-Sato identity for the kernel $k_{\boldsymbol \alpha}$}

Let $\boldsymbol \alpha \in \mathbb C^3$, and let $\boldsymbol \lambda$ be its associated spectral parameter. Then the geometric parameter associated to $(\lambda_1+1,\lambda_2+1,\lambda_3)$ is $(\alpha_1,\alpha_2,\alpha_3+2)$.

\begin{proposition} Assume  $\boldsymbol \alpha$ is not a pole. Let $\boldsymbol \lambda$ be its associated spectral parameter. Then there exists a scalar $b(\alpha_1,\alpha_2,\alpha_3)$ such that, for all $f_1,f_2,f_3\in \mathcal C^\infty(S)$
\begin{equation}\label{b}
\mathcal K_{\alpha_1,\alpha_2,\alpha_3+2}(\widetilde E_{\lambda_1,\lambda_2} (f_1\otimes f_2), f_3) = b(\alpha_1,\alpha_2,\alpha_3) \mathcal K_{\alpha_1,\alpha_2,\alpha_3}(f_1, f_2,f_3)
\end{equation} 
\end{proposition}

\begin{proof} The left hand side of \eqref{b} is $\boldsymbol \lambda$-invariant by the covariance property of the operator $\widetilde E_{\lambda_1,\lambda_2}$. The proposition follows from the generic uniqueness theorem for invariant trilinear forms (see \cite{co} Theorem 3.1).
\end{proof}
Equation \eqref{b} is a disguised form of a \emph{Bernstein-Sato identity} for the kernel
$k_{\boldsymbol \alpha} (x,y,z)=\vert x-y\vert^{\alpha_3}\vert y-z\vert^{\alpha_1}\vert z-x\vert^{\alpha_2}$. More precisely, equation \eqref{b} implies for the distribution kernels of both sides
\begin{equation}\label{BS}
{\widetilde E}_{\lambda_1,\lambda_2}^t\big( k_{\alpha_1,\alpha_2,\alpha_3+2}\big)= b(\alpha_1,\alpha_2,\alpha_3)\, k_{\alpha_1,\alpha_2,\alpha_3}\ ,
\end{equation}
where ${\widetilde E}_{\lambda_1,\lambda_2}^t$ is the transpose of the differential operator ${\widetilde E}_{\lambda_1,\lambda_2}$, acting on the variables $x,y$.

The explicit computation of the function $b$ was done in \cite{bc} by brute force computation (with the help of a computer, thanks Ralf !), starting from \eqref{BS} and the explicit expression (in the non compact picture) of (the analog of) the operator $\widetilde E_{\lambda_1,\lambda_2}$. We propose here a different computation of the function $b$, based on the identity
\begin{equation}\label{KE1}
\mathcal K_{\alpha_1,\alpha_2,\alpha_3+2}(\widetilde E_{\lambda_1,\lambda_2} (1\otimes 1),1) = b(\alpha_1,\alpha_2,\alpha_3) \,\mathcal K_{\alpha_1,\alpha_2,\alpha_3}(1,1,1)
\end{equation}
First recall some {formul\ae}   which are needed for the computation of the left hand side of \eqref{KE1}.
\begin{equation}\label{sint1}
\int_S \vert {\bf 1}-x\vert^s dx= (2\sqrt \pi)^{n-1} 2^s \frac{ \Gamma( \frac{s}{2}+\rho)}{\Gamma(\frac{s}{2}+2\rho)}\ ,
\end{equation}
which implies
\begin{equation}\label{lambdaint1}
{\widetilde J}_\lambda\, (1) = \pi^\rho\frac{2^{2\lambda}}{\Gamma(\lambda+\rho)}\, 1\ .
\end{equation}
Needless to say, these {formul\ae}   (and more to come) have to be interpreted as equality of meromorphic functions of $s$ (or $\lambda$). They are usually proved via  analytic continuation. 

\begin{lemma} The following identities hold
\smallskip

$i)$ 
\begin{equation}\label{sint2}
 \int_S \vert {\bf 1}-x\vert^{s} x_1 dx=- (2\sqrt \pi)^{n-1} 2^{s}\, \frac{s}{2}\, \frac{ \Gamma( \frac{s}{2}+\rho)}
{\Gamma(\frac{ s}{2}+n)}
\end{equation}
\smallskip

$ii)$ for any $y,u\in S$, 
\begin{equation}\label{sint3}
\int_S \vert x-y\vert^s \langle u,x\rangle dx =  - (2\sqrt \pi)^{n-1} 2^{s}\, \frac{s}{2}\, \frac{ \Gamma( \frac{s}{2}+\rho)}
{\Gamma(\frac{ s}{2}+n)}\, \langle u,y\rangle\ .
\end{equation}

\end{lemma}

\begin{proof}  For $i)$, as $x_1=1-\frac{1}{2}\vert {\bf 1}-x\vert^2$
\[\int_S \vert {\bf 1}-x\vert^{s} x_1 dx = \int_S \vert {\bf 1}-x\vert^s dx-\frac{1}{2} \int_S \vert {\bf1}-x\vert^{s+2} dx
\]
\[= (2\sqrt{\pi})^{n-1}\Big( 2^s\frac{\Gamma(\frac{s}{2}+\rho)}{\Gamma( \frac{s}{2}+2\rho)} -\frac{1}{2} 2^{s+2} \frac{ \Gamma(\frac{s}{2}+\rho+1)}{\Gamma(\frac{s}{2}+2\rho+1)}\Big)
\]
\[= (2\sqrt{\pi})^{n-1}\, 2^s\, \frac{\Gamma( \frac{s}{2}+\rho)}{\Gamma( \frac{s}{2}+2\rho+1)}\,\big(-\frac{s}{2}\big)\ .\]

For $ii)$, let $\varphi(u,y)$ be the value of the left hand side. A basic result in harmonic analysis on $S$ is that an operator $K$ on $\mathcal C^\infty(S)$ given by $Kf(y) =\int_S k(\vert x-y\vert) f(x) dx$ is scalar on any space of homogeneous harmonic polynomials,  in particular on $\mathcal H_1$, the space of homogeneous polynomial of degree 1. Hence $\varphi(y) = c\langle u,y\rangle$  for some $c$ independent of $u$ and $y$. Let $u=y=\bf 1$ and compute $c$, using $i)$.
\end{proof}

\begin{lemma}{\ }
\smallskip

$i)$ for $y\in S$, 
\[\int_S \vert x-y\vert^{s} \vert {\bf 1}-x\vert^2 dx = (2\sqrt \pi)^{n-1}  \frac{ \Gamma( \frac{s}{2}+\rho)}
{\Gamma(\frac{s}{2}+n)}2^{s}\big( s y_1+s+2(n-1)\big)\ .
\]

$ii)$ for $y,z\in S$, 
\begin{equation}\label{asb2}
\int_S \vert x-y\vert^s \vert z-x\vert^2 dx = (2\sqrt \pi)^{n-1}  \frac{ \Gamma( \frac{s}{2}+\rho)}
{\Gamma(\frac{ s}{2}+n)}2^{s}\big(s \langle z,y\rangle+s+2(n-1)\big)
\end{equation}
\end{lemma}

\begin{proof} For $i)$ use the fact that $\vert {\bf 1}-x\vert^2  = 2(1-x_1)$ and combine \eqref{sint1} and \eqref{sint2}.
For $ii)$,
let $\Psi(y,z)$ be the value of the left hand side. Then, by a change of variables,
for any $k\in SO(n)$, $\Psi(ky,kz) = \Psi(y,z)$. Hence $\Psi(y,z) = \psi(\langle y,z\rangle)$. Now let $z=\bf 1$ and compute $\psi$ using $i)$.
\end{proof}

\begin{lemma} For $u,t\in S$ and $r,s$ in $\mathbb C^2$
\begin{equation}\label{srint}
\begin{split}
 &\int_S\int_S \vert u-x\vert^s \vert t-y\vert^r \vert x-y\vert^2 dxdy
=\\ &(4\pi)^{n-1} \frac{ \Gamma( \frac{s}{2}+\rho)  \Gamma( \frac{r}{2}+\rho)}{\Gamma(\frac{s}{2}+n)\Gamma(\frac{r}{2}+n)}2^{s+r} \Big( \frac{rs}{4} \vert u-t\vert^2 + (n-1)(r+s) +2(n-1)^2\Big)\ .
\end{split}
\end{equation}
\end{lemma}
\begin{proof}
Slightly rewrite \eqref{asb2} as
\[\int_S \vert x-y\vert^s \vert z-x\vert^2 dx = (2\sqrt \pi)^{n-1}  \frac{ \Gamma( \frac{s}{2}+\rho)}
{\Gamma(\frac{ s}{2}+n)}2^{s}\big( -\frac{s}{2} \vert z-y\vert^2 +2s +2(n-1)\big)\ .
\]
Now, by Fubini's theorem 
\[\int_S\int_S \vert u-x\vert^s \vert t-y\vert^r \vert x-y\vert^2 dxdy
 = \int_S \vert t-y\vert^r \Big(\int_S \vert u-x\vert^s \vert x-y\vert^2 dx\Big)\, dy
\]
\[ = (2\sqrt \pi)^{n-1}  \frac{ \Gamma( \frac{s}{2}+\rho)}
{\Gamma(\frac{ s}{2}+n)}2^{s}\Big((-\frac{s}{2})\int_S \vert t-y\vert^r\vert u-y\vert^2 dy + \big(2s +2(n-1)\big) \int_S \vert t-y\vert^r dy\Big)
\] 
\[=(4\pi)^{n-1} \frac{ \Gamma( \frac{s}{2}+\rho)  \Gamma( \frac{r}{2}+\rho)}{\Gamma(\frac{s}{2}+n)\Gamma(\frac{r}{2}+n)}2^{s+r}\Big(\frac{rs}{4} \vert u-t\vert^2 -rs-(n-1)+(2s+2(n-1))(\frac{r}{2} +(n-1))\big)
\]
\end{proof}

\begin{lemma}
\begin{equation}\label{E1otimes1}
\begin{split}
&{\widetilde E}_{\lambda_1,\lambda_2} (1\otimes 1) (u,t) =\\&(\lambda_1+\rho)(\lambda_2+\rho) \Big((\lambda_1+\rho+1)(\lambda_2+\rho+1)\vert u-t\vert^2 -2(n-1)(\lambda_1+\lambda_2+2)\Big)\ .
\end{split}
\end{equation}
\end{lemma}

\begin{proof} First using \eqref{lambdaint1}
\[{\widetilde J}_{\lambda_1}\otimes {\widetilde J}_{\lambda_2}(1\otimes 1) = \pi^{n-1} \frac{2^{2\lambda_1+2\lambda_2}}{\Gamma( \lambda_1+\rho)\Gamma(\lambda_2+\rho)}1\otimes 1\ .
\]
Now
\[{\widetilde J}_{-\lambda_1-1} \otimes {\widetilde J}_{-\lambda_2-1} (\vert x-y\vert^2) = \frac{1}{\Gamma(-\lambda_1-1)\Gamma(-\lambda_2-1)}\times  \dots\]
\[\iint_{S\times S}\vert u-x\vert^{-(n-1)-2\lambda_1-2} \vert t-y\vert^{-(n-1)-2\lambda_2-2} \vert x-y\vert^2 dx dy
\]
and, using  \eqref{srint}, $\iint_{S\times S} \dots$ is equal to
\[ (4\pi)^{n-1} \frac{\Gamma(-\lambda_1-1)\Gamma(-\lambda_2-1)}{\Gamma(-\lambda_1+\rho)\Gamma( -\lambda_2+\rho)}2^{-2(n-1)-2\lambda_1-2\lambda_2-4}\  \times\]
\[\Big((\lambda_1+\rho+1)(\lambda_2+\rho+1)\vert u-t\vert^2 -2(n-1)(\lambda_1+\lambda_2+2)\Big)
\]
Now put all things together and take into account the normalization factor for the definition of $\widetilde E_{\lambda_1,\lambda_2}$ to get the result.
\end{proof}

The rest of the computation will be done using the geometric parameter, so first convert \eqref{E1otimes1}  to get

\begin{equation}\label{E12alpha}
\begin{split}
&{\widetilde E}_{\lambda_1,\lambda_2}(1\otimes 1) =\big(\frac{\alpha_2+\alpha_3}{2} +(n-1)\big)\big( \frac{\alpha_1+\alpha_3}{2}+(n-1)\big)\\&\Big(\big(\frac{\alpha_2+\alpha_3}{2}+n\big)\big(\frac{\alpha_1+\alpha_3}{2}+n\big) \vert u-t\vert^2 -(n-1)(\alpha_1+\alpha_2+2\alpha_3+2n+2) \Big)
\end{split}
\end{equation}

\begin{proposition} 
\begin{equation}\label{explicitb}
b(\alpha_1,\alpha_2,\alpha_3) = 
\end{equation}
\[(\alpha_1+\alpha_2+\alpha_3+2(n-1))(\alpha_1+\alpha_2+\alpha_3+n+1) (\alpha_3+n-1)(\alpha_3+2)
\]
\end{proposition}
\begin{proof} Recall that
\[ \mathcal K_{\alpha_1,\alpha_2,\alpha_3+2}(\widetilde E_{\lambda_1,\lambda_2}(1\otimes 1), 1) = b(\alpha_1,\alpha_2,\alpha_3)
 \mathcal K_{\alpha_1,\alpha_2,\alpha_3}(1,1,1)\ .
\]
Now, using \eqref{E12alpha} and the evaluation of the Bernstein-Reznikov integral (see \cite{c} Proposition 2.1)
\[ \mathcal K_{\alpha_1,\alpha_2,\alpha_3+2}(\widetilde E_{\lambda_1,\lambda_2}(1\otimes 1), 1) =\big(\frac{\alpha_2+\alpha_3}{2} +(n-1)\big)\big( \frac{\alpha_1+\alpha_3}{2}+(n-1)\big) 
\]
\[\Big\{\big(\frac{\alpha_2+\alpha_3}{2}+n\big)\big(\frac{\alpha_1+\alpha_3}{2}+n\big)\mathcal K_{\alpha_1,\alpha_2,\alpha_3+4}(1,1,1)\]\[ -(n-1)(\alpha_1+\alpha_2+2\alpha_3+2n+2) \mathcal K_{\alpha_1,\alpha_2,\alpha_3+2}(1,1,1)\Big\}
\]
\[= \big(\frac{\pi}{2}\big)^{\frac{3}{2}(n-1)} 2^{\alpha_1+\alpha_2+\alpha_3}\,4\,\frac{ \Gamma( \frac{\alpha_1+\alpha_2+\alpha_3}{2}+2\rho+1) \Gamma( \frac{\alpha_1}{2}+\rho)\Gamma( \frac{\alpha_2}{2}+\rho)
\Gamma( \frac{\alpha_3}{2}+\rho+1)}{\Gamma(\frac{\alpha_1+\alpha_2}{2}+2\rho)\Gamma(\frac{\alpha_2+\alpha_3}{2}+2\rho)\Gamma(\frac{\alpha_3+\alpha_1}{2}+2\rho)}\times
\]
\[\Big\{\big(\alpha_1+\alpha_2+\alpha_3+2n\big)\big(\alpha_3+n+1)-(n-1)\big(\alpha_1+\alpha_2+2\alpha_3+2n+2\big)\Big\}\ .
\]
Now 
\[\big(\frac{\pi}{2}\big)^{\frac{3}{2}(n-1)} 2^{\alpha_1+\alpha_2+\alpha_3}\,4\,\frac{ \Gamma( \frac{\alpha_1+\alpha_2+\alpha_3}{2}+2\rho+1) \Gamma( \frac{\alpha_1}{2}+\rho)\Gamma( \frac{\alpha_2}{2}+\rho)
\Gamma( \frac{\alpha_3}{2}+\rho+1)}{\Gamma(\frac{\alpha_1+\alpha_2}{2}+2\rho)\Gamma(\frac{\alpha_2+\alpha_3}{2}+2\rho)\Gamma(\frac{\alpha_3+\alpha_1}{2}+2\rho)}
\]
\[= 4\big(\frac{\alpha_1+\alpha_2+\alpha_3}{2}+(n-1)\big)\,\big(\frac{\alpha_3}{2}+\rho\big)\, \mathcal K_{\alpha_1,\alpha_2,\alpha_3}(1,1,1)\ ,
\]
and 
\[\big(\alpha_1+\alpha_2+\alpha_3+2n\big)\big(\alpha_3+n+1)-(n-1)\big(\alpha_1+\alpha_2+2\alpha_3+2n+2\big)
\]
\[=(\alpha_1+\alpha_2+\alpha_3 +n+1)(\alpha_3+2)\ .
\]
\end{proof}
A variant of \eqref{b} is 
\begin{equation}\label{BStilde}
\begin{split}
&\widetilde {\mathcal K}_{\alpha_1,\,\alpha_2,\,\alpha_3+2} \big(\widetilde E_{\lambda_1,\lambda_2}(f_1\otimes f_2), f_3\big)=\\ 4 (\alpha_1+\alpha_2+&\alpha_3+n+1)(\alpha_3+n-1)\, \widetilde {\mathcal K}_{\alpha_1,\,\alpha_2,\,\alpha_3}(f_1,f_2,f_3)
\end{split}
\end{equation}

\subsection{The family $\mathcal S^{(k)}_{\lambda_1,\lambda_2}$}

From Proposition \ref{constrE}, it is easy to construct covariant bi-differential operators. 

Denote by $res : \mathcal C^\infty(S\times S) \longrightarrow \mathcal C^\infty(S)$ the restriction map from $S\times S$ to $S$, defined for $f\in \mathcal C^\infty(S\times S)$ by 
\[f\longmapsto res(f),\qquad res\, (f)(x) = f(x,x)\ .
\]
\begin{lemma}\label{covres}
The map $res$ intertwines $\pi_{\mu_1}\otimes \pi_{\mu_2}$ with $\pi_{\mu_1+\mu_2+\rho}$.
\end{lemma}

For $\lambda_1,\lambda_2 \in \mathbb C$ define by induction on $k$ the differential operators $ \widetilde E^{(k)}_{\lambda_1,\lambda_2} $ on $S\times S$ by 
\[\widetilde E^{(0)}_{\lambda_1,\lambda_2} = Id,\qquad  \widetilde E^{(k)}_{\lambda_1,\lambda_2}= \widetilde E_{\lambda_1+k-1,\lambda_2+k-1} \circ \dots \circ \widetilde E_{\lambda_1,\lambda_2}
\]
and introduce the bi-differential operators $D_{\lambda_1,\lambda_2}^{(k)} : \mathcal C^\infty(S\times S)\longrightarrow \mathcal C^\infty(S)$ given by
\[D_{\lambda_1,\lambda_2}^{(0)}= res,\qquad D_{\lambda_1,\lambda_2}^{(k)}= res\circ \widetilde E^{(k)}_{\lambda_1,\lambda_2} \]

\begin{proposition} Let $\lambda_1,\lambda_2\in \mathbb C $ and $k\in \mathbb N$. Then
\[D^{(k)}_{\lambda_1,\lambda_2} \circ (\pi_{\lambda_1}\otimes \pi_{\lambda_2})(g) = \pi_{\lambda_1+\lambda_2+\rho+2k}(g)\circ D^{(k)}_{\lambda_1,\lambda_2}\ .
\]
\end{proposition}
This is a consequence of the covariance property of  $E_{\mu_1,\mu_2}$ (Proposition \ref{constrE} $i$) and Lemma \ref{covres}.

For $f_1,f_2,f_3\in \mathcal C^\infty(S)$, let
\[\mathcal S^{(k)}_{\lambda_1,\lambda_2} (f_1,f_2,f_3) = \int_S D^{(k)}_{\lambda_1,\lambda_2} (f_1\otimes f_2)(x)f_3(x) dx\ .
\]
This defines a distribution on $\mathcal C^\infty(S\times S\times S)$, supported on the diagonal $\mathcal O_4$.\footnote{To give an explicit expression of a distribution supported on $\mathcal O_4$, a choice of a transverse submanifold (and associated transverse differential operators) at each point of $\mathcal O_4$ is needed. Our choice leads to use transverse differential operators acting on the variables $(x,y)$.}   

\begin{proposition} Let $k\in \mathbb N$ and let $\boldsymbol \lambda=(\lambda_1,\lambda_2,\lambda_3)$ satisfy 
\[\lambda_1+\lambda_2+\lambda_3=-\rho-2k\ .
\]
 The trilinear form $\mathcal S^{(k)}_{\lambda_1,\lambda_2}$ belongs to $Tri(\boldsymbol \lambda, diag)$.
\end{proposition}

\begin{proof}
As $\lambda_3=-(\lambda_1+\lambda_2+\rho+2k)$, the $\boldsymbol \lambda$-invariance is a consequence of the covariance property of $D^{(k)}_{\lambda_1,\lambda_2}$ and the duality between $\pi_\lambda$ and $\pi_{-\lambda}$.
\end{proof}
The next statement is easily obtained from \eqref{BStilde} by induction  on $k$.
\begin{proposition}\label{bktilde} Let $\boldsymbol \alpha\in \mathbb C^3$. Then
\[\widetilde {\mathcal K}_{\alpha_1,\alpha_2,\alpha_3+2k}( \big(E^{(k)}_{\lambda_1, \lambda_2}(f_1\otimes f_2), f_3\big)=\]\[4^k\prod_{l=1}^{k} \big(\alpha_1+\alpha_2+\alpha_3+n-1+2l\big)\big(\alpha_3+2l)\big)\widetilde {\mathcal K}_{\alpha_1,\alpha_2,\alpha_3}(f_1,f_2,f_3)\ .
\]
\end{proposition}

 As said in the introduction, the family $\mathcal S^{(k)}_{\lambda_1,\lambda_2}$ was introduced in \cite{bc} because of its relation with the residues of the meromorphic family $\mathcal K^{\boldsymbol \lambda}$ at poles of type II. The presentation follows closely \cite{bc}, with changes due to the choice of the  compact picture instead of the non compact picture.
 
\begin{lemma}\label{lemmaK0intS}
 Let $\boldsymbol \alpha$ satisfy
$\alpha_1+\alpha_2+\alpha_3 = -2(n-1)$.
Then, for any $f\in \mathcal C^\infty(S\times S\times S)$
\begin{equation}\label{K0intS}
\widetilde {\mathcal K}_{\boldsymbol \alpha}(f)= \big(\frac{\pi}{16}\big)^{n-1} \frac{(n-2)!}{\Gamma(\rho)}\,\frac{1}{\Gamma(-\frac{\alpha_1}{2})\Gamma(-\frac{\alpha_2}{2})\Gamma(-\frac{\alpha_3}{2})}\, \int_S f(x,x,x)\, dx\ .
\end{equation}
\end{lemma}

\begin{proof} Let $\boldsymbol \lambda$ be the associated spectral parameter. As easily verified by a change of variable, the trilinear form \[(f_1,f_2,f_3) \longmapsto \int_S f_1(x)f_2(x)f_3(x)\, dx
\]
defines a (non trivial) $\boldsymbol \lambda$-invariant distribution on $S\times S\times S$. For $\boldsymbol \alpha\notin Z$ (i.e. $\boldsymbol \alpha$ in a dense open subset of the plane $\beta_1+\beta_2+\beta_3 = -2(n-1)$) the uniqueness theorem (see \cite{c}) for $\boldsymbol \lambda$-invariant trilinear forms applies and hence $\widetilde {\mathcal K}_{\boldsymbol \alpha}$ has to be a multiple of this distribution. To determine the value of the scalar,  evaluate both sides of  \eqref{K0intS} on the function $f=1$. By analytic continuation, the equality extends to all $\boldsymbol \alpha$'s in the plane $\alpha_1+\alpha_2+\alpha_3 =-2(n-1)$.
\end{proof}

\begin{proposition} Let $\boldsymbol \alpha\in \mathbb C^3$ such that $\alpha_1+\alpha_2+\alpha_3=-2(n-1)-2k$ for $k\in \mathbb N$. Then
\begin{equation}\label{KS}
\widetilde {\mathcal K}_{\boldsymbol \alpha} = \big(\frac{\pi}{16}\big)^{n-1} \frac{(n-2)!}{\Gamma(\rho+k)}16^{-k}\frac{1}{\Gamma(-\frac{\alpha_1}{2})\Gamma(-\frac{\alpha_2}{2})\Gamma(-\frac{\alpha_3}{2})}\,S^{(k)}_{\lambda_1,\lambda_2}
\end{equation}
\end{proposition}
\begin{proof} Observe that $(\alpha_1,\alpha_2,\alpha_3+2k)$ satisfies the assumption of Lemma \ref{lemmaK0intS}, so that 
\[\widetilde {\mathcal K}_{\alpha_1,\alpha_2,\alpha_3+2k}({\widetilde E}_{\lambda_1,\lambda_2}^{(k)} (f_1\otimes f_2),f_3)= \]\[ \big(\frac{\pi}{16}\big)^{n-1} \frac{(n-2)!}{\Gamma(\rho)}\,\frac{1}{\Gamma(-\frac{\alpha_1}{2})\Gamma(-\frac{\alpha_2}{2})\Gamma(-\frac{\alpha_3}{2}-k)}\, \int_S D_{\lambda_1,\lambda_2}^{(k)} (f_1\otimes f_2)(x)f_3(x)\, dx
\]
\[= \big(\frac{\pi}{16}\big)^{n-1} \frac{(n-2)!}{\Gamma(\rho)}\,\frac{1}{\Gamma(-\frac{\alpha_1}{2})\Gamma(-\frac{\alpha_2}{2})\Gamma(-\frac{\alpha_3}{2}-k)}\, \mathcal S^{(k)}_{\lambda_1,\lambda_2}(f_1,f_2,f_3)\ .
\]
Now, by Proposition \ref{bktilde}, the same quantity is equal to
\[4^k\prod_{l=1}^{k} \big(\alpha_1+\alpha_2+\alpha_3+n-1+2l\big)\big(\alpha_3+2l)\big)\,\widetilde {\mathcal K}_{\alpha_1,\alpha_2,\alpha_3}(f_1,f_2,f_3)\ .
\]
Observe that
\[\prod_{l=1}^{k} \big(\alpha_1+\alpha_2+\alpha_3+n-1+2l\big)= \prod_{j=0}^{k-1} (-(n-1)-2j)=(-1)^k2^k \prod_{j=0}^{k-1}(\rho+j) = (-1)^k 2^k (\rho)_k,
\]
\[\prod_{l=1}^{k} (\alpha_3 +2l)\, \Gamma(-\frac{\alpha}{2}-k) = (-1)^k 2^k \prod_{l=1}^k (-\frac{\alpha_3}{2}-l) \Gamma(-\frac{\alpha_3}{2}-k)
=(-1)^k 2^k \Gamma(-\frac{\alpha_3}{2}),
\]
and the result follows.
\end{proof}
\begin{proposition}\label{valueS} Let $k\in \mathbb N$ and let $\boldsymbol \lambda$ satisfying 
$\lambda_1+\lambda_2+\lambda_3= -\rho-2k$.
Then
 \smallskip
 
 \noindent
$\text {if } a_1+a_2+a_3 >k, \quad \mathcal S_{\lambda_1,\lambda_2}^{(k)}(p_{a_1,\,a_2,\,a_3}) 
= 0, \qquad \ ,$
\smallskip

\noindent
if $a_1+a_2+a_3\leq k,$
\begin{equation}\label{evSp}
\begin{split}
&\mathcal S_{\lambda_1,\lambda_2}^{(k)}(p_{a_1,\,a_2,\,a_3})  = (2\sqrt{2\pi})^{n-1}\, 4^k\,\frac{\Gamma(\rho+k)}{\Gamma(2\rho)}\,2^{2(a_1+a_2+a_3)}\\
&(-k)_{a_1+a_2+a_3} \,
(-k-\lambda_1)_{a_1}(-k-\lambda_2)_{a_2}(-k-\lambda_3)_{a_3}
\\
&(\rho+\lambda_1 +a_2+a_3)_{k-a_2-a_3}(\rho+\lambda_2 +a_3+a_1)_{k-a_3-a_1}(\rho+\lambda_3 +a_1+a_2)_{k-a_1-a_2}
\end{split}
\end{equation}
.
\end{proposition}
\begin{proof} Assume first that $\boldsymbol \lambda$ is a generic pole of type II. Then combine \eqref{KS}  and  the evaluation of $\widetilde {\mathcal K}^{\boldsymbol \lambda}(p_{a_1,a_2,a_3})$ (see \eqref{Kplambda}). Observe then that the two hand sides of \eqref{evSp} are holomorphic and conclude by the analytic continuation property.
\end{proof}

\section{The holomorphic families $\mathcal R^{(j,m)}_{\,.\,,\,.\,}$}

In this section, define for $\alpha\in \mathbb C$ the operator $\widetilde I_{\alpha}$ by 
\[
\widetilde I_{\alpha}f(x) = \frac{1}{\Gamma( \frac{\alpha}{2} +\rho)} \int_S \vert x-y\vert^\alpha f(y) dy\ .
\] 
where $f\in \mathcal C^\infty(S)$. Of course, they are nothing but the Knapp-Stein intertwining operators, but with a different parametrization.

For $\alpha_1,\alpha_2\in \mathbb C$, the formula
\[\mathcal R^{(3)}_{\alpha_1,\alpha_2}(f_1,f_2,f_3) = \int_S \widetilde I_{\alpha_2} f_1(z)\, \widetilde I_{\alpha_1} f_2(z)\, f_3(z) \,dz\ .
\]
defines a continuous trilinear form on $\mathcal C^\infty(S)\times \mathcal C^\infty(S)\times \mathcal C^\infty(S)$, which as usual is regarded as a distribution on $S\times S\times S$.

Let  $l\in \mathbb N$. Then the function $(x,y,z) \longmapsto\vert x-y\vert^{2l}$ is smooth on $S\times S\times S$. So, define the distribution $\mathcal R^{(3, l)}_{\alpha_1,\alpha_2}$ on $S\times S\times S$ as the product of $\mathcal R^{(3)}_{\alpha_1,\alpha_2}$ with $\vert x-y\vert^{2l}$, or more explicitly, 
\[ \mathcal R^{(3, l)}_{\alpha_1,\alpha_2}(f) =\int_{S\times S\times S} \mathcal R^{(3)}_{\alpha_1,\alpha_2}(x,y,z)\,\big(\vert x-y\vert^{2l}f(x,y,z)\big)\, dx\, dy \,dz\ .
\]
\begin{proposition}
Let $l\in \mathbb N$. For $\alpha_1,\alpha_2\in \mathbb C$, let $\boldsymbol \lambda$ be the spectral parameter associated to $(\alpha_1,\alpha_2,2l)$. 
\smallskip

$i)$ 
$ \mathcal R^{(3,l)}_{\alpha_1,\alpha_2}$ is $\boldsymbol \lambda$-invariant
\smallskip

$ii)$ $ \mathcal R^{(3,l)}_{\alpha_1,\alpha_2}$ depends holomorphically on $(\alpha_1,\alpha_2)$
\smallskip

$iii)$ 
\begin{equation}\label{RK}
\widetilde {\mathcal K}_{\alpha_1,\alpha_2, 2l} = \frac{1}{\Gamma(l+\rho)}\,\frac{1}{\Gamma( \frac{\alpha_1+\alpha_2}{2}+l+2\rho)}\,\mathcal R^{(3,l)}_{\alpha_1,\alpha_2}\ .
\end{equation}
\end{proposition}

\begin{proof} For $\Re \alpha_1, \Re \alpha_2$ large enough, 

\[
\mathcal R^{(3,l)}_{\alpha_1,\alpha_2}(f) = \frac{1}{\Gamma( \frac{\alpha_1}{2}+\rho)}\,\frac{1}{\Gamma( \frac{\alpha_2}{2}+\rho)}\,\int_{S\times S\times S}\!\!\!\! \!\!\!\!\vert y-z\vert^{\alpha_1}\vert z-x\vert^{\alpha_2}\, \vert x-y\vert^{2l}f(x,y,z) \,dx\,dy\,dz
\] 
\[ = \frac{1}{\Gamma( \frac{\alpha_1}{2}+\rho)}\,\frac{1}{\Gamma( \frac{\alpha_2}{2}+\rho)} \mathcal K_{\alpha_1,\alpha_2, 2l}(f)\ .
\]
Hence the distribution $ \mathcal R^{(3,l)}_{\alpha_1,\alpha_2}$ is $\boldsymbol \lambda$-invariant and satisfies \eqref{RK}. The rest follows by analytic continuation.
\end{proof}

\begin{proposition}
Let $a_1,a_2,a_3\in \mathbb N$. Then
\begin{equation}\label{evRp}
\begin{split} \mathcal R^{(3,l)}_{\alpha_1,\alpha_2}(p_{a_1,a_2,a_3})=(\frac{\pi}{2})^{\frac{3}{2} (n-1)} 2^{\alpha_1+\alpha_2+2l}2^{2(a_1+a_2+a_3)} \\
\times \quad \frac{(\frac{\alpha_1}{2}+\rho)_{a_1}\ (\frac{\alpha_2}{2}+\rho)_{a_2}\big(\frac{\alpha_1+\alpha_2}{2}+2\rho+a_1+a_2\big)_{l+a_3}}{\Gamma(\frac{\alpha_1}{2}+2\rho+l+a_1+a_3)\,\Gamma(\frac{\alpha_2}{2}+2\rho+l+a_2+a_3)} \ .
\end{split}
\end{equation}
\end{proposition}

Use the evaluation of the $K$-coefficients for $\widetilde {\mathcal K}_{\alpha_1,\alpha_2,2l}$ (see \eqref{Kpalpha} in Appendix) and the relation  \eqref{RK}.

By a cyclic permutation of the indices $\{ 1,2,3\}$, a similar construction holds for the families  for $\mathcal R^{(1,l)}_{\alpha_2,\alpha_3}$  and $\mathcal R^{(2,l)}_{\alpha_1,\alpha_3}$.
\section{ The multiplicity 2 theorem for $\boldsymbol \lambda$ in $Z_{1,II}$}

In this section, $\boldsymbol \lambda= (\lambda_1,\lambda_2,\lambda_3)$ is  assumed to be in $Z_{1,II}$. So, first $\boldsymbol \lambda$ belongs to some line of type II, which, up to a permutation of $\{1,2,3\}$ can be chosen to be $D^{+,1}_{l,m}$ given by the equations
\begin{equation*}
\lambda_1=-\rho-l,\quad \lambda_2+\lambda_3 = m\\
\end{equation*}
where $l\in \mathbb N, m\in \mathbb Z, m\equiv l \mod 2, \vert m\vert\leq l\ ,$.

For convenience, introduce $k=\frac{l-m}{2}$  ($k\in \mathbb N, k\leq l$)  so that  the equations become
\begin{equation}\label{D+}
 \lambda_1 = -\rho-l,\quad \lambda_1+\lambda_2+\lambda_3 = -\rho-2k\ .
\end{equation}
 Next the condition for $\boldsymbol \lambda \in Z_{1,II}$ (that is $\boldsymbol \lambda\notin Z_2\cup Z_3$) reads
\begin{equation}\label{D+Z}
\lambda_2 \notin \{ -k,-k+1,\dots, l-k\}\cup\big(-\rho-k-\mathbb N\big) \cup \big(\rho+(l-k)+\mathbb N\big)\  .
\end{equation}
as a consequence of Proposition \ref{Z1IIlambda}.

\begin{proposition}\label{nonzeroS}
 Let $\boldsymbol \lambda$ satisfy \eqref{D+} and \eqref{D+Z}. Then $\mathcal S^{(k)}_{\lambda_2,\lambda_3}\neq 0$.
\end{proposition}
\begin{proof} Formula \eqref{evSp}, applied for $a_1=a_2=a_3=0$ shows that 
\[\mathcal S^{(k)}_{\lambda_2,\lambda_3}(1) = NVT\, (-l)_k \,(\rho+\lambda_2)_k\,(\rho+\lambda_3)_k\ .
\]
 Now $l\geq k$ implies $(-l)_k\neq 0$ and \eqref{D+Z} implies $\lambda_2+\rho\notin -k-\mathbb N$, so that $(\rho+\lambda_2)_k\neq 0$, and similarly for $(\lambda_3+\rho)_k$. Hence
$\mathcal S^{(k)}_{\lambda_2,\lambda_3}(1) \neq 0$.
\end{proof}

Let $\alpha$ be the associated geometric parameter, so that
\begin{equation}\label{Z1II}
\begin{split}
&\hskip3cm \alpha_1+\alpha_2+\alpha_3 =-2(n-1)-2k,\quad  \alpha_1 = 2(l-k) \\
&\alpha_2 \notin \Big\{-(n-1),-(n-1)-2,\dots,-(n-1)-2l\Big\}\cup 2\mathbb N \cup-2(n-1)-2l-2\mathbb N
\end{split}
\end{equation}
It will be convenient to let $p=l-k$. Notice that $p\in \mathbb N$.
\subsection{The general case}

In this subsection, we assume that $\boldsymbol \alpha$ is a generic pole of type II, i.e. we assume moreover that $\alpha_2,\alpha_3\notin -(n-1)-2\mathbb N$.
\begin{lemma}\label{nonzeroR}
 Let $a_1,a_2,a_3$ satisfy
$a_2+a_3 >k+p, a_1+a_2>k, a_1+a_3>k$.
Then $\mathcal R^{(1,p)}_{\alpha_2,\alpha_3} (p_{a_1,a_2,a_3}) \neq 0$.
\end{lemma}

\begin{proof}  \eqref{evRp} yields
\begin{equation*}
\begin{split}
\mathcal R^{(1,p)}_{\alpha_2,\alpha_3} (p_{a_1,a_2,a_3})= C\,(\frac{\alpha_2}{2}+\rho)_{a_2}\, (\frac{\alpha_3}{2}+\rho)_{a_3}\, (-k-p+a_2+a_3)_{p+a_1}\\
\times \frac{1}{\Gamma(\frac{\alpha_2}{2}+(n-1)+p+a_1+a_3)}\, \frac{1}{\Gamma(\frac{\alpha_3}{2}+(n-1)+p+a_1+a_2)} 
\end{split}
\end{equation*}
First observe that $(\frac{\alpha_2}{2}+\rho)_{a_2}(\frac{\alpha_3}{2}+\rho)_{a_3}\neq 0$ as $\alpha_2,\alpha_3\notin -(n-1)-2\mathbb N$.  The assumption $a_2+a_3>k+p$ implies that $(-k-p+a_2+a_3)_{p+a_1}\neq 0$. Then $\frac{\alpha_2}{2} +2\rho+p\notin -k-\mathbb N$, hence  $a_1+a_3>k$ implies $\frac{1}{\Gamma(\frac{\alpha_2}{2}+(n-1)+p+a_1+a_3)} \neq 0$. Similarly,  $a_1+a_2>k$ implies $\frac{1}{\Gamma(\frac{\alpha_3}{2}+(n-1)+p+a_1+a_2)} \neq 0$. The statement follows.
\end{proof}
\begin{proposition} Let $\boldsymbol\alpha$  be a generic pole of type II which satisfies \eqref{Z1II} and let $\boldsymbol \lambda$ be its associated spectral parameter. Then $\mathcal S^{(k)}_{\lambda_2,\lambda_3}$ and $\mathcal R^{(1,p)}_{\alpha_2,\alpha_3}$ are two linearly independent $\boldsymbol \lambda$-invariant trilinear forms.
\end{proposition}

\begin{proof} By Proposition \ref{nonzeroS} $\mathcal S^{(k)}_{\lambda_2,\lambda_3}\neq 0$, and hence $Supp\,(\mathcal S^{(k)}_{\lambda_2,\lambda_3}) = \mathcal O_4$. 

Next  $Supp\, (\mathcal R^{(1,p)}_{\alpha_2,\alpha_3}) = S\times S\times S$. In fact, assume that $Supp\, (\mathcal R^{(1,p)}_{\alpha_2,\alpha_3})\subset \mathcal O_1\cup \mathcal O_2\cup \mathcal O_3\cup \mathcal O_4$. Given an arbitrary large number $L$, by choosing $a_1,a_2,a_3$ large enough, the function $p_{a_1,a_2,a_3}$ vanishes together with its partial derivatives up to order $L$ on $\mathcal O_1\cup \mathcal O_2\cup \mathcal O_3 \cup \mathcal O_4$. But by Lemma \ref{nonzeroR} $\mathcal R^{(1,p)}_{\alpha_2,\alpha_3} (p_{a_1,a_2,a_3}) \neq 0$ for arbitrary large $a_1,a_2,a_3$. Hence $Supp\, (\mathcal R^{(1,p)}_{\alpha_2,\alpha_3})\subset \mathcal O_1\cup \mathcal O_2\cup \mathcal O_3\cup \mathcal O_4$ is impossible. So $Supp\, (\mathcal R^{(1,p)}_{\alpha_2,\alpha_3}) = S\times S\times S$.

Having two different supports,  the two distributions are linearly independent.
\end{proof}

\begin{theorem}\label{mult2Z1IIgen}
 Let $\boldsymbol \alpha$ be a generic pole of type II and assume that $\boldsymbol \alpha\in Z_{1,II}$. Then $\dim Tri(\boldsymbol \lambda) = 2$. More precisely, if $\boldsymbol \alpha$ satisfies \eqref{Z1II}, then
\[Tri(\boldsymbol \lambda) = \mathbb C \,\mathcal R^{(1,p)}_{\alpha_2,\alpha_3}\oplus \mathbb C\, \mathcal S^{(k)}_{\lambda_2,\lambda_3}\ .
\]
\end{theorem}

\begin{proof}
Let $T$ be a $\boldsymbol \lambda$-invariant distribution. Consider its restriction to $\mathcal O_0$. As $Supp(\mathcal R^{(1,p)}_{\alpha_2,\alpha_3}) = S\times S\times S$, it follows that $\mathcal R^{(1,p)}_{\alpha_2,\alpha_3\,\vert\mathcal O_0}$ is $\neq 0$, and hence
$ T_{\vert \mathcal O_0} = C\,{\mathcal  R^{(1,p)}_{\alpha_2,\alpha_3}}_{\vert \mathcal O_0} $ for some constant $C$. Let $U = T-C\, \mathcal  R^{(1,p)}_{\alpha_2,\alpha_3}$. Then $U$ is a $\boldsymbol \lambda$-invariant distribution, which vanishes on $\mathcal O_0$. Now, as none of the $\alpha_j$'s belong to $-(n-1)-2\mathbb N$, Lemma 4.1 in \cite{c} shows that $U$ is supported on $\mathcal O_4$. Hence $U$ and $S^{1,k}_{\lambda_2,\lambda_3}$ both belong to $Tri(\boldsymbol \lambda, diag)$. In turn, the dimension of this space is equal to $dim\, {\it Sol}(\lambda_1,\lambda_2;k)$. In the present situation, notice that $\lambda_1 =2p\notin \{-1,-2,\dots -k\} \cup \{-\rho-\rho-1,\dots, -\rho-(k-1)\}$ and by Lemma A\ref{Sgen}, the dimension of the space of solutions is $\leq 1$. Hence, $dim\, {\it Sol}(\lambda_1,\lambda_2;k)=1$ and $U$ is a multiple of $\mathcal S^{(k)}_{\lambda_1,\lambda_2}$, and the conclusion follows.
\end{proof}
\subsection {The special case}

Let $\boldsymbol \alpha$ be a pole of type I+II in $Z_{1,II}$. Up to a permutation of the indices $1,2,3$, this amounts to the following assumptions

\begin{equation}\label{Z1IIspe}
\begin{split}
\alpha_1 =2p,\qquad \alpha_2 = -(n-1)-2k_2,\qquad \alpha_3 = -(n-1)+2q\\ p,k_2,q\in \mathbb N,\qquad  q\geq 1,\qquad k_2-p-q \geq 0\  .\hskip 1.5cm
\end{split}
\end{equation}
Notice that $\alpha_1+\alpha_2+\alpha_3 = -2(n-1)-2k$ where $k=k_2-p-q$. 
\begin{lemma}\label{nonzeroRspe}
 Let $\boldsymbol \alpha$ satisfy \eqref{Z1IIspe}. Then $Supp(\mathcal R^{1,p}_{\alpha_2,\alpha_3}) = \overline{\mathcal O_2}$. 
\end{lemma}

\begin{proof} Let $f_1,f_2,f_3\in \mathcal C^\infty(S)$ and assume that $Supp(f_1)\cap Supp(f_3)=\emptyset$. Then, using notation of section 13, as $\alpha_2\in -(n-1)-2\mathbb N$, $\widetilde I_{\alpha_2}$ is a differential operator, and hence $Supp(\widetilde I_{\alpha_2}f_1)\subset Supp(f_1)$. Hence $\mathcal R^{(1)}_{\alpha_2,\alpha_3}(f_1,f_2,f_3)=0$. This implies $Supp\,({\mathcal R}^{(1)}_{\alpha_2,\alpha_3})\subset \overline {\mathcal O_2}$, and {\it a fortiori} 
$Supp\,({\mathcal R}^{(1,p)}_{\alpha_2,\alpha_3})\subset \overline {\mathcal O_2}$. Next, a careful examination of \eqref{evRp}, for $a_1\leq k_1$ and $a_2,a_3$ large yields
 ${\mathcal R}^{(1,p)}_{\alpha_2,\alpha_3}(p_{a_1,a_2,a_3}) \neq 0 $. Hence $Supp({\mathcal R}^{(1,p)}_{\alpha_2,\alpha_3})\not \subset \mathcal O_4$. As ${\mathcal R}^{(1,p)}_{\alpha_2,\alpha_3})$ is $\boldsymbol \lambda$-invariant, its support is invariant under $G$, and hence $Supp\,({\mathcal R}^{(1,p)}_{\alpha_2,\alpha_3}) = \overline{\mathcal O_2}$.
\end{proof}

\begin{proposition}
 Let $\boldsymbol \alpha$ satisfy \eqref{Z1IIspe}. Then $\mathcal S^{(k)}_{\lambda_2,\lambda_3}$ and $\mathcal R^{(1,p)}_{\alpha_2,\alpha_3}$ are two linearly independent $\boldsymbol \lambda$-invariant trilinear forms.
\end{proposition}
\begin{proof} By Proposition \ref{nonzeroS}, $\mathcal S^{(k)}_{\lambda_2,\lambda_3}\neq 0$ and $Supp(\mathcal S^{(k)}_{\lambda_2,\lambda_3})= \mathcal O_4$. By Proposition \ref{nonzeroRspe}, $Supp(\mathcal R^{1,p}_{\alpha_2,\alpha_3})= \overline {\mathcal O_2}$. Hence the two distributions are linearly independent.
\end{proof}
\begin{theorem}\label{mult2Z1IIspe}
 Let $\boldsymbol \alpha$ be a pole of type I+II which belongs to $Z_{1,II}$. Then $\dim Tri(\boldsymbol \lambda) = 2$. More precisely, assume $\boldsymbol \alpha$ satisfies \eqref{Z1IIspe}. Then
\[Tri(\boldsymbol \lambda) = \mathbb C \mathcal S^{(k)}_{\lambda_2,\lambda_3}\oplus \mathbb C \mathcal R^{(1,p)}_{\alpha_2,\alpha_3}\  .
\]
\end{theorem}

To prepare for the proof,  the following result,is needed.

\begin{proposition}\label{noextZ1IIspe}
 The distribution $\mathcal K_{\boldsymbol \alpha, \mathcal O_0}$ cannot be extended as a $\boldsymbol \lambda$-invariant distribution on $S\times S\times S$.
\end{proposition}
\begin{proof} For $s\in \mathbb C$, let $\boldsymbol \alpha(s) = (\alpha_1, \alpha_2+2s, \alpha_3)$. For $\vert s \vert$ small, let
\[\mathcal F(s) = \frac{(-1)^k}{k!}\frac{(-1)^{k_2}}{k_2!}\, \Gamma( \frac{\alpha_1}{2}+\rho)\,\Gamma( \frac{\alpha_3}{2}+\rho)\,\frac{1}{s}\,\widetilde {\mathcal K}_{\boldsymbol \alpha(s)}\ .
\]
This defines a holomorphic function in a neighborhood of $0$, and the Taylor expansion at $0$ reads
\[\mathcal F(s) = F_0+sF_1+O(s^2)
\]
where $F_0$ and $F_1$ are distributions on $S\times S\times S$.
\begin{lemma} The distributions $F_0$ and $F_1$ satisfy
\smallskip

$i)$ $Supp(F_0) = \overline{\mathcal O_2}$
\smallskip

$ii)$ the restriction of $F_1$ to $\mathcal O_0$ coincides with $\mathcal K_{\boldsymbol \alpha, \mathcal O_0}$
\end{lemma}

\begin{proof} 

$i)$ As $\alpha_1(s) = \alpha_1 = 2p$, $\boldsymbol \alpha(s)$ belongs to the plane $\beta_1 = 2p$, and hence, by \eqref{RK} 
\[\widetilde {\mathcal K}_{2p, \alpha_2+2s, \alpha_3} = \frac{1}{\Gamma(p+\rho)}\frac{1}{\Gamma(-k+s)} \, \mathcal R^{(1,p)}_{\alpha_2+2s,\alpha_3}\ .
\]
Differentiate at $s=0$ to get
\[\big(\frac{d}{ds} \widetilde {\mathcal K}_{2p, \alpha_2+s, \alpha_3}\big) _{s=0} = \frac{1}{\Gamma(p+\rho)} (-1)^k\, k! \, \mathcal R^{(1,p)}_{\alpha_2,\alpha_3}\ .
\]
Hence $Supp(F_0) = Supp(\mathcal R^{(1,p)}_{\alpha_2,\alpha_3}=\overline {\mathcal O_2}$.
\smallskip

For $ii)$, assume that $\varphi\in \mathcal C^\infty(S\times S\times S)$ has its support contained in $\mathcal O_0$. Then
\[\big(\mathcal F(s),\varphi\big) = \frac{(-1)^k}{k!} \frac{1}{\Gamma(-k+s)}\frac{(-1)^{k_2}}{k_2!} \frac{1}{\Gamma(-k_2+s)} \big(\mathcal K_{\boldsymbol \alpha(s),\mathcal O_0}, \varphi\big)
\]
As $s\rightarrow 0$, 
\[\big(\mathcal F(s),\varphi\big) \sim s \big( \mathcal K_{\boldsymbol \alpha, \mathcal O_0}, \varphi\big)
\]
and hence $(F_0,\varphi) = \big( \mathcal K_{\boldsymbol \alpha,\mathcal O_0}, \varphi\big)$, thus proving $ii)$.
\end{proof}

To prove Proposition \ref{noextZ1IIspe}, restrict the previous situation to $\mathcal O_0\cup \mathcal O_2$, and apply Proposition A1. Hence $\mathcal K_{\boldsymbol \alpha, \mathcal O_0}$ cannot be extended to $\mathcal O_0\cup \mathcal O_2$ (a fortiori to $S\times S\times S$) as a $\boldsymbol \lambda$-invariant distribution.
\end{proof}

We now are in position to prove Theorem \ref{mult2Z1IIspe}. Let $T$ be a $\boldsymbol \lambda$-invariant distribution. Then $T_{\vert \mathcal O_0}$ has to be $0$ by Proposition \ref{noextZ1IIspe}.  Hence $T$ is supported in $\mathcal O_1\cup \mathcal O_2\cup \mathcal O_3\cup \mathcal O_4$. But by using twice Lemma 4.1 in \cite{c}, $Supp(T)$ has to be contained in $\mathcal O_2\cup \mathcal O_4$. Next, consider the restriction $T_{\vert O_4^c}$ of $T$ to the open $G$-invariant subset $\mathcal O_4^c$, which contains $\mathcal O_2$ as a closed submanifold. Then, again by Lemma 4.1 in \cite{c}, $T_{\vert O_4^c}$ has to be a multiple of  $\mathcal R^{(1,p)}_{\alpha_2,\alpha_3\, \vert\mathcal O'}$. Otherwise stated, there exists a constant $c_1$ such that $T-c_1\mathcal R^{(1,p)}_{\alpha_2,\alpha_3}$ is supported on $\mathcal O_4$. 

Observe that $\lambda_2 = \rho+\frac{\alpha_1+\alpha_3} {2}= p+q\notin \{ -1,-2,\dots, -k\}\cup \{ -\rho,-\rho-1,\dots,-\rho-(k-1)\}$, and hence by Lemma A1,  $\dim Tri(\boldsymbol \lambda, diag)\leq 1$. But $\mathcal S^{(k)}_{\lambda_2,\lambda_3}$ is a no trivial element of $Tri(\boldsymbol \lambda, diag)$ and  generates 
$Tri(\boldsymbol \lambda, diag)$. Hence there exists a constant $c_2$ such that $T-c_1\mathcal R^{(1,p)}_{\alpha_2,\alpha_3}=c_2 \mathcal S^{(k)}_{\lambda_2,\lambda_3}$, thus finishing the proof of Theorem \ref{mult2Z1IIspe}.
\section{The multiplicity 2 result for $\boldsymbol \alpha \in Z_{2,I}$}

Let $\boldsymbol \alpha$  be in $Z_{2,I}$. Up to a permutation of the indices, this amounts to
 \begin{equation}\label{Z2I}
 \alpha_1 = 2k_1,\quad \alpha_2 =-(n-1)-2k_2, \quad \alpha_3 = -(n-1)-2k_3
\end{equation}
where $k_1,k_2,k_3\in \mathbb N$, with $k=k_2+k_3-k_1\geq 0$.
The conditions imply that
\[\alpha_1+\alpha_2+\alpha_3 = -2(n-1)-2k\ .
\]
The point $\boldsymbol \alpha$ belongs to two lines contained in $Z$, namely
\[\mathcal D^-\qquad \qquad \Big\{\ \begin{matrix} \alpha_2 =-(n-1)-2k_2\\ \alpha_3 = -(n-1)-2k_3\end{matrix}\quad ,
\]
\[\mathcal D^+ \qquad \Big\{ \begin{matrix} \alpha_1+\alpha_2+\alpha_3 = -2(n-1)-2k\\ \alpha_1 = 2k_1\end{matrix}\quad .
\]

The associated spectral parameter $\boldsymbol \lambda$ is given by
\[ \lambda_1 = -\rho-k_2-k_3, \quad \lambda_2 = k_1-k_3,\quad \lambda_3 = k_1-k_2\ ,
\]
or
\[\lambda_1 =-\rho-k_2-k_3, \quad \lambda_2 = -k+k_2,\quad \lambda_3 = -k+k_3\ 
\]
where $k\leq k_2+k_3$.

Change notation and set
\[\lambda_1=-\rho-l_1, \quad \lambda_2 = m_2, \quad \lambda_3 = m_3,
\]
where
\[ l_1\in \mathbb N,\ m_2,m_3\in \mathbb Z,\quad l_1+m_2+m_3\equiv 0 \mod 2,\quad \vert m_2\pm m_3\vert \leq l_1\ .
\]
The correspondence between notations is given by
\[l_1=k_2+k_3,\quad m_2=-k_3+k_1, \quad m_3 = -k_2+k_1\ .
\]

\begin{proposition}\label{dKZ2I}
 The differential $d\widetilde {\mathcal K}_{\boldsymbol \alpha}$ of $\boldsymbol \beta \longmapsto \widetilde {\mathcal K}_{\boldsymbol \beta}$ at $\boldsymbol \alpha$ is of rank 1. The distributions $\mathcal R^{(1,k_1)}_{\alpha_2,\alpha_3}, \widetilde {\mathcal T}^{(2,k_2)}_{\alpha_1,\alpha_3}
,\widetilde {\mathcal T}^{(3,k_3)}_{\alpha_1,\alpha_2}
$ and $\mathcal S^{(k)}_{\lambda_2,\lambda_3}$ are proportional  to $d\widetilde {\mathcal K}_{\boldsymbol \alpha}$.
\end{proposition}

\begin{proof} Consider the differential of $\boldsymbol \beta \longmapsto \widetilde {\mathcal K}_{\boldsymbol \beta}$. As $\widetilde {\mathcal K}_{\boldsymbol \beta}$ vanishes on $\mathcal D^+$ and $\mathcal D^-$, the rank of its differential at $\boldsymbol \alpha= \mathcal D^+\cap \mathcal D^-$ is at most 1.

Consider the plane $\beta_1 = 2k_1$. For arbitrary $\beta_2,\beta_3$, by \eqref{RK}
\[\widetilde {\mathcal K}_{2k_1,\beta_2,\beta_3} = \frac{ 1}{\Gamma(k_1+\rho)} \frac{ 1}{\Gamma(\frac{\beta_2+\beta_3}{2}+k_1+2\rho)}\, \mathcal R_{\beta_2,\beta_3}^{(1,k_1)}\ .
\] 
By differentiation this implies
\begin{equation}\frac{d}{ds}\big( \widetilde {\mathcal K}_{2k_1,\alpha_2+s,\alpha_3+s}\big)_{s=0} = \frac{(-1)^k k!}{\Gamma(k_1+\rho)} \, \mathcal R^{(1,k_1)}_{\alpha_2,\alpha_3}\ .
\end{equation}

To see that $\mathcal R^{(1,k_1)}_{\alpha_2,\alpha_3}\neq 0$, use \eqref{evRp} to get
\[\mathcal R^{(1,k_1)}_{\alpha_2,\alpha_3}(p_{a_1,a_2,a_3})  = NVT \times
(-k_2)_{a_2}(-k_3)_{a_3} (-k_2-k_3+a_2+a_3)_{k_1+a_1}
\] 
\[ \times \frac{ 1}{\Gamma(-k_2+\rho+k_1+a_2+a_1)\Gamma(-k_3+\rho+k_1+a_1+a_3)}\ .\]
Choose $a_2=a_3=0$ and $a_1 = k_2+k_3-k_1$  to get $\mathcal R^{(1,k_1)}_{\alpha_2,\alpha_3}(p_{a_1,0,0})\neq 0$, and hence $\mathcal R^{1,k_1}_{\alpha_2,\alpha_3}\neq 0$. It follows that the rank of $dK_{\boldsymbol \alpha}$ is 1, and the fact that $\mathcal R^{(1,k_1)}_{\alpha_2,\alpha_3}$ is proportional to $d\widetilde {\mathcal K}_{\boldsymbol \alpha}$. 

As a consequence, any partial derivative in a direction that does not lie in the plane containing $\mathcal D^+$ and $\mathcal D^-$ of $\widetilde {\mathcal K}_{\boldsymbol \beta}$ at $\boldsymbol \alpha$ does not vanish, and is a multiple of $d\widetilde {\mathcal K}_{\boldsymbol \alpha}$.
So similar results can be obtained for the other invariant distributions, namely 

\begin{equation}\label{KTZ2I}
\frac{d}{ds}\big( \widetilde {\mathcal K}_{\alpha_1,\alpha_2, \alpha_3+2s}\big)_{s=0} = \frac{\pi^\rho (-1)^{k_3+k_2} 4^{-k_3}  k_2!}{\Gamma(\rho+k_3)\Gamma(\rho+k_1)}\, \widetilde {\mathcal T}^{(3,k_3)}_{\alpha_1,\alpha_2}\ ,
\end{equation}
and a similar formula after exchanging the role of $2$ and $3$, and
\begin{equation}
\frac{d}{ds}\big( \widetilde {\mathcal K}^{\lambda_1+2s,\lambda_2-s,\lambda_3-s}\big)_{s=0} = \frac{(-1)^{l_1-k} (l_1-k)!}{\Gamma(\rho+k+\lambda_2) \Gamma(\rho+k+\lambda_3)}\, \mathcal S^{(k)}_{\lambda_2,\lambda_3} \ .
\end{equation}
\end{proof}
The plane containing $\mathcal D^+$ and $\mathcal D^-$ admits the equation $\lambda_1=-\rho-l_1$ and use $(\lambda_2,\lambda_3)$ as coordinates in this plane. At $\boldsymbol \mu = \boldsymbol \lambda$, $\widetilde {\mathcal K}^{\boldsymbol \mu}$ and it partial derivatives $\displaystyle \frac{\partial \widetilde {\mathcal K}^{\boldsymbol \mu}}{\partial \mu_2}$ and 
$\displaystyle\frac{\partial \widetilde {\mathcal K}^{\boldsymbol \mu}}{\partial \mu_3}$ vanish, so let  for $s\in \mathbb C$ 
\begin{equation}\label{Z2Ilambdas}
\boldsymbol \lambda(s)=(\lambda_1,\lambda_2+2s, \lambda_3) = (-\rho-l_1,m_2+2s, m_3)\ , 
\end{equation}
\begin{equation}\label{Z2Ialphas}
\boldsymbol \alpha(s) = (2k_1+2s, -(n-1)-2k_2-2s, -(n-1)-2k_3+2s)\ .
\end{equation}
Then the distribution 
\[\mathcal Q=\mathcal Q_{l_1,m_2,m_3} = \lim_{s\rightarrow 0} \frac{1}{s^2} \widetilde {\mathcal K}^{\lambda(s)}
\]
is $\boldsymbol \lambda$-invariant and corresponds to a mixed second order partial derivative of $\widetilde{\mathcal K}^{\boldsymbol \lambda}$ as explained in the introduction.

\begin{lemma} The $K$-coefficients of the distribution $\mathcal Q$ are equal to $0$ unless 
\[a_2+a_3>k_2+k_3 \text\quad {and\quad either \quad } a_2\leq k_2 \text{\quad or\quad } a_3\leq k_3\ .
\]
Assume $a_2\leq k_2$ and $a_2+a_3> k_2+k_3$. Then
\begin{equation}\label{evQp}
\begin{split}
\mathcal Q(p_{a_1,a_2,a_3}) = \big(\frac{\sqrt{\pi}}{2}\big)^{3(n-1)} 2^{-\rho-2k+a_1+a_2+a_3}\\
(-1)^k\, k! \,(-k+a_1+a_2+a_3-1)! \,(\rho+k_1)_{a_1} \,(-k_2)_{a_2}\\
(-1)^{k_3}\, k_3!\, (-k_3+a_3-1)!\,
\frac{1} {\Gamma(-k_2-k_3+a_2+a_3)}\\ \frac{1}{\Gamma(m_2+\rho+a_3+a_1) \,\Gamma( m_3 +\rho+a_2+a_1)}\ .
\end{split}
\end{equation}
A similar formula holds for the symmetric situation, where $a_3\leq k_3$ and $a_2+a_3>k_2+k_3$. 
\end{lemma}
\begin{proof} Use  \eqref{Kplambda} in the Appendix. $\widetilde {\mathcal K}^{\boldsymbol \lambda(s)} (p_{a_1,a_2,a_3})$ is (essentially) a product of seven factors, which we now analyze. There is an inverse $\Gamma$ factor   which in the present situation does not depend on $s$ an is equal to $ \frac{1}{\Gamma( -l_1+a_2+a_3)}$. Hence, unless $a_2+a_3>l_1 = k_2+k_3$, this factor is $0$ for all $s$ and hence $\mathcal Q(p_{a_1,a_2,a_3})=0$. Now assume that $a_2+a_3>l_1 = k_2+k_3$. Another factor is  equal to $(-k+s)_{a_1+a_2+a_3}$. In this Pochahammer's symbol the last factor  is equal to $(-k+s+a_1+a_2+a_3-1)$. As $a_2+a_3>k_2+k_3 \geq k$, $-k+a_1+a_2+a_3-1\geq 0$, hence this factor has a simple zero at $s=0$. More explicitly
\[(-k+s)_{a_1+a_2+a_3}= [(-1)^k \,k!\,(-k+a_1+a_2+a_3-1)!] \, s + O(s^2)\ .
\]
Two other factors are equal  to $(-k_2-s)_{a_2}$ and. $(-k_3+s)_{a_3}$. If $a_2>k_2$ and $a_3>k_3$ each factor vanishes for $s=0$, so that $\widetilde {\mathcal K}^{\boldsymbol \lambda(s)}(p_{a_1,a_2,a_3})$ has a zero of order 3 at $s=0$ and hence $\mathcal Q(p_{a_1,a_2,a_3}) = 0$. Now assume $a_2\leq k_2$ and still $a_2+a_3>k_2+k_3$, which forces $k_3>a_3$. The factor $(-k_2-s)_{a_2}$ does not vanish for $s=0$ and its value is $(-k_2)_{a_2}$. The factor $(-k_3+s)_{a_3}$ has a simple zero at $s=0$ and more precisely,
$ (-k_3+s)_{a_3} = (-1)^{k_3} \, k_3!\, (-k_3+a_3-1)!\,s +O(s^2)$. The remaining factors offer no difficulty. Collect all facts to get the statement.
\end{proof}

\begin{corollary}\label{nonvanQ}
 The distribution $\mathcal Q$ is $\neq 0$. More precisely, for $a_1$ large enough, $a_2+a_3>k_2+k_3$ and either $a_2\leq k_2$ or $a_3\leq k_3$, $\mathcal Q(p_{a_1,a_2,a_3}) \neq 0$. Moreover, 
\[Supp(\mathcal Q) =\overline{ \mathcal O_2\cup \mathcal O_3} \ .
\]
\end{corollary}

\begin{proof} For $a_1$ large enough, the two inverse $\Gamma$ factors on the last line of \eqref{evQp}  do not vanish, so that the whole expression does not vanish, and hence $\mathcal Q\neq 0$. 

Let $\varphi\in \mathcal C^\infty(S\times S\times S)$ and assume $Supp\,( \varphi) \subset \mathcal O_0\cup \mathcal O_1$. Then 
\[\inf\Big(\vert x-y\vert, (x,y,z)\in Supp(\varphi)\Big) >0,\inf\Big(\vert x-z\vert, (x,y,z)\in Supp(\varphi)\Big) >0\ .
\]
The exponent of $\vert y-z\vert$ is  $\alpha_1(s)=2k_1+2s$, so that, for $s$ in a small neighborhood of $0$, the integral $\mathcal K_{\boldsymbol \alpha(s)}(\varphi)$ converges, and  $\lim_{s\rightarrow 0} \mathcal K_{\boldsymbol \alpha(s)}(\varphi)= \mathcal K_{\boldsymbol \alpha}(\varphi)$. In turn,
\[\mathcal K_{\boldsymbol \alpha(s)}(\varphi) = \Gamma(k_1+\rho+s) \Gamma(-k_2+s)\Gamma(-k_3+s) \Gamma(-k+s) \widetilde {\mathcal K}_{\boldsymbol \alpha(s)} (\varphi)
\]
\[= C\, \frac{1}{s}\, \mathcal Q(\varphi)+O(1)
\]
with $C\neq 0$, which forces $\mathcal Q(\mathcal \varphi) = 0$.

As a consequence $Supp(\mathcal Q) \subset \overline {\mathcal O_2\cup \mathcal O_3}$. As $Supp(\mathcal Q)$ is invariant under the action of $G$, if  $Supp(\mathcal Q) \neq \overline{\mathcal O_2 \cup \mathcal O_3}$, then $Supp(\mathcal Q)$ has to be contained in either $\overline {\mathcal O_2}$ or $\overline {\mathcal O_3}$. Assume  $Supp(\mathcal Q) \subset \overline {\mathcal O_2}$. Then $\mathcal Q(f)=0$ if the smooth function $f$ vanishes at a sufficiently large order on  $\overline {\mathcal O_2}$. In particular, $\mathcal Q(p_{a_1,a_2,a_3})=0$ if $a_2$ is large enough. But this contradicts Corollary \ref{nonvanQ} (choose $a_3\leq k_3$, $a_2$ such that $a_2+a_3>k_2+k_3$ and $a_1$ large). Hence the only possibility is $Supp( \mathcal Q) = \overline{\mathcal O_2\cup \mathcal O_3}$.
\end{proof}

\begin{proposition} The two distributions $ d\widetilde {\mathcal K}_{\boldsymbol \alpha}$ and $\mathcal Q$ are linearly independent.
\end{proposition}

\begin{proof} The statement follows from the study of the supports of the two distributions.
\end{proof}
\begin{theorem}\label{mult2Z2I}
 Let $\boldsymbol \lambda$ be in $Z_{2,I}$. Then 
\[\dim Tri(\boldsymbol \lambda) = 2\ .
\]
More precisely,
\[Tri(\boldsymbol \lambda) = \mathbb C \mathcal R^{(1,k_1)}_{\alpha_2,\alpha_3}\oplus \mathbb C\mathcal Q_{l_1,m_2,m_3}\ .
\]
\end{theorem}
Some preparation is needed for the proof of the theorem.
\begin{proposition} \label{noextZ2I}
Let $\boldsymbol \alpha$ satisfy conditions \eqref{Z2I}. The distribution $\mathcal K_{\boldsymbol \alpha, \mathcal O_0}$ cannot be extended to $S\times S\times S$ as a $\boldsymbol \lambda$-invariant distribution.
\end{proposition}

\begin{proof} Recall that three of the $\Gamma$ factors involved in the normalization of $\mathcal K_{\boldsymbol \alpha}$ become singular at $\boldsymbol \alpha$. 

For $s$ a complex number, let $\boldsymbol \alpha(s)$ as in \eqref{Z2Ialphas} and let 
\[\mathcal G(s) = -\,\frac{(-1)^k}{k!}\,\frac{(-1)^{k_2}}{k_2!}\,\frac{(-1)^{k_3}}{k_3!}\Gamma(k_1+\rho+s)\, \frac{1}{s^2}\,\widetilde {\mathcal K}_{\boldsymbol \alpha(s)}\ .
\]
As $\widetilde {\mathcal K}_{\boldsymbol \alpha}(s)$ vanishes together with its first derivative at $s=0$, $\mathcal G(s)$ is well defined  for $s$ in a neighborhood of $0$ and 
its Taylor expansion at $s=0$ reads
\[\mathcal G(s) = G_0+  sG_1+ O(s^2)\ .
\]
where $ G_0$ and $G_1$  are distributions on $S\times S\times S$.

\begin{lemma} {\ } 

\smallskip

$i)$ $Supp(G _0)= \overline{\mathcal O_2\cup \mathcal O_3}$ 
\smallskip

$ii)$ $G_1$ coincides on $\mathcal O_0$ with $\mathcal K_{\boldsymbol \alpha, \mathcal O_0}$.

\end{lemma}

\begin{proof} For $i)$, observe that $G_0 = -\,\frac{(-1)^k}{k!}\,\frac{(-1)^{k_2}}{k_2!}\,\frac{(-1)^{k_3}}{k_3!}\Gamma(k_1+\rho)\, \mathcal Q$. Hence $i)$ follows from Corollary \ref{nonvanQ}. 

In order to prove $ii)$, let $\varphi \in \mathcal C_c^\infty(\mathcal O_0)$. Then for $s\neq 0$,
\[\big(\mathcal F(s),\varphi\big)\]\[= \frac{(-1)^k}{k!} \frac{1}{\Gamma(-k+s)} \frac{(-1)^{k_2}}{k_2!} \frac{-1}{\Gamma(-k_2-s)}\frac{(-1)^{k_3}}{k_3!} \frac{ 1}{\Gamma(-k_3+s)}\frac{1}{s^2}\big( \mathcal K_{\boldsymbol \alpha(s), \mathcal O_0} , \varphi \big) .
\]
As $s\rightarrow 0$, this identity implies
\[ \big(\mathcal G(s), \varphi\big) \sim s \big({\mathcal K}_{\boldsymbol \alpha, \mathcal O_0}\,, \varphi\big)\ .
\]
so that $ii)$ follows.
\end{proof}
Now we are ready to prove Proposition \ref{noextZ2I}. Let $\mathcal O=\mathcal O_0\cup \mathcal O_2$. This is an open $G$-invariant subset which contains $\mathcal O_2$ as a closed submanifold. Restrict $\mathcal G(s), G_0$ and $G_1$ to $\mathcal O$. Now apply Proposition A1 to conclude that $G_1$ cannot be extended to $\mathcal O$ (a fortiori to $S\times S\times S$) as a $\boldsymbol \lambda$-invariant distribution. 
\end{proof}

We still need another result before starting the proof of Theorem \ref{mult2Z2I}.
\begin{proposition}\label{noextT3k3}
 The distribution $\mathcal T^{3,k_3}_{\alpha_1,\alpha_2, \mathcal O_4^c}$ cannot be extended to $S\times S\times S$ as a $\boldsymbol \lambda$-invariant distribution.
\end{proposition}

\begin{proof} As seen in section 9, for arbitrary $\beta_1,\beta_2\in \mathbb C$, the distribution $\mathcal T^{(3,k_3)}_{\beta_1,\beta_2, \mathcal O_4^c}$ is a $\lambda$-invariant distribution on $\mathcal O_4^c$. For $s\in \mathbb C$, let 
\[ \boldsymbol \alpha(s) = (\alpha_1, \alpha_2+2s, \alpha_3), \quad \lambda(s) = (\lambda_1+s,\lambda_2,\lambda_3+s)
\]
and let 
\[\mathcal F(s) = \frac{(-1)^k}{k!} \widetilde {\mathcal T}^{(3,k_3)}_{2k_1, \alpha_2+2s}\ .
\]
The Taylor expansion of $\mathcal F(s)$ at $s=0$ reads
\[\mathcal F(s) = F_0+sF_1+O(s^2)\ ,
\]
where $F_0,F_1$ are distributions on $S\times S\times S$. 

\begin{lemma} {\ }
\smallskip

$i)$ $Supp(F_0)=\mathcal O_4$

$ii)$ the restriction of $F_1$ to $\mathcal O_4^c$ coincides with $\mathcal T^{(3,k_3)}_{\alpha_1,\alpha_2, \mathcal O_4^c}$.
\end{lemma}
\begin{proof}
For $i)$, $\displaystyle F_0=\mathcal F(0)= \frac{(-1)^k}{k!} \widetilde {\mathcal T}^{3,k_3}_{2k_1, \alpha_2}$ and hence by \eqref{KTZ2I}, $Supp(F_0) = \mathcal O_4$.

For $ii)$, let $\varphi \in \mathcal C^\infty(S\times S\times S)$ and assume that $Supp(\varphi)\subset\mathcal O_4^c$.Then, using \eqref{normT} 
\[\big(\mathcal F(s),\varphi\big) = \frac{(-1)^k}{k!}\frac{1}{\Gamma(-k+s)} \big( \mathcal T^{(3,k_3)}_{2k_1,\alpha_2+2s,\mathcal O_4^c}, \varphi\big)\ ,
\]
and $ii)$ follows by letting $s\rightarrow 0$.
\end{proof}

Proposition A2 implies that  $\mathcal T^{(3,k_3)}_{\alpha_1,\alpha_2, \mathcal O_4^c}$ cannnot be extended to $S\times S\times S$ as a $\boldsymbol \lambda$-invariant distribution.
 \end{proof}
A similar statement holds for $\mathcal T^{(2,k_2)}_{\alpha_1,\alpha_3, \mathcal O_4^c}$. However, remark that the restriction of $Q$ to $\mathcal O_4^c$ is a (non trivial) linear combination of $\mathcal T^{(2,k_2)}_{\alpha_1,\alpha_3, \mathcal O_4^c}$ and $\mathcal T^{(3,k_3)}_{\alpha_1,\alpha_2, \mathcal O_4^c}$.

We can now proceed to the proof of Theorem \ref{mult2Z2I}. Let $T$ be a $\boldsymbol \lambda$-invariant distribution. The restriction of $T$ to $\mathcal O_0$ has to be $0$ by Proposition  \ref{noextZ2I}. Moreover, as $\alpha_1\notin (-(n-1)-2\mathbb N)$, the restriction of $T$ to $\mathcal O_0\cup\mathcal O_1$ is equal to $0$. Hence $Supp(T)\subset \mathcal O_2\cup \mathcal O_3\cup \mathcal O_4$. Consider the restriction $T_{\Omega_{12}}$ of $T$ to the $G$-invariant open set $\Omega_{12}=\mathcal O_0\cup \mathcal O_1\cup \mathcal O_2$. $T_{\Omega_{12}}$ is $\boldsymbol \lambda$-invariant and supported on the (relatively) closed submanifold $\mathcal O_2$. The restriction of $\mathcal Q_{l_1,m_2,m_3}$ to $\Omega_{12}$ is $\boldsymbol \lambda$-invariant and $\neq 0$ by Corollary \ref{nonvanQ}. Hence by Lemma 4.1 in \cite{c}, there exists a constant $c$ such that $T$ and $cQ$ coincide on $\Omega_{12}$. Now consider $T-cQ$. It is a $\boldsymbol \lambda$-invariant distribution which is supported on $\overline {\mathcal O_3}$. Its restriction to $\mathcal O_4^c$ is supported on $\mathcal O_3$ and must coincide on $\mathcal O_4^c$ with a multiple of $\mathcal T^{(3,k_3)}_{\alpha_1,\alpha_2, \mathcal O_4^c}$ again by Lemma 4.1 in \cite{c}. By Proposition \ref{noextT3k3}, this forces $T-cQ$ to be $0$ on $\mathcal O_4^c$, and hence, $T-cQ$ is supported on $\mathcal O_4$. 

To finish the proof it is enough to prove that the space $Tri(\boldsymbol \lambda, diag)$ has dimension 1 (hence generated by the distribution $\mathcal R^{(1,k_1)}_{\alpha_2,\alpha_3}$), because this will imply that $T=c\mathcal Q + d \mathcal R^{(1,k_1)}_{\alpha_2,\alpha_3}$ for some constant $d$, which is the content of Theorem \ref{mult2Z2I}. Observe that \[\lambda_1 = -\rho-k_2-k_3\notin \{-\rho,-\rho-1,\dots, -\rho-(k-1)\}\cup \{-1,-2,\dots, -k\}\] as $k_2+k_3\geq k$. Hence, by Appendix 2, this implies $dim Tri(\boldsymbol \lambda)=1$. Q.E.D.

\section{ The multiplicity 2 result for $\boldsymbol \alpha \in Z_{2,II}$}

Let $\boldsymbol \alpha\in Z_{2,II}$. Up to a permutation of the indices,  it amounts to 
\begin{equation}\label{Z2II}
\alpha_1 = 2k_1,\quad  \alpha_2 = 2k_2,\quad \alpha_3 = -2(n-1)-2k_3
\end{equation}
where $k_1,k_2,k_3\in \mathbb N $ and $k_1+k_2\leq k_3$. Observe that $\alpha_1+\alpha_2+\alpha_3 = -2(n-1)-2(k_3-k_1-k_2)=-2(n-1)-2k$ where $k=k_3-k_1-k_2$, so that $\boldsymbol \alpha$ is a pole of type II. If $n-1$ is odd, $\boldsymbol \alpha$ is not a pole of type I, but if $n-1$ is even, then $\alpha_3 = -(n-1)-2(\rho+k_3)$ so that $\boldsymbol \alpha$ is a pole of type I+II.

The alternative description of $Z_{2,II}$ in terms of the spectral parameter is
\begin{equation}\label{Z2IIlambda}
\begin{split}
&\lambda_1 = -\rho-l_1,\quad  \lambda_2 = -\rho-l_2,\quad \lambda_3 = \rho+m_3\\
 l_1,l_2, m_3\in \mathbb N,\quad &
l_1+l_2+m_3\equiv 0 \mod 2,\quad \vert l_1-l_2\vert \leq m_3\leq l_1+l_2\ .
\end{split}
\end{equation}
Also let $l_1+l_2-m_3 = 2k$. Then, $l_1,l_2$ satisfy $k\leq l_1,k\leq l_2$.

The point $\boldsymbol \lambda$ belongs to two lines contained in $Z$, namely
\[ \mathcal D^1 = \mathcal D^{1,+}_{-l_2+m_3}  \qquad \mathcal D^2=  \mathcal D^{2,+}_{-l_1+m_3}\ .
\]

\begin{proposition}\label{RZ2II}
 {\ }
\smallskip

$i)$ the differential $d\widetilde {\mathcal K}_{\boldsymbol \alpha} $ of $\boldsymbol \beta \longmapsto \widetilde {\mathcal K}_{\boldsymbol \beta}$ at $\boldsymbol \alpha$ is of rank 1

$ii)$ the distributions $ \mathcal R^{(1,k_1)}_{\alpha_2,\alpha_3}$ and $\mathcal R^{(2,k_2)}_{\alpha_1,\alpha_3}$ are $\neq 0$ and proportional to $d\widetilde {\mathcal K}_{\boldsymbol \alpha} $. Moreover, 

$\bullet$ If $n-1$ is odd, $Supp(\mathcal R^{(1,k_1)}_{\alpha_2,\alpha_3}) =S\times S\times S$.

$\bullet$ If $n-1$ is even, then $Supp(\mathcal R^{(1,k_1)}_{\alpha_2,\alpha_3})= \overline{\mathcal O_3}$.

\end{proposition}

\begin{proof} As $\widetilde {\mathcal K}_{\boldsymbol \beta}$ vanishes on $\mathcal D^1$ and $\mathcal D^2$, the rank of $d\widetilde {\mathcal K}_{\boldsymbol \alpha}$ is at most 1. 

 From (49) follows for generic $\beta_3$ 
\[\widetilde {\mathcal K}_{2k_1,2k_2, \beta_3}= \frac{1}{\Gamma(\rho+k_1)} \frac{ 1}{\Gamma( \frac{k_1+k_2}{2}+\frac{\beta_3}{2} +2\rho)}\mathcal R^{(1,k_1)}_{2k_2, \beta_3}
\]
Let $\beta_3$ tend to $\alpha_3 = -2(n-1)-2k_3$ to obtain
\begin{equation}\label{derKZ2II}
\Big({\frac{\partial}{\partial \beta_3}\widetilde {\mathcal K}_{2k_1,2k_2, \beta_3}}\Big)_{\beta_3 = \alpha_3} = \frac{1}{\Gamma(\rho+k_1)}\,(-1)^k\, k! \, \mathcal R^{(1,k_1)}_{2k_2, \alpha_3}
\end{equation}
and a similar formula for $\mathcal R^{2,k_2}_{2k_1,\alpha_3}$. 

Next, use (50) (after permuting the indices) for $a_3 = k_3$. Then

\[\mathcal R^{(1,k_1)}_{2k_2, \alpha_3}(p_{a_1,a_2,a_3})= C\, (k_2+\rho)_{a_2} \,(-\rho-k_3)_{k_3}\, (k_2+a_2)_{k_1+k_3}
\]
\[\frac{1}{\Gamma( k_2+2\rho+k_1+a_2+k_3)}\,\frac{1}{\Gamma(k_1+a_1)}
\]
which is clearly different of $0$. Hence  $\mathcal R^{(1,k_1)}_{2k_2, \alpha_3}$ is $\neq 0$, $d \widetilde {\mathcal K}_{\boldsymbol \alpha}$ is of rank 1 and proportional to $\mathcal R^{(1,k_1)}_{2k_2, \alpha_3}$.  Hence $i)$ holds true, and after permuting the indices $1$ and $2$,  the first part of $ii)$ follows.

For the second statement in $ii)$, assume first that $n-1$ is odd. Then (50) implies that for $a_3>k_3$ and $a_1,a_2$ arbitrary, the $K$-coefficient $\mathcal R^{(1,k_1)}_{2k_2, \alpha_3}(p_{a_1,a_2,a_3})$ does not vanish. This implies $Supp(\mathcal R^{(1,k_1)}_{\alpha_2,\alpha_3})= S\times S\times S$.

Now assume that  $n-1$ is even. Let $\varphi\in \mathcal C^\infty(S\times S\times S)$ and assume that $Supp(\varphi)\subset \mathcal O_0\cup \mathcal O_1\cup \mathcal O_2$. This implies 
\[\inf \{ \vert x-y\vert, (x,y,z)\in Supp(\varphi)\} >0\ .
\]
Hence the integral which defines $\mathcal K_{2k_1,2k_2, \beta_3}(\varphi)$ is convergent. For generic $\beta_3$, 
\[\mathcal K_{2k_1,2k_2,\beta_3} (\varphi)= \Gamma( k_2+\rho)\, \Gamma( \frac{\beta}{2} +\rho)\, \mathcal R^{(1,k_1)}_{2k_2,\beta_3}(\varphi)\ .
\]
Now as $\beta_3$ tends to $\alpha_3 \in -(n-1)-2\mathbb N$, the factor $\Gamma( \frac{\beta_3}{2}+\rho)$ becomes singular, so that necessarily,
$\mathcal R^{(1,k_1)}_{2k_2,\beta_3}(\varphi) = 0$. Hence $Supp(\mathcal R^{(1,k_1)}_{2k_2,\alpha_3}\subset \overline {\mathcal O_3}$.

As for $a_3=k_3$ and $a_1,a_2$ arbitrary, $\mathcal R^{(1,k_1)}_{2k_2, \alpha_3}(p_{a_1,a_2,a_3})\neq 0$, this rules out the possibility that 
$Supp(\mathcal R^{(1,k_1)}_{2k_2, \alpha_3}) = \mathcal O_4$. Hence $ii)$ follows.
\end{proof}
\begin{proposition}
The distribution $\mathcal S^{(k)}_{\lambda_1,\lambda_2}$ is $\neq 0$ and $Supp(\mathcal S^{(k)}_{\lambda_1, \lambda_2})=\mathcal O_4$. 
\end{proposition}
\begin{proof}
It suffices to show that some $K$-coefficient of $\mathcal S^{(k)}_{\boldsymbol \lambda}$ does not vanish. Choose $a_1=a_2=0$ and $a_3=k$. Then, from (48) follows
\begin{equation}\label{Snot0}
\mathcal S^{(k)}_{ \lambda_1,\lambda_2}(p_{0,0,k}) = C\, (-k)_k\,(-k-\rho-m_3)_k \,(2\rho+m_3)_k\ \neq \ 0\ . 
\end{equation}
\end{proof}

The relation of $\mathcal S^{(k)}_{ \lambda_1,\lambda_2}$ to derivatives of $\widetilde{\mathcal K}^{\boldsymbol \mu}$ at $\boldsymbol \lambda$ (as suggested in  the introduction) is given by the following lemma.

\begin{lemma}\label{KSZ2II}
\[ \frac{d^2}{ds^2} \big(\widetilde {\mathcal K}^{\lambda_1-s, \lambda_2+s, \lambda_3} \big)_{s=0} =  \gamma\,\mathcal S^{(k)}_{\lambda_1,\lambda_2}\ .
\]
with 
\[\gamma = 2\big(\frac{\pi}{16}\big)^{n-1} \frac{(n-2)!}{\Gamma(\rho+k)}16^{-k}(-1)^{k_1} k_1!(-1)^{k_2} k_2! \frac{1}{\Gamma( n-1+k_3))}
\]
\end{lemma}
\begin{proof}  Let \[\boldsymbol \lambda(s) = (\lambda_1-s,\lambda_2+s,\lambda_3), \qquad \boldsymbol \alpha(s) = (\alpha_1+2s, \alpha_2-2s, \alpha_3)\ .\]
Observe that  $\boldsymbol \alpha(s)$ belongs to the plane $\{ \beta_1+\beta_2+\beta_3=-2(n-1)-2k\}$. In particular, by \eqref{KS}
\begin{equation}
\widetilde {\mathcal K}_{\boldsymbol \alpha(s)} = \big(\frac{\pi}{16}\big)^{n-1} \frac{(n-2)!}{\Gamma(\rho+k)}16^{-k}\frac{1}{\Gamma(-k_1+s)\Gamma(-k_2-s)\Gamma(-\frac{\alpha_3}{2})}\,S^{(k)}_{\lambda_1-s,\lambda_2+s}\ .
\end{equation}
As
\[\frac{d^2}{ds^2} \big(\widetilde {\mathcal K}^{\lambda_1-s, \lambda_2+s, \lambda_3} \big)_{s=0}  =2 \lim_{s\rightarrow 0} \frac{\widetilde{\mathcal K}_{\boldsymbol \alpha(s)}}{s^2} \]
the result follows by letting $s$ tend to $0$.
\end{proof}

\begin{proposition}\label{RSindep}
 The distributions $R^{(1,k_1)}_{k_2,\alpha_3}$ and $\mathcal S^{(k)}_{\lambda_1,\lambda_2}$ are two linearly independant $\boldsymbol \lambda$-invariant distributions.
\end{proposition}
\begin{proof}
The two distributions $ R^{(1,k_1)}_{k_2,\alpha_3}$ and $\mathcal S^{(k)}_{\lambda_1,\lambda_2}$ are $\neq 0$ and have unequal supports. Hence they are linearly independant.
\end{proof}
\begin{theorem}\label{mult2Z2II}
 Let $\boldsymbol \alpha$ be in $Z_{2,II}$. Then 
$dim \ Tri( \boldsymbol \lambda) = 2$. More precisely, assume that $\boldsymbol \alpha$ satisfies \eqref{Z2II}. Then
\[Tri(\boldsymbol \lambda) = \mathbb C\,  R^{(1,k_1)}_{k_2,\alpha_3}\oplus \mathbb C\, \mathcal S^{(k)}_{\lambda_1,\lambda_2}\ .
\]
\end{theorem}

From Proposition \ref{RSindep} follows that $\dim Tri(\boldsymbol \lambda)\geq 2$. Hence it will be sufficient to prove that $\dim(Tri(\boldsymbol \lambda) \leq 2$, and this is the content of the next subsections.
\subsection {The odd case}
Assume that $\boldsymbol\alpha$ satisfy \eqref{Z2II}, and assume that $n-1$ is odd. In this case, $\boldsymbol \alpha$ is a generic pole of type II. 

\begin{proposition} Let $T$ be a $\boldsymbol \lambda$-invariant distribution on $S\times S\times S$. Then there exist two constants $c,d$ such that $T= c\, R^{(1,k_1)}_{k_2,\alpha_3} + d\, \mathcal S^{(k)}_{ \lambda_1,\lambda_2}$.
\end{proposition}

\begin{proof} The restriction to $\mathcal O_0$ of any $\boldsymbol \lambda$-invariant distribution has to be proportional to $\mathcal K_{\boldsymbol \alpha,\mathcal O_0}$. Hence there exists a constant $c$ such that $S=T-c R^{1,k_1}_{k_2,\alpha_3}$ is supported in $\mathcal O_1\cup \mathcal O_2\cup \mathcal O_3\cup \mathcal O_4$. 

Now $\alpha_1,\alpha_2,\alpha_3 \notin (-(n-1)-2\mathbb N)$. As already seen, this implies that $S$ is supported in $\mathcal O_4$, or otherwise said $S$ belongs to $Tri(\boldsymbol \lambda, diag)$.  The dimension of this space is equal to $\dim {\it Sol}(\lambda_1,\lambda_2;k)$. 

Notice that $\boldsymbol \lambda_3=\rho+m_3\notin \{ -\rho,-\rho-1,\dots, -\rho-(k-1)\}\cup \{ -1,-2,\dots, -k\}$ and hence, by Appendix 3  $dim \ Tri(\boldsymbol \lambda, diag)= 1$. So there exists a constant $d$ such that $S=d\, \mathcal S^{(k)}_{\boldsymbol \lambda}$. Summarizing, $T= c\, R^{1,k_1}_{k_2,\alpha_3} + d\, \mathcal S^{(k)}_{\boldsymbol \lambda}$. Q.E.D.
\end{proof}

\subsection{The even case}

Assume $\boldsymbol \alpha$ satisfies \eqref{Z2II}, and assume that $n-1$ is even. In this case, $\boldsymbol \alpha$ is a pole of type I+II.

\begin{proposition}\label{noextZ2IIp}
 The distribution $\mathcal K_{\boldsymbol \alpha, \mathcal O_0}$ cannot be extended to a $\boldsymbol \lambda$-invariant distribution on $S\times S\times S$.
\end{proposition}

Two of the normalizing $\Gamma$ factors become singular at $\boldsymbol \alpha$, namely $\Gamma(\frac{\alpha_3}{2}+\rho)$ and $\Gamma( \frac{ \alpha_1+\alpha_1+\alpha_2}{2}+2\rho)$. For $s\in \mathbb C$, set
\[\boldsymbol \alpha(s) = (\alpha_1,\alpha_2, \alpha_3+2s), \qquad \boldsymbol \lambda(s) = (\lambda_1+s,\lambda_2+s, \lambda_3)\ ,
\]
and  notice that $\boldsymbol \alpha(s)$ is transverse at $\boldsymbol \alpha$ to both planes of poles. Let
\[\mathcal F(s) =\Gamma(k_1+\rho)\Gamma( k_2+\rho) (-1)^{\rho+k_3} (\rho+k_3)! (-1)^k\ k!\,\frac{1}{s}\, \widetilde {\mathcal K}_{\boldsymbol \alpha(s)}\ .\]
This defines a distribution-valued holomorphic function on $\mathbb C$, which admits the following Taylor expansion at $0$
\[\mathcal F(s) = F_0+ s F_1+O(s^2)\ .
\]
\begin{lemma} {\ }
\smallskip

$i)$ $F_0$ is $\boldsymbol \lambda$-invariant and $Supp(F_0)= \overline {\mathcal O_3}$.
\smallskip

$ii)$ the restriction of $F_1$ to $\mathcal O_0$ coincides with $\mathcal K_{\boldsymbol \alpha, \mathcal O_0}$.
\end{lemma}
\begin{proof}
For $i)$, notice that $F_0$ is a multiple ($\neq 0$) of $\mathcal R^{1,k_1}_{\alpha_2, \alpha_3}$ (cf \eqref{derKZ2II}), hence $Supp(F_0) = \overline{\mathcal O_3}$ by Proposition \ref{RZ2II}.

 Let $s\neq 0$, and let $\varphi \in \mathcal C^\infty_c(\mathcal O_0)$. Then
\[\mathcal F(s)(\varphi)= \frac{ (-1)^{\rho+k_3}\,(\rho+k_3)!}{ \Gamma( -\rho-k_3 +s)}\, \frac{ (-1)^k\,k!}{\Gamma( -k+s)}\,\frac{1}{s}\, \mathcal K_{\boldsymbol \alpha(s)}(\varphi)\ .
\]
Let $s$ tend to $0$. The right hand side is equivalent to  $s\,\mathcal K_{\boldsymbol \alpha, \mathcal O_0}( \varphi)$. Hence $F_0(\varphi) = 0$ and $F_1(\varphi) = \mathcal K_{\boldsymbol \alpha, \mathcal O_0} (\varphi)$, from which $ii)$ follows.
\end{proof}

Proposition \ref{noextZ2IIp} now follows from Proposition A1 in Appendix.

To finish the proof of Theorem \ref{mult2Z2II}, assume $T$ is a $\boldsymbol \lambda$-invariant distribution. By Proposition \ref{noextZ2IIp}, $T$ must be $0$ on $\mathcal O_0$. Hence $Supp(T)$ is included in $\mathcal O_1\cup \mathcal O_2\cup \mathcal O_3\cup \mathcal O_4$. As $\alpha_1, \alpha_2\notin -(n-1)-2\mathbb N$, it even follows that $Supp(T)\subset \overline {\mathcal O_3}$. By Lemma 4.1 in \cite{c}, the restriction of $T$ to $\mathcal O_4^c$ has to be proportional to (the restriction to $\mathcal O_4^c$ of) $\mathcal R^{(1,k_1)}_{2k_2,\alpha_3}$. In other words, there exists a constant $c$ such that 
$T-c \,\mathcal R^{(1,k_1)}_{2k_2,\alpha_3}$ is $\boldsymbol \lambda$-invariant and supported in $\mathcal O_4$. The final conclusion is obtained as in the proof for the odd case.

\subsection{Invariant trilinear form for three finite dimensional representations}

The case where $\boldsymbol \lambda$ is in $Z_{2,II}$ allows to draw conclusions for finite-dimensional representations.

The following lemma is valid in full generality and in fact could be used for other situations (cf \cite{ks} for a related idea). 

\begin{lemma}\label{KJK}
 Let $\boldsymbol \lambda=(\lambda_1,\lambda_2,\lambda_3) \in \mathbb C^3$. Then, for any $f_1,f_2,f_3\in \mathcal C^\infty(S)$ 
\[\widetilde{\mathcal K}^{-\lambda_1,\lambda_2,\lambda_3} (\widetilde J_{\lambda_1} f_1,f_2,f_3)= \frac{\pi^\rho}{\Gamma(-\lambda_1+\rho)}\widetilde {\mathcal K}^{\lambda_1,\lambda_2,\lambda_3} (f_1,f_2,f_3)
\]
\end{lemma}

\begin{proof} Because of the intertwining relation satisfied by $\widetilde J_{\lambda_1}$, the left hand side is a $\boldsymbol \lambda$-invariant trilinear form. Assume that $\boldsymbol \lambda$ is not a pole. By the generic uniqueness theorem, there exists a constant $c$ such that, for all $f_1,f_2,f_3\in \mathcal C^\infty (S)$,
\[\widetilde{\mathcal K}^{-\lambda_1,\lambda_2,\lambda_3} (\widetilde J_{\lambda_1} f_1,f_2,f_3)= c\, \widetilde {\mathcal K}^{\lambda_1,\lambda_2,\lambda_3} (f_1,f_2,f_3)\ .
\]
Now recall that 
\begin{equation}\label{evK1}
\widetilde {\mathcal K}^{\boldsymbol \lambda}(1,1,1) = \big(\frac{\sqrt{\pi}}{2}\big)^{3(n-1)} \frac{ 2^{\lambda_1+\lambda_2+\lambda_3}}{\Gamma(\rho+\lambda_1)\Gamma(\rho+\lambda_2)\Gamma(\rho+\lambda_3)}
\end{equation}
 (see (10) in \cite{c}).
Assume moreover that for $j=1,2,3$, $\lambda_j\notin -\rho-\mathbb N$, so that  $\widetilde {\mathcal K}^{\boldsymbol \lambda}(1,1,1)\neq 0$. Hence the constant $c$ can be computed from \eqref{evK1}, using \eqref{lambdaint1}. The general case follows by analytic continuation.
\end{proof}

\begin{lemma} Assume $\boldsymbol \lambda$ satisfy \eqref{Z2IIlambda}. For all $f_1,f_2,f_3\in \mathcal C^\infty(S)$,
\[\mathcal S^{(k)}_{\lambda_1,\lambda_2}({\widetilde J}_{\rho+l_1} f_1, \widetilde J_{\rho+l_2} f_2, {\widetilde J}_{-\rho-m_3} f_3) = 0\ .
\]
\end{lemma}
\begin{proof}  For $s\in \mathbb C$, using three times \eqref{KJK}
\[\widetilde{\mathcal K}^{\lambda_1-s, \lambda_2+s, \lambda_3} (\widetilde J_{\rho+l_1+s} f_1, \widetilde J_{\rho+l_2-s} f_1,\widetilde J_{-\rho-m_3} f_1) =\]\[ \frac{ \pi^{3\rho}}{\Gamma( -l_1+s)\Gamma( -l_2-s)\Gamma( \rho+m_3)}\ \widetilde {\mathcal K}^{-\lambda_1+s,-\lambda_2-s,- \lambda_3} (f_1,f_2,f_3)\ .
\]
For $s\neq 0$, multiply both sides by $\frac{1}{s^2}$ and let $s\rightarrow 0$. Then, by using Lemma \ref{KSZ2II}
\[\mathcal S^{(k)}_{\lambda_1,\lambda_2}({\widetilde J}_{\rho+l_1} f_1, \widetilde J_{\rho+l_2} f_2, {\widetilde J}_{-\rho-m_3} f_3) = \gamma' \
 \widetilde {\mathcal K}^{-\lambda_1,-\lambda_2,- \lambda_3} (f_1,f_2,f_3)
 \]
 where $\gamma'\neq 0$. Observe that $\boldsymbol \lambda' = (-\lambda_1,-\lambda_2,-\lambda_3)$ belongs to $Z$. In fact,
\[\lambda_1' = \rho+l_1, \lambda_2'=\rho+l_2, \lambda_3' = -\rho-m_3,
\]
and $\lambda_1'-\lambda_2'= l_1-l_2$, where $l_1-l_2\equiv m_3 \mod 2$ and $\vert l_1-l_2\vert \leq m_3$. Hence $\widetilde {\mathcal K}^{-\lambda_1,-\lambda_2,- \lambda_3}\equiv 0$ and the conclusion follows.
\end{proof}
Now recall (cf \eqref{rho+k} and \eqref{-rho-k}) that for $m\in \mathbb N$, $Im(\widetilde J_{\rho+m}) = \mathcal P_m$  is an invariant subspace for the  representation $\pi_{-\rho-m}$, and the restriction of $\pi_{\-\rho-m}$ to $\mathcal P_m$ yields a finite dimensional irreducible representation $\rho_m$ of $G$. Dually,  $Im(\widetilde J_{-\rho-m}) = \mathcal P_m^\perp$ and $\pi_{\rho+m}$ descends  to a representation $\rho_m'$ of $G$ on $\mathcal P_m' = \mathcal C^\infty(S)/\mathcal P_m^\perp$. Moreover, $Ker(\widetilde J_{\rho+m}) = \mathcal P_m^\perp$, and $\widetilde J_{\rho+m}$ passes to the quotient and gives raise to an intertwining operator for the representations $\rho_m'$ and $\rho_m$. Moreover, the duality between $\pi_{\lambda}$ and $\pi_{-\lambda}$ realizes $\rho'_m$ as the contragredient representation of $\rho_m$.

\begin{corollary} Let $f_1\in \mathcal P_{l_1}, f_2\in \mathcal P_{l_2}, f_3 \in \mathcal P_{m_3}^\perp$. Then 
\[  \mathcal S^{(k)}_{\lambda_1,\lambda_2} (f_1,f_2,f_3) = 0\ .
\]
\end{corollary}
\begin{proof} Use $Im(\widetilde J_{\rho+l}) = \mathcal P_l$ and $Im (\widetilde J_{-\rho-m}) = \mathcal P_m^\perp$.
\end{proof}

Next, first restrict the trilinear form $\mathcal S^{(k)}_{\lambda_1,\lambda_2}$ to $\mathcal P_{l_1}\times \mathcal P_{l_2}$ in the first two variables. Then pass to the quotient $mod \ \mathcal P_{m_3}^\perp$ on the third variable to obtain a trilinear form on $\mathcal P_{l_1}\times \mathcal P_{l_2} \times \mathcal P_{m_3}'$. This trilinear form is not trivial. In fact, observe that $\vert x-y\vert^{2k} \in \mathcal P_{l_1}\otimes \mathcal P_{l_2}$ as $l_1,l_2\leq k$, and $\mathcal S^{(k)}_{\lambda_1,\lambda_2}(\vert x-y\vert^{2k} \otimes 1)\neq 0$ by \eqref{Snot0}. This leads to the following statement, which realizes $\rho_{m_3}$ as a component of $\rho_{l_1}\otimes \rho_{l_2}$.

\begin{proposition} Let $l_1,l_2,m_3\in \mathbb N$ satisfy \eqref{Z2IIlambda}, and set 
\[\lambda_1=-\rho-l_1,\quad \lambda_2 =-\rho-l_2,\quad  l_1+l_2-m_3 = 2k\]
Then
\smallskip

$i)$ the trilinear form $\mathcal S^{(k)}_{\lambda_1,\lambda_2}$  yields a non trivial continuous trilinear form on $\mathcal P_{l_1}\times \mathcal P_{l_2} \times \mathcal P_{m_3}'$ which is 
invariant w.r.t. $(\rho_{l_1}, \rho_{l_2},\rho_{m_3}')$.
\smallskip

$ii)$ the bi-differential operator $D^{(k)}_{\lambda_1,\lambda_2}$ maps $\mathcal P_{l_1}\otimes \mathcal P_{l_2}$ into $\mathcal P_{m_3}$ and is covariant w.r.t. $(\rho_{l_1}\otimes \rho_{l_2}, \rho_{m_3})$.
\end{proposition}
 \section{Multiplicity 3 theorem for $\boldsymbol \alpha \in Z_3$}
Let $\boldsymbol \alpha$ be in $Z_3$. Then
\[\alpha_1 = -(n-1)-2k_1,\quad  \alpha_2 = -(n-1)-2k_2,\quad \alpha_3 = -(n-1)-2k_3
\]
where $k_1,k_2,k_3\in \mathbb N$. Correspondingly,
\[ \lambda_1 = -\rho-l_1,\quad \lambda_2 = -\rho-l_2,\quad \lambda_3 = -\rho-l_3\]
where $l_1,l_2,l_3\in \mathbb N$ satisfy
\[l_1+l_2+l_3\equiv 0 \mod 2,\qquad \vert l_1-l_2\vert\leq l_3\leq l_1+l_2\ .
\]
If $n-1$ is odd, then $\boldsymbol \alpha$ is  not a pole of type II. If $n-1$ is even (hence $\rho\in \mathbb N$), then $\alpha_1+\alpha_2+\alpha_3=-2(n-1)-2(\rho+k_1+k_2+k_3)\in -2(n-1)-2\mathbb N$ and $\boldsymbol \alpha$ is a pole of type II.
\subsection{The odd case}

Assume in this subsection that $n-1$ is odd. 
\begin{proposition}\label{T1T2T3indepev}
 The three distributions $\widetilde {\mathcal T}^{(1,k_1)}_{\alpha_2, \alpha_3}, \widetilde {\mathcal T}^{(2,k_2)}_{\alpha_1, \alpha_3},\widetilde {\mathcal T}^{(3,k_3)}_{\alpha_1, \alpha_2}$ are three linearly independent $\boldsymbol \lambda$-invariant distributions.
\end{proposition}

\begin{proof} By Proposition \ref{suppT3}
$Supp(\widetilde {\mathcal T}^{1,k_1}_{\alpha_2, \alpha_3})= \overline {\mathcal O_1}$ and a similar result by permuting the indices for the two other distributions. The statement follows.
\end{proof}

\begin{theorem}\label{mult3Z3odd}
 The space $Tri(\boldsymbol \lambda)$ is of dimension $3$. More precisely, 

\[Tri(\boldsymbol \lambda) = \mathbb C \widetilde {\mathcal T}^{(1,k_1)}_{\alpha_2, \alpha_3} \oplus \mathbb C \widetilde {\mathcal T}^{(2,k_2)}_{\alpha_1, \alpha_3} \oplus \mathbb C \widetilde {\mathcal T}^{(3,k_3)}_{\alpha_1, \alpha_2}\ .
\]
\end{theorem}
\begin{proof}
Because of Proposition \ref{T1T2T3indepev}, it suffices to prove that any $\boldsymbol \lambda$-invariant distribution $T$ is a linear combination of  $\widetilde {\mathcal T}^{(1,k_1)}_{\alpha_2, \alpha_3},  \widetilde {\mathcal T}^{(2,k_2)}_{\alpha_1, \alpha_3} , \widetilde {\mathcal T}^{(3,k_3)}_{\alpha_1, \alpha_2} $. The next proposition is a  first step in this direction.
\begin{proposition} \label{noextZ3odd}
The distribution $\mathcal K_{\boldsymbol \alpha, \mathcal O_0}$ cannot be extended to a $\boldsymbol \lambda$-invariant distribution on $S\times S\times S$.
\end{proposition}
Among the four $\Gamma$ factors in the normalization of $\boldsymbol \beta \longmapsto \mathcal K_{\boldsymbol \beta}$, three become singular at $\boldsymbol \alpha$, namely $\Gamma( \frac{\beta_1}{2}+\rho), \Gamma(\frac{\beta_2}{2}+\rho), \Gamma(\frac{\beta_3}{2}+\rho)$. For $s$ in $\mathbb C$ let
\[\boldsymbol \alpha(s) = (\alpha_1+2s, \alpha_2+2s, \alpha_3+2s)
,\qquad \boldsymbol \lambda (s) = (\lambda_1+2s, \lambda_2+2s, \lambda_3+2s)\ .
\]
 The distribution-valued  function defined for $s\in \mathbb C, s\neq 0$ and  $\vert s\vert$ small,  by 
\[\mathcal F(s) = (-1)^{k_1}\, k_1!\, (-1)^{k_2}\, k_2!\, (-1)^{k_3}\, k_3!\,\Gamma(-\rho-k_1-k_2-k_3+3s)\,\frac{1}{s^2}\, \widetilde {\mathcal K}_{\boldsymbol \alpha(s)}\ .
\]
can be extended holomorphically at $s=0$, because $\boldsymbol \beta\longmapsto \widetilde {\mathcal K}_{\boldsymbol \beta}$ vanishes at $\boldsymbol \alpha$ as well as its first derivatives. The Taylor expansion of $\mathcal F(s)$ at $0$ reads
\[\mathcal F(s) = F_0+s F_1 +O(s^2)\ , 
\]
where $F_0$ and $F_1$ are distributions on $S\times S\times S$.

\begin{lemma}\label{Z3iF2F3} {\ }
\smallskip

$i)$ $Supp(F_0) = \overline {\mathcal O_1\cup \mathcal O_2\cup \mathcal O_3}$
\smallskip

$ii)$ $F_1$ coincides on $\mathcal O_0$ with $\mathcal K_{\boldsymbol \alpha, \mathcal O_0}$
\end{lemma}

\begin{proof} 
Let $\varphi \in \mathcal C^\infty(S\times S\times S)$ and assume that its support is contained in $\mathcal O_0$. For $s\neq 0$, $\boldsymbol \alpha(s)$ is not a pole and 
\[\mathcal F(s)( \varphi)= \frac{(-1)^{k_1}\, k_1!}{\Gamma(-k_1+s)}\,\frac{(-1)^{k_2}\, k_2!}{\Gamma(-k_2+s)}\,
\frac{(-1)^{k_3}\, k_3!}{\Gamma(-k_3+s)}\,\frac{1}{s^2}\,
\mathcal K_{\boldsymbol \alpha(s), \mathcal O_0}( \varphi)
\]
As $s$ tends to $0$, the right hand side is equal to $s\, \mathcal K_{\boldsymbol \alpha, \mathcal O_0} ( \varphi) +O(s^2)$ so that $\big(F_0,\varphi)=0$ and
 $F_1(\varphi) = \mathcal K_{\boldsymbol \alpha, \mathcal O_0}( \varphi)$, from which $ii)$ follows. 
 
Next,
\[ \lim_{s\rightarrow 0} \frac{1}{s^2}\,\widetilde{\mathcal K}_{\boldsymbol \alpha(s)} = 2\big(\frac{\partial^2 \widetilde{\mathcal K}_{\boldsymbol \beta}}{\partial \beta_1 \partial \beta_2 }(\boldsymbol \alpha) +\frac{\partial^2 \widetilde{\mathcal K}_{\boldsymbol \beta}}{\partial \beta_2 \partial \beta_3 }(\boldsymbol \alpha)+\frac{\partial^2 \widetilde{\mathcal K}_{\boldsymbol \beta}}{\partial \beta_3 \partial \beta_1 }(\boldsymbol \alpha)\Big)
\]
On the other hand,  use \eqref{KT} to obtain,  for $s_1,s_2 \in \mathbb C$,
\[\widetilde{\mathcal K}_{ \alpha_1+_1, \alpha_2+s_2, \alpha_3} = \frac{ (-1)^{k_3} 2^{-2k_3}}{\Gamma(\rho+k_3)\Gamma(-k_1+s_1)\Gamma(-k_2+s_2)} \widetilde {\mathcal T}^{(3,k_3)}_{\alpha_1+s_1,\alpha_2+s_2}
\]
hence
\[\frac{\partial^2 \widetilde {\mathcal K}_{\boldsymbol \beta}}{\partial \beta_1 \partial \beta_2} (\boldsymbol \alpha)= \frac{(-1)^{k_3} 2^{-2k_3}(-1)^{k_1} k_1! (-1)^{k_2} k_2!}{\Gamma(\rho+k_3)}\, \widetilde {\mathcal T}^{(3,k_3)}_{\alpha_1,\alpha_2}\ ,
\]
and similar formul{\ae} by permuting the indices $1,2,3$. So, $F_0$ is a linear combination (with coefficients all $\neq 0$) of three distributions with support respectively $\overline {\mathcal O_1}, \overline {\mathcal O_2}$ and $\overline {\mathcal O_3}$, so that $Supp(F_0)= \overline {\mathcal O_1\cup \mathcal O_2\cup \mathcal O_3}$.
\end{proof}

Now the proof of Proposition \ref{noextZ3odd} is obtained by using once more Proposition A1, after restricting the situation to (say) $\mathcal O=\mathcal O_0\cup \mathcal O_1$.

To finish the proof of Theorem \ref{mult3Z3odd}, 
let $T$ be a $\boldsymbol \lambda$-invariant distribution. Then, Proposition \ref{noextZ3odd} shows that the restriction $T_{\mathcal O_0}$ has to be $0$, and so $Supp(T)\subset \overline {\mathcal O_1\cup \mathcal O_2\cup \mathcal O_3}$. Let $\mathcal O = \mathcal O_0\cup \mathcal O_1$. The restriction $T_{\mathcal O}$ of $T$ to $\mathcal O$ is supported in $\mathcal O_1$, and hence by Lemma 4.1 in \cite{c}, there exists a constant $c_1$ such that $T$ coincides on $\mathcal O$ with  $c_1 \widetilde {\mathcal T}^{1,k_1}_{\alpha_2,\alpha_3}$. In other words, $T-c_1 \widetilde {\mathcal T}^{1,k_1}_{\alpha_2,\alpha_3}$ is a $\boldsymbol \lambda$-invariant distribution supported in $\overline{\mathcal O_2\cup \mathcal O_3}$. By arguing in a similar way, there exists $c_2$ and $c_3$ such that $T-c_1 \widetilde {\mathcal T}^{1,k_1}_{\alpha_2,\alpha_3}-c_2 \widetilde {\mathcal T}^{2,k_2}_{\alpha_1,\alpha_3}-c_3 \widetilde {\mathcal T}^{3,k_3}_{\alpha_1,\alpha_2}$ is a $\boldsymbol \lambda$-invariant distribution supported on $\mathcal O_4$.  But as $\boldsymbol \lambda$ is \emph{not} a pole of type II, and by Lemma 4.2 in \cite{c}, there exists no such non trivial distribution, so that $T=c_1 \widetilde {\mathcal T}^{1,k_1}_{\alpha_2,\alpha_3}+c_2 \widetilde {\mathcal T}^{2,k_2}_{\alpha_1,\alpha_3}+c_3 \widetilde {\mathcal T}^{3,k_3}_{\alpha_1,\alpha_2}$.
\end{proof}
\subsection{The even case}

Assume in this subsection that $n-1$ is even, so that $\rho$ is an integer. Then, as
$\alpha_1+\alpha_2+\alpha_3 = -2(n-1)-2(\rho+k_1+k_2+k_3)$,  $\boldsymbol \alpha$ is a pole of type II, and we let $k= \rho+k_1+k_2+k_3$. 
\begin{proposition}\label{indZ3ev}
 The three distributions $\widetilde {\mathcal T}^{(1,k_1)}_{\alpha_2, \alpha_3}, \widetilde {\mathcal T}^{(2,k_2)}_{\alpha_1, \alpha_3},\widetilde {\mathcal T}^{(3,k_3)}_{\alpha_1, \alpha_2}$ are three linearly linearly independant $\boldsymbol \lambda$-invariant distributions.
\end{proposition}

\begin{proof} Although the statement is the same as in the previous case, the proof is quite different. In fact, recall that the three distributions are supported in $\mathcal O_4$. Use once again the evaluation of (some) $K$-coefficients of the three distributions to prove their linear independence. The following result was obtained during  the proof of Proposition \ref{indepT2T3}.

\begin{lemma} {\ }
\smallskip

$i)$ If $a_3>k_3$, then $\widetilde {\mathcal T}^{(3,k_3)}_{\alpha_1, \alpha_2}(p_{a_1,a_2,a_3})=0$
\smallskip

$ii)$ Let $a_1=k_1+\rho-1, a_2 = k_2+1, a_3=k_3$. Then
$\widetilde {\mathcal T}^{(3,k_3)}_{\alpha_1, \alpha_2}(p_{a_1,a_2,a_3})\neq 0$.
\end{lemma}
Up to permutation of the indices, a similar result holds true for the two other distributions. As $n\geq 4$ and $n-1$ even,  $\rho\geq 2$ so that $a_1>k_1$. Moreover, $a_2>k_2$. As a consequence of part $i)$ of the lemma, \[\widetilde {\mathcal T}^{(1,k_1)}_{\alpha_2, \alpha_3}(p_{a_1,a_2,a_3})=0, \quad \widetilde {\mathcal T}^{(2,k_2)}_{\alpha_1, \alpha_3}(p_{a_1,a_2,a_3}) = 0\ .\]
 The linear independance of the three distributions follows from this result and its variants under permutation of the indices $1,2,3$.
\end{proof}

\begin{theorem}\label{dimZ3ev} 
The space $Tri(\boldsymbol \lambda)$ is of dimension $3$. More precisely, 

\[Tri(\boldsymbol \lambda) = \mathbb C \widetilde {\mathcal T}^{(1,k_1)}_{\alpha_2, \alpha_3} \oplus \mathbb C \widetilde {\mathcal T}^{(2,k_2)}_{\alpha_1, \alpha_3} \oplus \mathbb C \widetilde {\mathcal T}^{(3,k_3)}_{\alpha_1, \alpha_2}\ .
\]
\end{theorem}

\begin{proof} Three facts need to be established before giving the proof, namely
\smallskip

$\bullet$ the distribution $\mathcal K_{\boldsymbol \alpha,\mathcal O_0}$ cannnot be extended to a $\boldsymbol \lambda$-invariant distribution on $S\times S\times S$.
\smallskip

$\bullet$ the distribution $\mathcal T_{\boldsymbol \alpha_1,\alpha_2, \mathcal O_4^c}^{(3,k_3)}$ cannot be extended to a $\boldsymbol \lambda$-invariant distribution on $S\times S\times S$.
\smallskip

$\bullet$ $\dim Tri(\boldsymbol \lambda) =3$.

\begin{proposition}\label{noextZ3ev}
 The distribution $\mathcal K_{\boldsymbol \alpha,\mathcal O_0}$ cannot be extended to a $\boldsymbol \lambda$-invariant distribution on $S\times S\times S$.
\end{proposition}

For $s\in \mathbb C$ define 
\[ \boldsymbol \alpha(s) = (\alpha_1+s,\alpha_2+s,\alpha_3+s), \quad \boldsymbol \lambda (s) = (\lambda_1+s, \lambda_2+s,\lambda_3+s)\ .
\]
Observe that for $s\neq 0$ and $\vert s\vert$ small, $\boldsymbol \alpha(s)$ is not a pole. Next, the four Gamma factor in the normalization of $\mathcal K_{\boldsymbol \alpha}$ are singular at $\boldsymbol \alpha$.

By the same argument used in the odd case, the distribution-valued  function defined for $s\in \mathbb C$, $s\neq 0$ by
\[\mathcal F(s) = \frac{8}{3}\,\frac{(-1)^{k_1}}{k_1!}\,\frac{(-1)^{k_2}}{\,k_2!}\,\frac{(-1)^{k_3}}{k_3!}\,\frac{(-1)^k}{k!}\,\frac{1}{s^2}\widetilde {\mathcal K}_{\boldsymbol \alpha(s)}
\]
can be continued to a holomorphic function in a neighborhood of $s=0$.  The Taylor expansion of $\mathcal F$ at $0$ reads 
\[\mathcal F(s) =  F_0+s F_1+s^2 F_2+ O(s^3)\ 
\]
where $F_0,F_1$ and $F_2$ are distributions on $S\times S\times S$.
\begin{proposition}\label{F0F1F2} {\ }
\smallskip

$i)$ $F_0$ is $\boldsymbol \lambda$-invariant and $Supp(F_0) = \mathcal O_4$.
\smallskip

$ii)$ $Supp(F_1) = \overline{\mathcal O_1\cup\mathcal O_2\cup \mathcal O_3}$
\smallskip

$iii)$ the restriction of $F_2$ to $\mathcal O_0$ coincides with $\mathcal K_{\boldsymbol \alpha, \mathcal O_0}$.
\end{proposition}
\begin{proof}
The proof of $i)$ is the same as in the odd case, except that now the three distributions $\widetilde {\mathcal T}^{(1,k_1)}_{\alpha_2, \alpha_3}, \widetilde {\mathcal T}^{(2,k_2)}_{\alpha_1, \alpha_3},\widetilde {\mathcal T}^{(3,k_3)}_{\alpha_1, \alpha_2}$ are supported in $\mathcal O_4$.

Next, let $\varphi\in \mathcal C^\infty(S\times S\times S)$ and assume that $Supp(\varphi)\subset \mathcal O_0$. Then for $s\neq 0$ and $\vert s\vert $ small, 
\[\big(\mathcal F(s),\varphi\big) = \frac{8}{3}\,\frac{(-1)^{k_1}}{k_1!}\,\frac{(-1)^{k_2}}{\,k_2!}\,\frac{(-1)^{k_3}}{k_3!}\,\frac{(-1)^{k}}{k!} \]
\[\frac{1}{\Gamma(-k_1+\frac{s}{2})\,\Gamma(-k_2+\frac{s}{2})\,\Gamma(-k_3+\frac{s}{2})\,\Gamma(-k+\frac{3s}{2})} \,\frac{1}{s^2}{\mathcal K}_{\boldsymbol \alpha(s)}(\varphi)\ .
\]
As $s\rightarrow 0$, the right hand side is equal to 
$s^2 \mathcal K_{\boldsymbol \alpha, \mathcal O_0}(\varphi)+O(s^3)
$, so that $F_0(\varphi)=0, F_1(\varphi)=0$ and $F_2(\varphi) = \mathcal K_{\boldsymbol \alpha, \mathcal O_0}(\varphi)$. Hence $iii)$ follows, and also $ Supp(F_1)\subset \overline {\mathcal O_1\cup \mathcal O_2\cup \mathcal O_3}$.

It remains to determine the supports of  $F_1$. Let $a_1,a_2,a_3\in \mathbb N$. Then, using (11) in \cite{c}, the $K$-coefficient  $\mathcal F(s)(p_{a_1,a_2,a_3})$ (up to a factor which does not vanish for $s=0$) is equal to
\[ 
\frac{1}{s^2}\,(-k+\frac{3}{2}s)_{a_1+a_2+a_3} (-k_1+\frac{s}{2})_{a_1}(-k_2+\frac{s}{2})_{a_2}(-k_3+\frac{s}{2})_{a_3}\\ \ \cdots\] 
\[ \frac{1}{\Gamma(-k_1-k_2+a_1+a_2+s)}\, \frac{1}{\Gamma(-k_2-k_3+a_2+a_3+s)}\, \frac{1}{\Gamma(-k_3-k_1+a_3+a_1+s)}\ .
\]
\begin{lemma}\label{nonvanF1}
 Let $a_1,a_2,a_3\in \mathbb N$ and assum that 
\[a_1\leq k_1,\quad a_2> k_1+k_2,\quad  a_3>k_1+ k_3,\quad  a_1+a_2+a_3>k\ .\]
Then $F_1(p_{a_1,a_2,a_3}) \neq 0$.
\end{lemma}

\begin{proof} Under the assumptions, the three $\Gamma$ factors are regular at $s=0$, the factor $(-k_1+s)_{a_1}$ does not vanish at $s=0$ and the three remaining factors $(-k+\frac{3}{2}s)_{a_1+a_1+a_3}, (-k_2+s)_{a_2}, (-k_3+s)_{a_3}$ have a simple zero at $s=0$. Hence the conclusion holds true.
\end{proof}

From the definition of $F_1$ and part $i)$ of Proposition \ref{F0F1F2} follows that  $F_{1\vert \mathcal O_4^c}$ is  $\boldsymbol \lambda$-invariant. Hence $Supp(F_1)$ is invariant under $G$. Now Lemma  \ref{nonvanF1} shows that $Supp(F_1)$ is not included in $\overline{\mathcal O_2\cup \mathcal O_3}$. In fact, given an arbitrary integer $L$, choosing $a_2,a_3$ large enough, all partial derivatives of the function $p_{a_1,a_2,a_3}$ of order $\leq L$ vanish on $\overline{\mathcal O_2\cup \mathcal O_3}$. Hence if $T$  is a distribution on 
$S\times S\times S$ supported on $\overline{\mathcal O_2\cup \mathcal O_3}$, then $T(p_{a_1,a_2,a_3)})=0$ for $a_2,a_3$ large enough (take for $L$ the order of the distribution $T$). By permuting the indices $1,2,3$, the only remaining possibility for $Supp(F_1)$ is $Supp(F_1) = \overline {\mathcal O_1\cup \mathcal O_2\cup \mathcal O_3}$.
\end{proof}

Let us now come to the proof of Proposition \ref{noextZ3ev}. Restrict the previous situation to $\mathcal O = {\mathcal O}_0\cup {\mathcal O}_1$. The distribution-valued function $\mathcal G(s)$ defined for $s\neq 0$ by
\[ \mathcal G(s) = \frac{1}{s} \,\mathcal F(s)_{\vert \mathcal O}\ .
\]
can be extended to a holomorphic function in a neighborhood of $0$ in $\mathbb C$.
Then, as $s\rightarrow 0,$
 \[\mathcal G(s) = F_{1,\mathcal O} + sF_{2, \mathcal O}+O(s^2),\]
and $Supp(F_{1,\mathcal O})= \mathcal O_1$. Now $F_{1,\mathcal O}$ is $\boldsymbol \lambda$-invariant. An application of Proposition A1 implies that $F_{2,\mathcal O_0}$ cannot be extended to a $\boldsymbol \lambda$-invariant distribution on $\mathcal O$. Hence, a fortiori the distribution $\mathcal K_{\boldsymbol \alpha, \mathcal O_0}$ cannot be extended to a $\boldsymbol \lambda$-invariant distribution on $S\times S\times S$.
 \end{proof}
\begin{proposition}\label{noextTZ3ev}
 The distribution 
${\mathcal T}^{(1,k_1)}_{\alpha_2,\alpha_3, \mathcal O_4^c} $
cannot be extended to $S\times S\times S$ as a $\boldsymbol \lambda$-invariant distribution. 
\end{proposition}

The proof is very similar to the proof of Proposition \ref{noextT3k3} so that we omit it.

\begin{proposition} {\ }
\[\dim Tri(\boldsymbol \lambda, diag) = 3\ .
\]
\end{proposition}

\begin{proof} As seen earlier, it suffices to prove that $\dim Sol(\lambda_1,\lambda_2;k)\leq 3$ (see Appendix 3). In the present case,
$\lambda_1 = -\rho-k_2-k_3$, and as $k=k_1+k_2+k_3+\rho$,  $\lambda_1\in \{-\rho,-\rho- 1,\dots, -\rho-(k-1)\}\cap \{ -1,-2,\dots, -(k-1)\}$ and similarly for $\lambda_2$. Moreover, $-\lambda_1-\lambda_2 = (n-1) +k_2+k_3 +2k_1\geq k$. The conclusion follows by Lemma A9 $iii)$.
\end{proof}

We now come to the proof of the main theorem in this section. Let $T$ be a $\boldsymbol \lambda$-invariant distribution. By Proposition \ref{noextZ3ev}, the restriction of $T$ to $\mathcal O_0$ has to be $0$. Hence $Supp(T)\subset \overline{\mathcal  O_1\cup \mathcal O_2\cup \mathcal O_3}$. Let $\mathcal O_{01}=\mathcal O_0\cup \mathcal O_1$. This is an open set of $S\times S\times S$ which contains $\mathcal O_1$ as a (relative) closed submanifold. The restriction $T_{\vert \mathcal O_{01}}$ of $T$ to $\mathcal O_{01}$ is supported on $\mathcal O_1$ and $\boldsymbol \lambda$-invariant. Hence $T_{\vert \mathcal O_{01}}$ must be a multiple of $\mathcal T^{1,k_1}_{\alpha_2,\alpha_3, \mathcal O_4^c}$. By Proposition \ref{noextTZ3ev}, this forces $T_{\vert \mathcal O_{01}}=0$. By permuting the indices $1,2,3$, $T$ has to be supported in $\mathcal O_4$, hence belongs to $Tri(\boldsymbol \lambda, diag)$. Hence $Tri(\boldsymbol \lambda) = Tri(\boldsymbol \lambda, diag)$. But  $\dim Tri(\boldsymbol \lambda, diag)=\dim {\it Sol}(\lambda_1,\lambda_2;k)\leq 3$, which finishes the proof of Theorem \ref{dimZ3ev}.

\section*{Appendix}

\subsection*{A1. Evaluation of the $K$-coefficients of $\widetilde{\mathcal K}_{\boldsymbol \alpha}$}

For the convenience of the reader, the formal{\ae} giving the evaluation of the $K$-coefficients of $\widetilde{\mathcal K}_{\boldsymbol \alpha}$ (sometimes called Bernstein-Reznikov integrals), which were obtained in \cite{c} Proposition 3.2 are recalled. For $\boldsymbol \alpha\in \mathbb C^3$, and $a_1,a_2,a_3\in \mathbb N$
\begin{equation}\label{Kpalpha}
\widetilde {\mathcal K}_{\boldsymbol \alpha} (p_{a_1,\,a_2,\,a_3}) =\quad \big(\frac{\pi}{2}\big)^{\frac{3}{2}(n-1)} 2^{\alpha_1+\alpha_2+\alpha_3}2^{ 2(a_1+a_2+a_3)}\dots
\end{equation}
\begin{equation*}
\frac{ \big(\frac{\alpha_1+\alpha_2+\alpha_3}{2}+2\rho\big)_{a_1+a_2+a_3}
\big(\frac{\alpha_1}{2} +\rho\big)_{a_1}\big(\frac{\alpha_2}{2} +\rho\big)_{a_2}\big(\frac{\alpha_3}{2} +\rho\big)_{a_3}} 
{\Gamma(\frac{\alpha_1+\alpha_2}{2}+2\rho+a_1+a_2)\,\Gamma(\frac{\alpha_2+\alpha_3}{2}+2\rho+a_2+a_3)\,\Gamma(\frac{\alpha_3+\alpha_1}{2}+2\rho+a_3+a_1)}
\end{equation*}
and its counterpart in terms of the spectral parameter $\boldsymbol \lambda$
\begin{equation}\label{Kplambda}
\begin{split}
& \widetilde {\mathcal K}^{\boldsymbol \lambda} (p_{a_1,a_2,a_3})  =(\frac{\sqrt{\pi}}{2})^{3(n-1)}2^{\lambda_1+\lambda_2+\lambda_3}2^{2(a_1+a_2+a_3)}\big(\frac{\lambda_1+\lambda_2 +\lambda_3+\rho}{2}\big)_{a_1+a_2+a_3}\\
&\frac{(\frac{-\lambda_1+\lambda_2 +\lambda_3+\rho}{2})_{a_1}(\frac{\lambda_1-\lambda_2 +\lambda_3+\rho}{2})_{a_2}(\frac{\lambda_1+\lambda_2 -\lambda_3+\rho}{2})_{a_3}}
{\Gamma(\lambda_1 +\rho+a_2+a_3)\Gamma(\lambda_2 +\rho+a_3+a_1)\Gamma(\lambda_3 +\rho+a_1+a_2)}\ .
\end{split}
\end{equation}
\subsection*{A2. Non extension results}

\begin{proposition*}\label{noextI}
 Let $\mathcal O$ be a $G$-invariant open subset of $S\times S\times S$, and assume that $\mathcal O_1$ is a (relatively) closed submanifold of $\mathcal O$. Let $\boldsymbol \alpha$ be a pole of type I$_1$, i.e. $\alpha_1 = -(n-1)-2k_1$ for some $k_1\in \mathbb N$.  Let $\boldsymbol \alpha(s)$ be a holomorphic curve defined for $s$ in a neighborhood of $0$, such that $\boldsymbol \alpha(0)=\boldsymbol \alpha$ and   transverse to the plane $\beta_1 =-(n-1)-2k_1$ at  $\boldsymbol \alpha$. Let $\boldsymbol\lambda(s)$  be the associated spectral parameter and let  $\boldsymbol \lambda=\boldsymbol \lambda(0)$.
 
 Let $\mathcal F(s)$ be a family of distributions on $\mathcal O$, depending holomorphically on $s$  in a neighborhood of $0$, and such that 

$i)$ for any $s$, $\mathcal F(s)$ is $\boldsymbol \lambda(s)$-invariant

$ii)$ $\mathcal F(s) = F_0+sF_1+O(s^2)$ as $s\longrightarrow 0$, where $F_0$ and $F_1$ are distributions on $\mathcal O$

$iii)$ $Supp(F_0) = \mathcal O_1$.
\smallskip

Then the restriction  $F_{1\vert \mathcal O'}$ of $F_1$ to $\mathcal O' = \mathcal O\setminus \mathcal O_1$ is $\boldsymbol \lambda$-invariant, but can not be extended to $\mathcal O$ as a $\boldsymbol \lambda$-invariant distribution.

\end{proposition*}
\begin{proposition*}\label{noextII}
 Let $\boldsymbol \alpha\in \mathbb C^3$ be a pole of type II, i.e. $\alpha_1+\alpha_2+\alpha_3=-2(n-1)-2k$ for some $k\in \mathbb N$. Let $\boldsymbol \alpha(s)$ be a holomorphic curve defined for $s$ in a neighborhood of $0$, such that $\boldsymbol \alpha(0)=\boldsymbol \alpha$ and transverse to the plane $\beta_1+\beta_2+\beta_3 =-2(n-1)-2k$ at $\boldsymbol \alpha$. Let $\boldsymbol \lambda(s)$ (resp.
$\boldsymbol \lambda$) be the associated spectral parameter.

Let $\mathcal F(s)$ be a family of distributions on $S\times S\times S$, depending holomorphically on $s$ near $0$, and such that

 $i)$ $\mathcal F(s)$ is $\boldsymbol \lambda(s)$-invariant
 
 $ii)$ $\mathcal F(s) = F_0+sF_1+O(s^2)$ as $s\longrightarrow 0$, where $F_0$ and $F_1$ are distributions on $S\times S\times S$
 
 $iii)$ $Supp(F_0)=\mathcal O_4$
 \smallskip
 
 Then the restriction $F_{1, \mathcal O'}$ of $F_1$ to $\mathcal O'=S\times S \times S\setminus \mathcal O_4$ is $\boldsymbol \lambda$-invariant, but cannot be extended
 to $S\times S\times S$ as a $\boldsymbol \lambda$-invariant distribution.
\end{proposition*}

\begin{proof}(sketch) The statement here is formulated in terms of holomorphic curves (in all applications only complex lines occur), to be closer to  the spirit of Oshima's results (see \cite{osh} section 7). 

Let $\boldsymbol \mu = (\mu_1,\mu_2,\mu_3) = \frac{d}{ds}_{\vert s=0}\lambda(s)$. The Taylor expansion at $s=0$ of the identity $\mathcal F(s)\circ \boldsymbol \pi_{\boldsymbol \lambda (s)} = \mathcal F(s)$ yields the functional relation, valid for any $g\in G$
\[F_1\circ \pi_{\boldsymbol \lambda}(g)-F_1 = C_g\, F_0\ ,
\]
where $C_g(x,y,z) = -(\mu_1 \ln \kappa(g^{-1},x)+\mu_2 \ln \kappa(g^{-1},y)+\mu_3 \ln \kappa(g^{-1},z)$. 

The subgroup $A$ (a split one-parameter Cartan subgroup of $G$) has a fixed point in $\mathcal O_1$ (resp. $\mathcal O_4$), say ${\mathbf o}$. The transversality  assumption guarantees that, at  ${\bf o}$ the character \[\chi_{\bf o} : A\ni a\longmapsto C_{a}({\bf o})\] is \emph{not} trivial. The rest of the proof is as in \cite{c}, Proposition 6.1 and 6.2.

\end{proof} 

\subsection*{A3. Discussion of the system ${\it S}(\lambda_1,\lambda_2;k)$}

Recall that, for $\boldsymbol \lambda \in \mathbb C^3$, $Tri(\boldsymbol \lambda, diag)$ is the space of $\boldsymbol \lambda$-invariant distributions on $S\times S\times S$ which are supported on the diagonal $\mathcal O_4 = \{(x,x,x), x\in S\}$. 

For $(\lambda_1,\lambda_2)\in \mathbb C^2$  and $k\in \mathbb N$, ${\mathcal B}{\mathcal D}_G(\lambda_1,\lambda_2;k)$ is the space of bi-differential operators $D:\mathcal C^\infty(S\times S) \longrightarrow \mathcal C^\infty(S)$ which are covariant w.r.t. $(\pi_{\lambda_1} \otimes \pi_{\lambda_2}, \pi_{\lambda_1+\lambda_2+\rho+2k})$ (see \cite {c} section 7).

Let $k\in \mathbb N$ and let $(\lambda_1,\lambda_2)\in \mathbb C^2$. Consider the system ${\it S}(\lambda_1,\lambda_2;k)$ of homogeneous linear equations in the unknowns $(c_{r,t}), 0\leq r,t,\  r+t\leq k$ given by
\begin{equation*}
\begin{aligned}
4(r+1)(r+1+\lambda)c_{r+1,\,t}&\\ +2(k-r-t)(k-r+t-1+\rho+\mu) c_{r,\,t} &\\ -(k-r-t+1)(k-r-t) c_{r,\,t-1} &= 0
\end{aligned}
\qquad E^{(1)}_{r,t}
\end{equation*}
\begin{equation*}
\begin{aligned}
4(t+1)(t+1+\mu)c_{r,\,t+1}&\\ +2(k-r-t)(k+r-t-1+\rho+\lambda) c_{r,\,t}&\\ -(k-r-t+1)(k-r-t) c_{r-1,\,t}&=0
\end{aligned}
\qquad E^{(2)}_{r,t}
\end{equation*}
for $r+t\leq k-1$.

The system was introduced  by Ovsienko and Redou (see \cite{or}).  Denote by ${\it Sol}(\lambda_1,\lambda_2;k)$ the space of solutions of the system ${\it S}(\lambda_1,\lambda_2;k)$.

\begin{proposition*}\label{BDS}
 Let $\lambda_1,\lambda_2, \lambda_3\in \mathbb C^3$ satisfy
$\lambda_1+\lambda_2+\lambda_3 = -\rho-2k$
for some $k\in \mathbb  N$. Then
\[ Tri(\lambda_1,\lambda_2, \lambda_3, diag) \sim  \mathcal {BD}_G(\lambda_1,\lambda_2; k) \sim {\it Sol}(\lambda_1,\lambda_2;k)\ .\]
\end{proposition*}
\begin{proof}The first isomorphism is proved in \cite{bc}, the second is due to Ovsienko and Redou (\cite{or}). As a consequence, the dimensions of the three spaces are equal.
\end{proof}
The system $S(\lambda_1,\lambda_2;k)$ was studied by Ovsienko and Redou who proved that generically $\dim {\it Sol}(\lambda_1,\lambda_2;k) = 1$, but for our purpose, a full discussion of the system (depending on the parameters $\lambda_1,\lambda_2$ and $k)$ is needed.
 
 The next result will help lower the number of situations to be considered.
\begin{proposition*}[Symmetry principle] {\ }
\smallskip

$i)$ $\dim {\it Sol}(\lambda_1,\lambda_2 ; k)=\dim {\it Sol}(\lambda_2,\lambda_1 ; k)$
\smallskip

$ii)$ let $\lambda_3 = -\rho-2k-\lambda_1-\lambda_2$. Then 
\[\dim {\it Sol}(\lambda_1,\lambda_2;k)=\dim{\it Sol}(\lambda_2,\lambda_3;k)=\dim {\it Sol}(\lambda_3,\lambda_1;k)\ .
\]
\end{proposition*}

\begin{proof} {\ }
\smallskip

For $i)$ , observe (cf \cite{or}) that if $(c_{r,t})$ is a solution of the system  ${\it S}(\lambda_1,\lambda_2;k)$, then $(c_{t,r})$ is a solution of ${\it S}(\lambda_2,\lambda_1;k)$.
\smallskip

For $ii)$, observe  that $\dim Tri(\lambda_1,\lambda_2,\lambda_3; diag)$ is invariant by any permutation of $\{\lambda_1,\lambda_2,\lambda_3\}$. Hence  Proposition \ref{BDS} implies  the statement.
\end{proof}

The strategy to estimate $\dim {\it Sol} (\lambda_1,\lambda_2 ; k)$ is twofold. An algebraic study of the system will produce upper bounds for the dimension. In counterpart, the analytic study  done in sections 10, 13, 14,15 and 16 has produced linearly independent distributions inside $Tri( \boldsymbol \lambda, diag)$, giving a lower bound for the dimension. A careful inspection shows that the lower and upper estimates coincide. Of course, it could be possible to write explicitly potential solutions produced by the algebraic study and verify that they are indeed solutions of the system (as claimed in \cite{or} for the generic case).

\begin{lemma*}\label{Sgen} Let $\lambda_2\notin \{ -1,-2,\dots, -k\}\cup \{ -\rho,-\rho-1,\dots, -\rho-(k-1)\}$. Then $dim \,{\it Sol}(\lambda_1,\lambda_2;k)\leq1$. 
\end{lemma*}
\begin{proof}  Equation $E^{(1)}_{r,0}$ reads
\[4(r+1)(r+1+\lambda_1)c_{r+1,0} +2(k-r)(k+\rho+\lambda_2-1 -r) c_{r,0} = 0\ .
\]

As $\lambda_2 \notin \{ -\rho,-\rho-1,\dots, -\rho-(k-1)\}$,  $k+\rho+\lambda_2 -1 -r\neq 0$ for $0\leq r\leq k-1$. Hence $c_{r,0}$ can be computed from $c_{r+1,0}$. Choose $c_{k,0}$ as principal unknown. Then, for $0\leq r\leq k-1$,  $c_{r,0}$ can be computed from $c_{k,0}$.

Equation $E^{(2)}_{r,t}$  reads
\[4(t+1)(t+1+\lambda_2)c_{r,t+1}+2(k-r)(k+r-t-1+\rho+\lambda_1)c_{r,t}\]\[-(k-r-t+1) (k-r-t))c_{r-1,t} = 0\ .
\]
As $\lambda_2 \notin \{ -1,-2,\dots, -k\}$, the coefficient of the unknown $c_{r,t+1}$ is never $0$, and hence the coefficient $c_{r,t+1}$ can be computed from $c_{r,t}$ and $c_{r-1,t}$. Hence all unknowns can be computed from $c_{k,0}$, so that $\dim {\it Sol}(\lambda_1,\lambda_2;k)\leq 1$. 
\end{proof}

We may now assume that $\lambda_1,\lambda_2 \in \{ -\rho, -\rho-1,\dots,-\rho-(k-1)\} \cup \{ -1,-2,\dots, -k\}$. Notice that if $n-1$ is odd (hence $\rho\notin \mathbb N$), the two sets $\{ -\rho, -\rho-1,\dots,-\rho-(k-1)\}$ and $\{ -1,-2,\dots, -k\}$ are disjoint, making in this case the study somewhat easier. When $n-1$ is even, then $\rho$ is an integer and 
\[\{ -\rho, -\rho-1,\dots,-\rho-(k-1)\} \cap \{ -1,-2,\dots, -k\} \]\[\begin{matrix}\emptyset\ \quad &\text{ if } k<\rho\\ \{-\rho,\dots,  -k\}\quad &\text{ if } {k\geq \rho}\end{matrix}
\]

The general strategy to study the system was already used in the proof of Lemma A\ref{Sgen}. First study equations $E^{(1)}_{r,0}$ which involve only the unknowns $c_{r,0}$, then pass to the determination of the unknowns $c_{r,1}$, and so on. 

\begin{lemma*}\label{mix}
 Let $\lambda_1 =-k_1$ for some $k_1\in \mathbb \{ 1,2,\dots, k-1\}$ and $\lambda_2 = -\rho-l_2$ for some $l_2\in \{ 0,1,\dots, k-1\}$, but $\lambda_2 \notin \{ -1,-2,\dots, -k\}$. 

$i)$ if $k_1+l_2< k$, then $ \dim {\it Sol}(\lambda_1,\lambda_2;k) \leq 1$

$ii)$ if $k_1+l_2\geq k$, then $\dim{\it Sol}(\lambda_1,\lambda_2;k)\leq 2$.
 
\end{lemma*}

\begin{proof}

Equation $E^{(1)}_{k_1-1,0}$  reads
\begin{equation}\label{Ek1}
0\, c_{k_1,0} +2(k-k_1+1)(k-k_1-l_2) c_{k_1-1,0} = 0
\end{equation}

Equation $E^{(1)}_{k-l_2-1,0}$ reads
\begin{equation}\label{Ek2}
 4(r+1)(k-l_2-k_1) c_{k-l_2,0} + 0\, c_{k-l_2-1,0} = 0\ .
\end{equation}

Assume that $k_1+l_2<k$. \eqref{Ek1} implies $c_{k_1-1,0} = 0$ and \eqref{Ek2} implies $c_{k-l_2} = 0$. Repeated use of $E^{(1)}_{r,0}$ for $r\neq k_1-1, k-l_2-1$ leads to the following results 
\smallskip

$\bullet$ $c_{r,0}= 0$ for $0\leq r\leq k_1-1$ and for $k-l_2\leq r\leq k$
\smallskip

$\bullet$ $c_{k_1,0}$ can be chosen as principal unknown and the remaining unknowns $c_{r,0}, k_1<r<k-l_2$ can be expressed in terms of $c_{k_1,0}$.
Next use the same method as in Lemma \ref{Sgen} to prove that all $c_{r,t}$ can be expressed in terms of $c_{k_1,0}$. This gives $\dim \big(\mathcal Sol(\lambda_1,\lambda_2,k\big)\leq 1$ and $i)$ follows.
\medskip

Assume now that $k_1+l_2\geq k$. In this case, $k-l_2-1<k_1$, and choose $c_{k-l_2-1,0}$ and $c_{k_1,0}$ as principal unknowns. For $k-l_2\leq r\leq k_1-1$ (this interval is empty in case $k_1+l_2=k$) $c_{r,0} = 0$. For $0\leq r\leq k-l_2-1$, $c_{r,0}$ can be expressed in terms of $c_{k-l_2-1}$, whereas for $k_1<k\leq r$, $c_{r,0}$ can be expressed in terms of $c_{k_1,0}$. Hence all $c_{r,t}$ can be expressed in terms of the two chosen principal unknowns, and so $\dim {\it Sol}(\lambda_1,\lambda_2;k)\leq 2$.
\end{proof}

\begin{lemma*}\label{Sk1k2}
 Let $\lambda_1 = -k_1, \lambda_2=-k_2$ where $1\leq k_1,k_2\leq k$. 
\smallskip

$i)$ assume that $k_1+k_2\leq k$. Then  $\dim {\it Sol}(\lambda_1,\lambda_2;k) \leq 1$.
\smallskip

$ii)$ assume that  $k_1+k_2>k$ and assume that $\lambda_2 \notin \{-\rho,-\rho-1,\dots, -\rho-(k-1)\}$. Then $\dim {\it Sol}(\lambda_1,\lambda_2;k)\leq 2$.

\end{lemma*}
\begin{proof} {\ }

In case $i)$, the assumptions imply that $\lambda_3 = -\rho-2k+k_1+k_2$. Rewrite this as $\lambda_3 = -\rho-k-(k-k_1-k_2)$, so that $\lambda\in -\rho-k-\mathbb N$ and hence $\lambda_3 \notin \{ -\rho,\dots, -\rho-(k-1)\}\cup \{ -1,-2,\dots, -k\}$. Use the symmetry principle and Lemma \ref{Sgen} to conclude.

In case of $ii)$, assume first that $\lambda_1 \notin \{-\rho,-\rho-1,\dots, -\rho-(k-1)\}$.  The same considerations as in the beginning of the proof of Lemma  \ref{mix} show that $c_{r,0}=0$ for $r\leq k_1-1$. Choose  $c_{k_1,0}$  as first principal unknown, then the  unknowns $c_{r,0}$ for $r\geq k_1$ can be expressed in terms of $c_{k_1,0}$. Now as long as $t\leq k_2-1$, it is possible to use $E^{(2}_{r,t-1}$ to compute $c_{r,t}$ in terms of $c_{r,t-1}$ and $c_{r-1,t-1}$. It follows that, for $t\leq k_2-1$,  $c_{r,t}=0$ for $r\leq k_1-1$  and $c{_r,t}$ can be expressed in terms of $c_{k_1,0}$ for $r\geq k_1$. By symmetry, choosing $c_{k_2,0}$ as second principal unknown, it is possible for $r\leq k_1$ to express all $c_{r,t}$ in terms of $c_{0,k_2}$. But $r+t\leq k$ and $k>k_1+k_2$ implies that either $r\leq k_1-1$ or $t\leq k_2-1$, hence all unknowns can be expressed in terms of the two principal unknowns $c_{k_1,0}$ and $c_{0,k_2}$ and the statement follows.

It remains to consider the case where $\lambda_1 \in \{-\rho,-\rho-1,\dots, -\rho-(k-1)\}$. This assumption forces $n-1$ to be even, $\rho\leq k$ and hence $\lambda_1 =-k_1$, where $\rho\leq k_1\leq k$. Moreover, $0\leq k_2\leq \rho-1$. Then $\lambda_3 = -\rho-(2k-k_1-k_2)=-\rho-l_3$ with 
 with $0\leq  l_3 = 2k-k_1-k_2< k$ so that $\lambda_3 \in \{ -\rho, -\rho-1,\dots, -(\rho-k-1)\}$, but writing the same equality as 
$\lambda_3 =-k-((k-k_1)+(\rho-k_2))$ shows that  $\lambda_3 \notin \{-1,-2,\dots, -k\}$. Moreover, $l_3+k_1 = 2k-k_1-k_2+k_1 = 2k-k_2\geq k$. So $(\lambda_1,\lambda_3)$ satisfy the assumptions of  Lemma \ref{mix} $ii)$. By the symmetry principle,  the conclusion follows also in this case.
\end{proof}

\begin{lemma*} Let $\lambda_1 = -\rho-l_1, \lambda_2 = -\rho-l_2$ where
$0\leq l_1,l_2\leq k-1$, and $\lambda_1,\lambda_2 \notin \{ -1,-2,\dots, -k\}$. Then $\dim {\it Sol}(\lambda_1,\lambda_2;k)\leq1$.
\end{lemma*}

\begin{proof} Choose $c_{0, 0}$ as  principal unknown. Then, by using $E^{(1)}_{r-1,0}$  it is possible for $r\leq k-l_2-1$ to express $c_{r,0}$ in terms of $c_{0,0}$. Then necessarily $c_{k-l_2 ,0}=0$ and for $r\geq k-l_2+1$ $c_{r,0}=0$. As $\lambda_2\notin \{-1,-2,\dots, -k\}$, it is possible, using $E^{(2)}_{r,t}$ to express $c_{r,t}$ in terms of $c_{0,0}$. 

\end{proof}

\begin{lemma*} Let $n-1$ be even (so that $\rho\in \mathbb N$), and let $\lambda_1=-k_1,\lambda_2 =-k_2$, where $\rho\leq k_1,k_2\leq k$. 
\smallskip

$i)$ if $k_1+k_2\leq k$, then $\dim {\it Sol}(\lambda_1,\lambda_2;k)=1$.
\smallskip

$ii)$ if $k<k_1+k_2< k+\rho$, then $\dim {\it Sol}(\lambda_1,\lambda_2;k)=2$.
\smallskip

$iii)$ if $k_1+k_2\geq k+\rho$, then $\dim {\it Sol}(\lambda_1,\lambda_2;k)=3$.
\end{lemma*}

\begin{proof} If $k_1+k_2\leq k$, then $\lambda=-\rho-2k+k_1+k_2 = -\rho-k-(k-k_1+k_2)$, so that $\lambda_3 \notin \{ -\rho,-\rho-1,\dots, -\rho-(k-1)\}\cup \{ -1,-2,\dots, -k\}$. Use Lemma A1 and the symmetry principle to get $i)$.

Assume now that $k<k_1+k_2<k+\rho$. Then 
\[\lambda_3 = -\rho-2k+k_1+k_2= -\rho-k+(k_1+k_2-k)= -k-(k+\rho-k_1-k_2) \ . \]

Hence $\lambda_3 \in \{-\rho,-\rho-1,\dots, -\rho-(k-1)\}$ but $\lambda_3 \notin \{-1,-2,\dots, -k\}$. Use Lemma $ii)$ to conclude.

Assume now that $k_1+k_2\geq k+\rho$. 

Consider first the equation $E^{(1)}_{r,0}$, which reads
\[4(r+1)(r+1-k_1)c_{r+1,t}+2(k-r)(k-r-1+\rho-k_2) c_{r,0}
\]
The first coefficient vanishes for $r=k_1-1$ and the second vanishes for $r= k+\rho-1-k_2$. Notice that $0<k-k_2 +\rho-1\leq k_1-1 \leq k-1$.
Choose $c_{0,0}$  as first principal unknown. Then use $E^{(1)}_{r-1,0}$ to compute  $c_{r,0}$ in terms of $c_{0,0}$ up to $r= k+\rho-1-k_2$. The unknowns $c_{r,0}$ have to be $0$ for $k+\rho-k_2\leq r\leq k_1-1$ (if any). Now choose $c_{k,0}$ (second principal unknown). All $c_{r,0}$ for $r>k_1$ are determined by the equations $E^{(1)}_{r,0}$ and can be expressed in terms of $c_{k,0}$. 

Next consider the equations $E^{(2)}_{0,t}$. For $t\leq k+\rho-1-k_1$, the unknowns $c_{0,t}$ can be expressed in terms of $c_{0,0}$, for $k+\rho-k_1\leq t\leq k_2$ (if any) $c_{0,t}$ has to be $0$. For $t> k_2$, $c_{t,0}$ can be expressed in terms of $c_{0,r}$. 

Let now $1\leq t \leq k_2-1$, let $r+t\leq k$, and consider the equation $E^{(2)}_{r,t-1}$. Then the coefficient of the unknown $c_{r,t}$ is equal to $2(k-r-t+1)(t-k_2)\neq 0$. Hence, if $c_{r,t-1}$ and $c_{r-1,t-1}$ are already computed, $c_{r,t}$ can be computed. As all $c_{r,0}, 0\leq r\leq k$ have been expressed in terms of $c_{r,0}$ and $c_{0,0}$, $c_{r,t}$ can be also expressed in terms of $c_{k,0}$ and $c_{0,0}$. The statement follows. 
\end{proof}

Lemmas A1 up to A5 give an upper estimate of the dimension of the space of solutions. It remains to check that in each case, this estimate matches the number of independant $\boldsymbol \lambda$-invariant distributions supported on the diagonal which were constructed in the previous sections. Eventually, the dimension of ${\it Sol}(\lambda_1,\lambda_2; k)$ is given by the following table.
\vfill\eject
\centerline{\bf Table for $\dim {\it Sol}(\lambda_1,\lambda_2;k)$}

\bigskip

\begin{tabular}{| c | c | c | c |}
\hline
&$\begin{matrix}\lambda_1=-k_1\\ \lambda_1 \notin E_k^\rho\end{matrix}$ 
&$\begin{matrix}\lambda_1 =-\rho-l_1\\\lambda_1\notin E_k\end{matrix}$
&$\begin{matrix}\lambda_1=-k_1\\ \lambda_1 \in E_k\cap E_k^\rho\end{matrix}$ \\
\hline
$\begin{matrix} \lambda_2 = -k_2\\ \lambda_2 \notin E_k^\rho\end{matrix}$
&$\begin{matrix} k_1+k_2\leq k\hskip 1.6cm \mathbf{1}\\  k_1+k_2>k \hskip 1.6cm \mathbf {2}\end{matrix}$
&$\begin{matrix} l_1+k_2 <k \hskip 0.5cm \mathbf{1}\\l_1+k_2\geq k \hskip 0.5cm \mathbf {2}\end{matrix}$
&$\begin{matrix}k_1+k_2\leq k\hskip 1.75cm \mathbf{1}\\k<k_1+k_2<k+\rho \quad \mathbf{2}\\k_1+k_2\geq k+\rho \hskip 1.15cm \mathbf{3}\end{matrix}$
\\
\hline
$\begin{matrix}\lambda_2 =-\rho-l_2\\\lambda_2\notin E_k^\rho\end{matrix}$
&$\begin{matrix} k_1+l_2<k\hskip 1.65cm\mathbf{1}\\k_1+l_2\geq k\hskip 1.65cm \mathbf{2}\end{matrix}$
&$\mathbf{1}$
&$\begin{matrix}k_1+l_2<k\hskip 1.8cm \mathbf{1}\\ k_1+l_2\geq k\hskip 1.8cm \mathbf{2}\end{matrix}$
\\
\hline
$\begin{matrix}\lambda_2=-k_2\\ \lambda_2\in E_k\cap E_k^\rho\end{matrix}$
&$\begin{matrix}k_1+k_2\leq k\hskip 1.75cm \mathbf{1}\\k<k_1+k_2<k+\rho \quad \mathbf{2}\\k_1+k_2\geq k+\rho \hskip 1.15cm \mathbf{3}\end{matrix}$
&$\begin{matrix}k_2+l_1<k\hskip 0.5cm \mathbf{1}\\ k_2+l_1\geq k\hskip 0.5cm \mathbf{2}\end{matrix}$
&$\begin{matrix}k_1+k_2\leq k\hskip 1.75cm \mathbf{1}\\k<k_1+k_2<k+\rho \quad \mathbf{2}\\k_1+k_2\geq k+\rho \hskip 1.15cm \mathbf{3}\end{matrix}$

\\
\hline
\end{tabular}
\bigskip

where
\[E_k=\{ -1,-2,\dots,-k\},\qquad E_k^\rho=\{ -\rho, -\rho-1,\dots, -\rho-(k-1)\}\ .\]

\medskip
\footnotesize{\noindent Address\\ Jean-Louis Clerc\\Institut Elie Cartan, Universit\'e de Lorraine, 54506 Vand\oe uvre-l\`es-Nancy, France\\}
\medskip
\noindent \texttt{{jean-louis.clerc@univ-lorraine.fr
}}

\end{document}